\documentclass[11pt]{amsart}
\usepackage[english]{babel}
\usepackage[utf8]{inputenc}

\usepackage[inner=1.1in, outer=1.1in, bottom =2.9cm, top =2.9cm]{geometry}

\makeatletter
\newcommand*\Cdot{\mathpalette\Cdot@{.5}}
\newcommand*\Cdot@[2]{\mathbin{\vcenter{\hbox{\scalebox{#2}{$\m@th#1\bullet$}}}}}
\makeatother

\usepackage{amsmath, amssymb, amsfonts, amsthm}
\usepackage{graphicx, overpic}
\usepackage{mathrsfs}
\usepackage{mathtools}
\usepackage[usenames,dvipsnames,svgnames,table]{xcolor}
\usepackage{enumerate}
\usepackage{enumitem}
\usepackage{shuffle} 
\usepackage{lipsum}
\usepackage{color}
\usepackage{phoenician}
\usepackage{xkeyval}
\usepackage{mathrsfs}
\usepackage{tikz}
\usepackage{pst-node} 
\usepackage{tikz-cd} 
\usepackage[OT2,T1]{fontenc} 
\usepackage{thm-restate} 
\usepackage{float}
\usepackage{eucal}
\usepackage{microtype}
\usepackage[strict]{changepage} 
\usepackage{esvect}
\usepackage{etoolbox}
\usepackage{wrapfig}
\usepackage{caption}
\usepackage{cite}
\usepackage{mathrsfs}
\DeclareMathAlphabet{\mathpzc}{OT1}{pzc}{m}{it}
\usepackage{bbding}
\usepackage{array}
\usepackage{makecell}
\usepackage{stmaryrd}
\usepackage{lmodern}
\usepackage{amsbsy}
\usepackage{esvect}
\usepackage{bbding}
\usepackage{thm-restate}
\usepackage[toc,page, titletoc,title]{appendix}
\graphicspath{ {figures/} }
\usepackage{url}
\usepackage[pagebackref]{hyperref} 

\makeatletter
\providecommand*{\twoheadrightarrowfill@}{%
  \arrowfill@\relbar\relbar\twoheadrightarrow
}
\providecommand*{\twoheadleftarrowfill@}{%
  \arrowfill@\twoheadleftarrow\relbar\relbar
}
\providecommand*{\xtwoheadrightarrow}[2][]{%
  \ext@arrow 0579\twoheadrightarrowfill@{#1}{#2}%
}
\providecommand*{\xtwoheadleftarrow}[2][]{%
  \ext@arrow 5097\twoheadleftarrowfill@{#1}{#2}%
}
\makeatother

\makeatletter
\newcommand*{\relrelbarsep}{.386ex}
\newcommand*{\relrelbar}{%
  \mathrel{%
    \mathpalette\@relrelbar\relrelbarsep
  }%
}
\newcommand*{\@relrelbar}[2]{%
  \raise#2\hbox to 0pt{$\m@th#1\relbar$\hss}%
  \lower#2\hbox{$\m@th#1\relbar$}%
}
\providecommand*{\rightrightarrowsfill@}{%
  \arrowfill@\relrelbar\relrelbar\rightrightarrows
}
\providecommand*{\leftleftarrowsfill@}{%
  \arrowfill@\leftleftarrows\relrelbar\relrelbar
}
\providecommand*{\xrightrightarrows}[2][]{%
  \ext@arrow 0359\rightrightarrowsfill@{#1}{#2}%
}
\providecommand*{\xleftleftarrows}[2][]{%
  \ext@arrow 3095\leftleftarrowsfill@{#1}{#2}%
}
\makeatother

\makeatletter
\newcommand{\colim@}[2]{%
  \vtop{\m@th\ialign{##\cr
    \hfil$#1\operator@font colim$\hfil\cr
    \noalign{\nointerlineskip\kern1.5\ex@}#2\cr
    \noalign{\nointerlineskip\kern-\ex@}\cr}}%
}
\newcommand{\colim}{%
  \mathop{\mathpalette\colim@{\rightarrowfill@\scriptscriptstyle}}\nmlimits@
}
\renewcommand{\varprojlim}{%
  \mathop{\mathpalette\varlim@{\leftarrowfill@\scriptscriptstyle}}\nmlimits@
}
\renewcommand{\varinjlim}{%
  \mathop{\mathpalette\varlim@{\rightarrowfill@\scriptscriptstyle}}\nmlimits@
}
\makeatother

\DeclareSymbolFont{cyrletters}{OT2}{wncyr}{m}{n}
\DeclareMathSymbol{\Sh}{\mathalpha}{cyrletters}{"58}

\makeatletter
\newcommand*\bigcdot{\mathpalette\bigcdot@{.5}}
\newcommand*\bigcdot@[2]{\mathbin{\vcenter{\hbox{\scalebox{#2}{$\m@th#1\bullet$}}}}}
\makeatother

\usetikzlibrary{tikzmark,decorations.pathreplacing}
\usetikzlibrary{shapes,shadows,arrows}
\usetikzlibrary{decorations.markings}
\usetikzlibrary{positioning,calc}
\tikzset{near start abs/.style={xshift=1cm}}

\DeclareSymbolFont{symbolsC}{U}{txsyc}{m}{n}
\DeclareMathSymbol{\Searrow}{\mathrel}{symbolsC}{117}

\hypersetup{
	colorlinks=true,
	linkcolor=black,
	filecolor=black,      
	urlcolor=blue,
	citecolor= 	red,
}
\DeclareSymbolFont{extraup}{U}{zavm}{m}{n}
\DeclareMathSymbol{\varheart}{\mathalpha}{extraup}{86}
\DeclareMathSymbol{\vardiamond}{\mathalpha}{extraup}{87}

\newcommand{\bigslant}[2]{{\raisebox{.2em}{$#1$}\left/\raisebox{-.2em}{$#2$}\right.}}

\theoremstyle{definition}
\newtheorem{thm}{Theorem}[section]
\newtheorem{cor}{Corollary}[thm]
\newtheorem{lem}[thm]{Lemma}
\newtheorem{prop}[thm]{Proposition}

\theoremstyle{definition}

\newtheorem{remark}{Remark}[section]

\newcommand{\gs}{\sigma}

\newcommand{\tH}{\mathtt{H}}
\newcommand{\tQ}{\mathtt{Q}}

\newcommand{\tM}{\mathtt{M}}
\newcommand{\tE}{\mathtt{E}}

\newcommand{\tY}{\mathtt{Y}}

\newcommand{\tC}{\mathtt{C}}

\newcommand{\tP}{\mathtt{P}}

\newcommand{\tm}{\mathtt{m}}

\newcommand{\tc}{\mathtt{c}}

\newcommand{\tp}{\mathtt{p}}

\newcommand{\bR}{\mathbb{R}}
\newcommand{\bN}{\mathbb{N}}
\newcommand{\bC}{\mathbb{C}}

\newcommand{\bB}{\mathbb{B}}
\newcommand{\bZ}{\mathbb{Z}}

\newcommand{\cC}{\CMcal{C}}

\newcommand{\cP}{\CMcal{P}}

\newcommand{\cF}{\CMcal{F}}

\newcommand{\cG}{\CMcal{G}}

\newcommand{\cH}{\CMcal{H}}

\newcommand{\cS}{\CMcal{S}}

\newcommand{\cT}{\CMcal{T}}
\newcommand{\cK}{\CMcal{K}}
\newcommand{\cX}{\CMcal{X}}

\newcommand{\cU}{\CMcal{U}}

\newcommand{\Hom}{\operatorname{Hom}}

\newcommand{\bE}{\textbf{E}}

\newcommand{\bO}{\textbf{O}}

\newcommand{\sS}{\textsf{S}}

\newcommand{\Lie}{ \textbf{Lie} }
\newcommand{\Com}{\textbf{Com}}
\newcommand{\Ass}{\textbf{Ass}}

\newcommand{\bL}{\textbf{L}}
\newcommand{\Gam}{\boldsymbol{\Gamma}}
\newcommand{\Sig}{\boldsymbol{\Sigma}}  

\newcommand{\shuff}{\hat{\boldsymbol{\Sigma}}^\ast} 
\newcommand{\adshuff}{\check{\boldsymbol{\Sigma}}^\ast}

\newcommand{\res}{\parallel}

\newcommand{\bT}{\mathbb{T}}

\makeatletter
\@namedef{subjclassname@2020}{\textup{2020} Mathematics Subject Classification}
\makeatother

\newcommand{\la}{\langle}
\newcommand{\ra}{\rangle}

\newcommand{\wt}{\widetilde}

\definecolor{Red}{rgb}{0.8,0,0.2}

\newcommand{\GG}[1]{}

\makeatletter
\def\@footnotecolor{red}
\define@key{Hyp}{footnotecolor}{%
 \HyColor@HyperrefColor{#1}\@footnotecolor%
}
\def\@footnotemark{%
    \leavevmode
    \ifhmode\edef\@x@sf{\the\spacefactor}\nobreak\fi
    \stepcounter{Hfootnote}%
    \global\let\Hy@saved@currentHref\@currentHref
    \hyper@makecurrent{Hfootnote}%
    \global\let\Hy@footnote@currentHref\@currentHref
    \global\let\@currentHref\Hy@saved@currentHref
    \hyper@linkstart{footnote}{\Hy@footnote@currentHref}%
    \@makefnmark
    \hyper@linkend
    \ifhmode\spacefactor\@x@sf\fi
    \relax
  }%
\makeatother

\hypersetup{footnotecolor=blue}

\makeatletter
\patchcmd{\@startsection}
  {\@afterindenttrue}
  {\@afterindentfalse}
  {}{}
\makeatother

\title[Hopf Monoids, Permutohedral Cones, and Generalized Retarded Functions]{Hopf Monoids, Permutohedral Cones, and\\ Generalized Retarded Functions}
\author{William Norledge}
\address[William Norledge]{Pennsylvania State University}
\email{wxn39@psu.edu}
\author{Adrian Ocneanu}
\address[Adrian Ocneanu]{Pennsylvania State University}
\email{axo2@psu.edu}
\thanks{The second author was partly supported by NSF grants DMS-9970677, DMS-0200809, DMS-0701589. Both authors were partly supported by Templeton Religion Trust grant TRT-0159.}

\begin{document}

\renewcommand{\chapterautorefname}{Chapter}
\renewcommand{\sectionautorefname}{Section}
\renewcommand{\subsectionautorefname}{Section}

\renewcommand{\chapterautorefname}{Chapter}
\renewcommand{\sectionautorefname}{Section}
\renewcommand{\subsectionautorefname}{Section}

\begin{abstract}
The commutative Hopf monoid of set compositions is a fundamental Hopf monoid internal to vector species, having undecorated bosonic Fock space the combinatorial Hopf algebra of quasisymmetric functions. We construct a geometric realization of this Hopf monoid over the adjoint of the (essentialized) braid hyperplane arrangement, which identifies the monomial basis with signed characteristic functions of the interiors of permutohedral tangent cones. We show that the indecomposable quotient Lie coalgebra is obtained by restricting functions to chambers of the adjoint arrangement, i.e. by quotienting out the higher codimensions. The resulting functions are characterized by the Steinmann relations of axiomatic quantum field theory, demonstrating an equivalence between the Steinmann relations, tangent cones to (generalized) permutohedra, and having algebraic structure internal to species. Our results give a new interpretation of a construction appearing in the mathematically rigorous formulation of renormalization by Epstein-Glaser, called causal perturbation theory. In particular, we show that operator products of time-ordered products correspond to the \hbox{H-basis} of the cocommutative Hopf monoid of set compositions, and generalized retarded products correspond to a spanning set of its primitive part Lie algebra.
\end{abstract}

\maketitle

\setcounter{tocdepth}{1} 
\hypertarget{foo}{ }
\tableofcontents

\section*{Introduction}\label{intro}

André Joyal's theory of set species \cite{joyal1981theorie}, \cite{joyal1986foncteurs}, \cite{bergeron1998combinatorial}, and more generally stuff types \cite{baez2001finite}, \cite{morton2006categorified}, is the result of applying what is sometimes called the gauge principle, or categorification, to exponential generating functions in enumerative combinatorics, which, in this context, says:
\begin{center}
\textit{do not identify sets with the same cardinality; instead, just remember all the ways in which}\\[3pt]
\textit{they can be identified, that is, remember the bijections between them.}\smallskip
\end{center}
This amounts to using arbitrary finite sets to label combinatorial objects, instead of always using e.g. $[n]=\{1,\dots,n\}$, and taking sets of labeled objects acted on by relabelings, instead of just taking their cardinalities. Every set species induces a generating function by forgetfully decategorifying. More recently, Aguiar and Mahajan \cite{aguiar2010monoidal}, \cite{aguiar2013hopf} showed that the plethora of combinatorial graded Hopf algebras which appear in the literature are similarly decategorifications of Hopf monoids internal to vector species. Their species approach beautifully unifies the study of combinatorial Hopf algebras.

Formally, a set (resp. vector) species $\textbf{p}$ is a presheaf of sets (resp. vector spaces) on the category $\textsf{S}$ of finite sets and bijections.\footnote{\ we refer to copresheaves on $\textsf{S}$ as \emph{cospecies}} The value $\textbf{p}[I]$ of $\textbf{p}$ on a finite set $I$ is interpreted as the (linearized) collection of all combinatorial objects of a certain type which have been labeled by $I$. Up to isomorphism, a species is an infinite sequence of objects such that the $n$th object is equipped with a right action of the symmetric group $\textsf{S}_n$ of the set $\{1,\dots, n\}$, sometimes called a `symmetric sequence' or `S-module'.

Let $(\textsf{g})\textsf{Vec}$ denote the category of ($\bN$-graded) vector spaces over a field $\Bbbk$ of characteristic zero. Joyal showed that vector species are equivalently analytic endofunctors on $\textsf{Vec}$ via a certain generalized bosonic Fock space construction, given by
\[   
\overline{\cK}_{(-)}(\textbf{p}): \textsf{Vec} \to   \textsf{gVec}, \qquad     V\mapsto  \overline{\cK}_{V}(\textbf{p}) =\bigoplus_{ n\in \bN } \textbf{p}[n] \otimes_{\textsf{S}_n}      V^{ \otimes n }.\footnote{\ where $\textbf{p}[n]= \textbf{p}[\{ 1, \dots, n \}]$}   
\]
Thus, the $\textbf{p}[n]$ get treated as coefficients of a power series whose argument is a vector space. The classical bosonic Fock space is recovered by setting $\textbf{p}$ equal to the exponential species $\textbf{E}$, which has $\textbf{E}[I]=\Bbbk$ for all finite sets $I$.\footnote{\ Bo\.zejko, Gu\c t\u a and Maassen showed that creation-annihilation operators can be generalized to this setting \cite{guta00}, \cite{MR1923173}, and these ideas were further developed in \cite[Chapter 19]{aguiar2010monoidal}} The analytic endofunctor $\widehat{\textbf{p}}$ associated to $\textbf{p}$ is obtained by forgetting the grading,
\[
\widehat{\textbf{p}}: \textsf{Vec} \to \textsf{Vec}, \qquad V\mapsto  \overline{\cK}_{V}(\textbf{p}) 
.\]
In the case that $\textbf{p}[I]=0$ for all large enough $I$, i.e. `polynomial species', one recovers classical Schur functors $ \textsf{finVec}\to  \textsf{finVec}$. 


Species may be equipped with a handful of monoidal products, which categorify familiar operations on formal power series. In particular, we can take the Day convolution of set species or vector species with respect to the disjoint union of finite sets,
\[    \textbf{p}_1\bigcdot  \textbf{p}_2[I]=   \textbf{p}_1\otimes_{\text{Day}} \textbf{p}_2[I]=  \coprod_{S\sqcup T=I} \textbf{p}_1[S] \otimes \textbf{p}_2[T]    .  \]
This is often called the Cauchy product of species, since it categorifies the Cauchy product of formal power series. It is induced by `pointwise multiplying' species as analytic endofunctors,
\[
\widehat{\textbf{p}_1 \bigcdot \textbf{p}_2}(V)\cong   \widehat{\textbf{p}}_1(V)\otimes \widehat{\textbf{p}}_2(V)
.\] 
Aguiar and Mahajan's Hopf theory in species concerns Hopf monoids defined in set species and vector species with respect to the Day convolution. In the case of Hopf monoids in vector species, we also have internal Lie (co)algebras and universal (co)enveloping algebras, and analogs of the Poincar\'e-Birkhoff-Witt and Cartier-Milnor-Moore theorems. The undecorated bosonic Fock space of a vector species is its image under the functor
\[\overline{\cK}(-) = \overline{\cK}_{\Bbbk}(-) :\textsf{VecSp}\to \textsf{gVec}, \qquad \textbf{p}\mapsto \overline{\cK}_{\Bbbk}(\textbf{p})=   \bigoplus_{ n\in \bN } \big (\textbf{p}[n]\big )_{\textsf{S}_n}       .\] 
Many well-known graded Hopf algebras in combinatorics are the undecorated Fock spaces of Hopf monoids in vector species \cite[Part III]{aguiar2010monoidal}, the crucial point being that the various generalized Fock space constructions preserve Hopf monoids. 
 

Note that (symmetric May) operads are also monoids internal to species, but with respect to the monoidal product induced by composing species as analytic endofunctors, called plethysm,\footnote{\ plethysm was originally used as a name for the image of this monoidal product in the Grothendieck ring of the category of vector species, which is the ring of symmetric functions $\Lambda=\textbf{Sym}$} i.e. the structure of an operad on $\textbf{p}$ is equivalently the structure of a monad on $\widehat{\textbf{p}}$. As long as $\textbf{p}_2[\emptyset]=0$, plethysm is given by
\[    \textbf{p}_1\circ  \textbf{p}_2[I]= \coprod_{ P } \textbf{p}_1[P] \otimes \bigotimes_{S_j\in P} \textbf{p}_2[S_j] .  \]
The coproduct is over all set partitions $P=\{S_1,\dots,S_k\}$ of $I$. This monoidal product categorifies the composition of formal power series. 

There is an equivalent description of Hopf theory in species in terms of left (co)modules of the (co)operads $\Com^{ (\ast) }$, $\Ass^{ (\ast) }$, $\Lie^{ (\ast) }$ \cite[Appendix B.5]{aguiar2010monoidal}, i.e. (co)algebras over the corresponding left (co)action (co)monads. This is a useful perspective; it puts Hopf theory in species within the context of both the generalization to \emph{right} (co)modules of (co)operads, and the generalization to (co)algebras over other (co)monads on species, in particular algebras over $\textsf{S}$-colored operads \cite[Section 2.3.2]{ward2019massey}, \cite[Section 3]{MR3134040}, also studied in the guise of $\cF$-ops for a Feynman category $\cF$ \cite[Definition 1.5.1]{MR3636409}, for example modular operads. 



For the foundations of Hopf theory in species, see \cite{aguiar2010monoidal}, \cite{aguiar2013hopf}. Aguiar and Mahajan's clean category-theoretic approach clarifies and generalizes the work of several people, in particular Barratt \cite{barratt1978twisted}, Joyal \cite{joyal1986foncteurs}, Schmitt \cite{Bill93}, and Stover \cite{stover1993equivalence}. The reflection hyperplane arrangement of the type $A$ root system, called the braid arrangement, provides consistent geometric interpretations of the theory, which motivates the development of aspects of the theory over generic real hyperplane arrangements \cite{aguiar2017topics}, \cite{aguiar2020bimonoids}.\footnote{\ Aguiar and Mahajan say in \cite{aguiar2020bimonoids} that a more structured theory for reflection hyperplane arrangements, e.g. other Dynkin types, exists, and will be the subject of a separate work.} In this paper, we stay in type $A$, but we extend the geometric interpretations to the adjoint\footnote{\ in the sense of \cite[Section 1.9.2]{aguiar2017topics}} of the braid arrangement. The adjoint braid arrangement lives in the dual root space, and consists of hyperplanes which are spanned by coroots. This hyperplane arrangement has several names. It is known as the restricted all-subset arrangement \cite{kamiya2010ranking}, \cite{kamiya2012arrangements}, \cite{billera2012maximal}, \cite[Section 6.3.12]{aguiar2017topics}, the resonance arrangement \cite{MR2836109}, \cite{cavalierires}, \cite{billerabooleanprod}, \cite{gutekunst2019root}, and the root arrangement \cite{MR3917218}. Its spherical representation is called the Steinmann planet, or Steinmann sphere, by physicists, e.g. \cite[Figure A.4]{epstein2016}, which may be identified with the boundary of the convex hull of coroots. The Steinmann sphere is the adjoint analog of the type $A$ Coxeter complex, see \autoref{fig:convexhulls}. 


 \begin{figure}[t]
	\centering
	\includegraphics[scale=0.5]{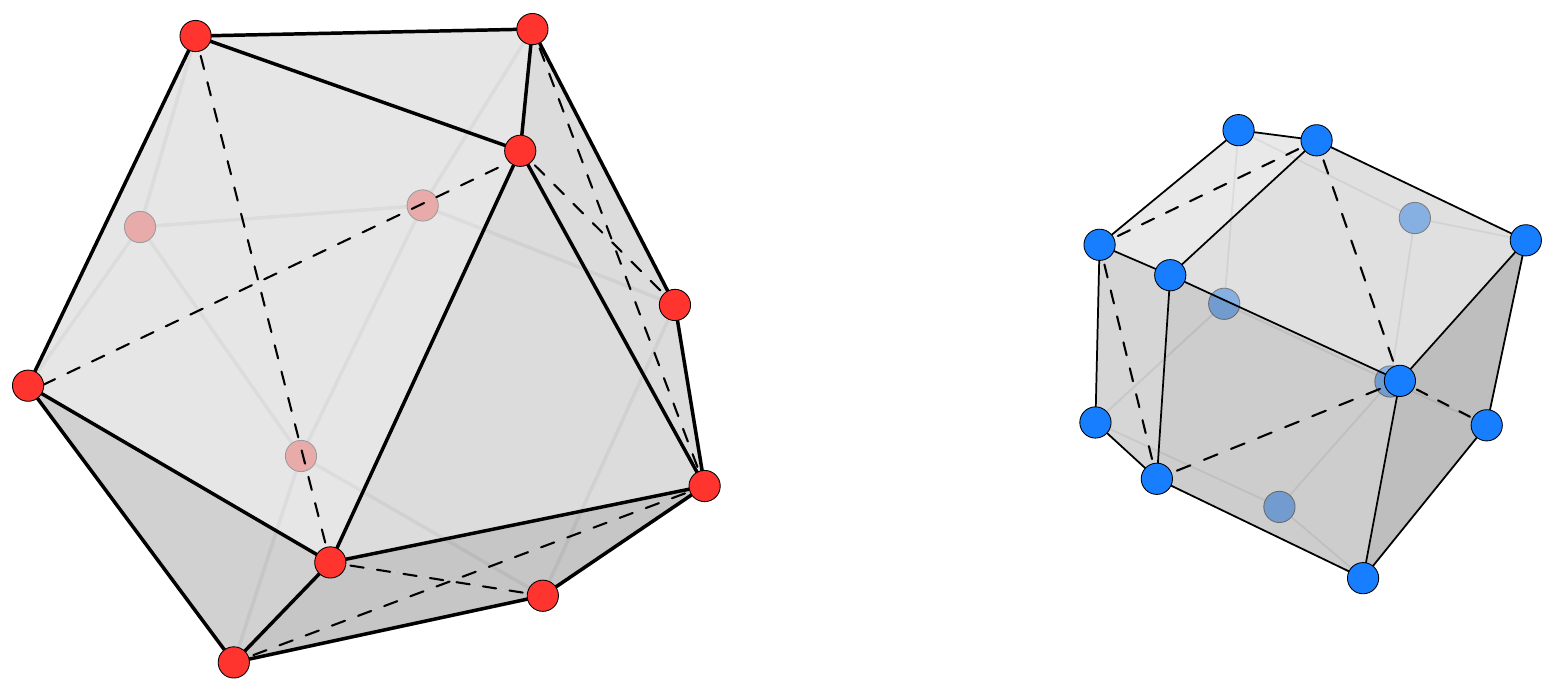}
	\caption{The convex hulls of the type $A$ coroots (left) and fundamental weights (right) in four coordinates. The intersections with the adjoint braid arrangement and braid arrangement are shown with dotted lines, which define the Steinmann sphere and Coxeter complex respectively.}
	\label{fig:convexhulls}
\end{figure} 

At the heart of Hopf theory in species is the cocommutative Hopf monoid of set compositions $\Sig$, together with its dual, the commutative Hopf monoid of set compositions $\Sig^\ast$. These Hopf monoids play a special role because set compositions index (co)associative algebraic operations in species. More familiar objects are perhaps the undecorated bosonic Fock spaces of $\Sig$ and $\Sig^\ast$, which are the algebras of noncommutative symmetric functions and quasisymmetric functions respectively,
\[ \text{N}\Lambda=   \textbf{NSym}\cong \overline{\cK}(\Sig) \qquad \text{and} \qquad \text{Q}\Lambda= \textbf{QSym}\cong \overline{\cK}(\Sig^\ast).\footnote{\ these algebras play the role of symmetric functions $\Lambda= \textbf{Sym}$ in quantum Schur-Weyl duality at $q=0$ \cite[Sections 3-4]{thibon01nsymqsym}}    \]
See \cite[Chapter 17]{aguiar2010monoidal}. In this paper, we geometrically realize $\Sig^\ast$ over the adjoint braid arrangement by identifying the monomial basis with signed characteristic functions of open permutohedral tangent cones (in fact, we shall think of them as formal linear combinations of open conical spaces, and so refer to them as characteristic functionals). We denote this realization by $\adshuff$. A related construction appears in \cite[Section 5]{menous2013mould} for the undecorated \emph{full} Fock space of $\boldsymbol{\Sigma}^\ast$, which is word quasisymmetric functions   
\[\text{P}\Pi=\textbf{WQSym}\cong \cK(\Sig^\ast) ,\] 
also called noncommutative quasisymmetric functions $\textbf{NCQSym}$. This graded Hopf algebra has been studied in several places, e.g. \cite[Section 6.2.4]{MR2225808}, \cite{MR2209212},  \cite{MR2228332},    \cite{MR2555523}. 
 
There is a more classical geometric realization of $\Sig^\ast$ over the braid arrangement, where the monomial basis is identified with characteristic functions of relatively open faces. We denote this realization by $\shuff$. The realizations $\adshuff$ and $\shuff$ are dual in the sense of polyhedral algebras \cite[Theorem 2.7]{MR1731815}. To express this duality, we introduce the cone basis of $\Sig^\ast$, which is the image of set compositions under the standard homomorphism $\textbf{O}\twoheadrightarrow \Sig^\ast$, where $\textbf{O}$ is the commutative Hopf monoid of preposets. The geometric realizations of the cone basis are characteristic functions/functionals of closed convex cones. 

As we show, the beauty of the adjoint realization $\adshuff$ is that its indecomposable quotient is obtained by simply restricting functionals to chambers, i.e. by quotienting out the higher codimensions. We denote the resulting geometrically realized Lie coalgebra by $\check{\textbf{Z}}\textbf{ie}^\ast$. It is isomorphic to the Lie coalgebra $\textbf{Zie}^\ast$ which is the dual of the free Lie algebra on the positive exponential species
\[
\textbf{Zie}=\CMcal{L}ie(\textbf{E}_+^\ast)=  \Lie \circ \textbf{E}_+^\ast 
.\] 
See \cite[Section 11.9]{aguiar2010monoidal}. The adjoint analog of this construction, i.e. the restriction of $\shuff$ to Weyl chambers, is a geometric realization $\hat{\textbf{L}}^\ast$ of the commutative Hopf monoid of linear orders $\textbf{L}^\ast$, which is the dual of the universal enveloping algebra of the Lie operad
\[
\textbf{L}=  \mathcal{U} ( \Lie)=\textbf{E}^\ast \circ \Lie
.\] 
See \cite[Section 15]{aguiar2013hopf}. The observation we make here is that moving to the adjoint arrangement reverses the order of plethysm. The geometric aspect of this reversal is the following discussion. 

Let the \emph{symmetry} of a piecewise-constant function on a hyperplane arrangement be the degree to which it is constant in the direction of one-dimensional flats. Roughly speaking, `lumping' coordinates into lumps corresponds to higher codimensions on the braid arrangement and increased symmetry on the adjoint braid arrangement, whereas `cutting' coordinates into blocks corresponds to increased symmetry on the braid arrangement and higher codimensions on the adjoint braid arrangement. For precise definitions, see \autoref{preposet}.
\bgroup
\def\arraystretch{1.2}  
\begin{table}[H] 
\begin{tabular}{|c|c|c|}
\hline
  &   \begin{tabular}{@{}c@{}}higher codimensions\end{tabular} & \begin{tabular}{@{}c@{}}increased symmetry\end{tabular}   \\ \hline
 \begin{tabular}{@{}c@{}}braid arrangement \end{tabular}  &   lumping    &   cutting     \\ \hline
 \begin{tabular}{@{}c@{}}adjoint braid arrangement \end{tabular}  &   cutting    &   lumping     \\ \hline
\end{tabular}
\end{table}
\egroup
\noindent
Note also that lumping corresponds to right coactions of cooperads, whereas cutting corresponds to left coactions of cooperads. On the braid arrangement, by either quotienting out codimensions and then symmetry, or symmetry and then codimensions, we obtain the following commutative square,
\begin{center}
\begin{tikzcd}[column sep=large,row sep=large] 
\shuff  	\arrow[d, twoheadrightarrow, "\text{sym}"']     \arrow[r,  twoheadrightarrow,  "\text{codim}"]   &  \hat{\bL}^\ast  	\arrow[d,  twoheadrightarrow, "\text{sym}"]     \\
\hat{\textbf{Z}}\textbf{ie}^\ast 	\arrow[r,  twoheadrightarrow , "\text{codim}"']     & \hat{\textbf{L}}\textbf{ie}^\ast
\end{tikzcd}
\end{center}
On the adjoint braid arrangement, by either quotienting out codimensions and then symmetry, or symmetry and then codimensions, we obtain the following commutative square,
\begin{center}
\begin{tikzcd}[column sep=large,row sep=large] 
\adshuff   	\arrow[d, twoheadrightarrow, "\text{codim}"']     \arrow[r,  twoheadrightarrow,  "\text{sym}"]   & \check{\bL}^\ast  	\arrow[d,  twoheadrightarrow, "\text{codim}"]     \\
\check{\textbf{Z}}\textbf{ie}^\ast 	\arrow[r,  twoheadrightarrow , "\text{sym}"']     & \check{\textbf{L}}\textbf{ie}^\ast
\end{tikzcd}
\end{center}
In this paper, we only consider the quotients by codimensions.

Let $\textbf{L}^\vee$ denote the species of formal $\Bbbk$-linear combinations of chambers of the adjoint braid arrangement. The (classical) Steinmann relations are certain four-term linear relations on the components of $\textbf{L}^\vee$, first appearing in the foundations of axiomatic quantum field theory \cite{steinmann1960zusammenhang}, \cite{steinmann1960}, \cite[p. 827-828]{streater1975outline}. More recently, they have been studied in the context of the revived non-local approach to scattering amplitudes, where they appear to be related to cluster algebras \cite{drummond2018cluster}, \cite{caron2019cosmic}, \cite{Caron-Huot:2020bkp}. 

Let a Steinmann functional over $I$ be a linear functional on the vector space $\textbf{L}^\vee[I]$ which respects the Steinmann relations. In \cite{lno2019}, it was shown that Steinmann functionals (which were denoted there by $\Gam^\ast$) form a Lie coalgebra in species, with cobracket the discrete differentiation of functionals across hyperplanes. The Steinmann relations are exactly what one needs in order to have factorization of the derivative, and so they are necessary for a Lie cobracket. However, they are also sufficient, because it turns out that if you can factorize once (in all possible ways), then you can factorize arbitrarily often \cite[Theorem 5.3]{lno2019}. In this paper, we show that our geometric realization of $\textbf{Z}\textbf{ie}^\ast$ is precisely the Lie coalgebra of Steinmann functionals from \cite{lno2019},
\[      \Gam^\ast=\check{\textbf{Z}}\textbf{ie}^\ast  \qquad \big ( \cong \textbf{Z}\textbf{ie}^\ast=  \Lie^\ast  \circ \textbf{E}_+\big )  .   \]
Since $\check{\textbf{Z}}\textbf{ie}^\ast$ is equivalently the span of characteristic functionals of generalized permutohedral tangent cones, this result is clearly closely related to the universality of generalized permutohedra \cite[Theorem 6.1]{aguiar2017hopf}, see \autoref{ref}. 

Dually, we obtain a geometric realization $\check{\textbf{Z}}\textbf{ie}$ of the Lie algebra $\textbf{Zie}$ as the quotient of $\textbf{L}^\vee$ by the Steinmann relations, where the Lie structure (in the form of a left module of the Lie operad $\Lie$) is the action of semisimple Lie elements on faces \cite[Section 3.2]{lno2019}. The Lie algebra $\check{\textbf{Z}}\textbf{ie}$ is very closely related to a structure sometimes called the Steinmann algebra \cite[Section 6]{Ruelle}, \cite[Section III.1]{bros}, which is an ordinary graded Lie algebra based on the structure map of $\check{\textbf{Z}}\textbf{ie}$. In particular, see Ruelle's identity \cite[Equation 6.6]{Ruelle}, \cite[Section 4.3]{epstein1976general}, which describes the Lie bracket.

The dual polyhedral algebras $\shuff[n]$ and $\check{\Sig}^\ast[n]$ were first studied in \cite{early2017canonical}. Note that Early denotes $\check{\Sig}^\ast[n]$ by $\hat{\cP}^n$. The quotients of $\check{\Sig}^\ast[n]$ corresponding to the Lie cooperad $\textbf{Lie}^\ast$, linear orders $\textbf{L}^\ast$, and $\textbf{Zie}^\ast$, denoted respectively
\[ \cP_1^n\cong  \Lie^\ast[n], \qquad  \hat{\cP}_1^n\cong \textbf{L}^\ast[n], \qquad  \cP^n\cong \textbf{Zie}^\ast[n], \]
were considered in relation to the cone basis, and a certain second basis. The $\tc$-basis of $\textbf{Zie}^\ast$, defined in \autoref{lie}, corresponds to the image of this second basis. 



\subsection*{Perturbative Algebraic Quantum Field Theory} 

Our results give a new interpretation of a construction appearing in the mathematically rigorous formulation of renormalization by Epstein-Glaser \cite{ep73roleofloc},\footnote{\ based on earlier work of Stückelberg, Bogoliubov-Shirkov, and others} known as causal perturbation theory, and in the algebraic formalism of Epstein-Glaser-Stora \cite{egs74}, \cite{epstein1976general} for studying generalized retarded functions. Causal perturbation theory has since been absorbed into the modern mathematically clear and precise theory known as perturbative algebraic quantum field theory (pAQFT), see \cite{klaus2000micro}, \cite{rejzner2016pQFT}, \cite{dutsch2019perturbative}, \cite[\href{https://ncatlab.org/nlab/show/geometry+of+physics+--+perturbative+quantum+field+theory}{nLab}]{perturbative_quantum_field_theory}.

Let $I$ be a finite set, and let $\cX$ be a smooth manifold which is additionally a time-oriented globally hyperbolic Lorentzian manifold. In \cite[p. 157]{epstein2016}, a system of operator products of time-ordered products, or, after taking vacuum expectation values, a system of generalized \hbox{time-ordered} functions, is defined to be a function on set compositions $\Sigma[I]$ of $I$ into distributions on the space of configurations 
\[  \cX^I=\big \{\text{functions } I\to \cX   \big\}\] 
which satisfies certain physically motivated properties. In the case of a system of operator products of time-ordered products, these are operator-valued distributions, i.e. linear functionals sending compactly supported smooth functions to elements of the algebra of free quantum observables, which in pAQFT is a formal power series $\ast$-algebra obtained via Moyal formal deformation quantization. Note that the key property, known as causal factorization, which gives rise to the causal additivity of the corresponding perturbative $\text{S}$-matrix scheme, is naturally expressed in terms of an aspect of the Hopf structure of $\Sig$ known as the Tits product, see e.g. \cite[Section 1.4.6]{brown08}. The Tits product is the action of $\Sig$ on itself by Hopf powers \cite[Section 13]{aguiar2013hopf}. 

By taking compatible maps for each finite set $I$, we obtain a morphism of species, denoted
\[   
\text{T}:  \Sig \to  \text{Dist}(\cX^{(-)})
,\qquad  
(S_1,\dots, S_k)\mapsto \text{T}(S_1)\dots \text{T}(S_k)  
.\]
As far as we are aware, the species style notation $\text{T}(S_1)\dots\text{T}(S_k)$ for operator-valued distributions first appears in Steinmann's book \cite{steinbook71}, and then extensively in \cite{ep73roleofloc}, where all at once it exposes \hbox{species-theoretic} algebraic structures in QFT. The notation was formalized in terms of set compositions by Epstein-Glaser-Stora \cite[Section 4.1]{epstein1976general}. It is discussed in \cite[Remark 15.33]{perturbative_quantum_field_theory} (its use there is slightly different to \cite{ep73roleofloc}, but the same structure emerges). See also \cite[Chapter 4]{MR3753672}, where the connection with species is made. 

Let $\cU$ be the universal enveloping map which embeds the primitive part Lie algebra $\textbf{Zie}$ of $\Sig$,
\[
\cU:\textbf{Zie}\hookrightarrow \Sig
.\] 
See \autoref{lie}. In the setting where $\textbf{Zie}$ is realized over the adjoint braid arrangement, we show that the composition
\[      
\textbf{Zie}\xrightarrow{\cU}    \Sig \xrightarrow{\text{T}}  \text{Dist}(\cX^{(-)})     
\]
corresponds to the association of generalized retarded products, or generalized retarded functions, to chambers of the adjoint braid arrangement (=geometric cells of the Steinmann sphere), as defined in \cite[Equation 79, p. 260]{ep73roleofloc}, \cite[Equation 1, p.26]{epstein1976general}, \cite[Equations 35, 36]{epstein2016}. See also \cite{Huz1}, \cite[Section 9]{Huz2}. Thus, generalized retarded products/functions span the image of the primitive part of $\Sig$ in distributions. 

The primitive elements which correspond to chambers of the adjoint braid arrangement are called Dynkin elements in \cite[Chapter 14.1]{aguiar2017topics}, where they are constructed for arbitrary real hyperplane arrangements. These Dynkin elements are associated to generic halfspaces of a hyperplane arrangement, which are in natural bijection with chambers of the adjoint arrangement. 

It is well-known that graded Hopf algebras encode combinatorial aspects of renormalization in perturbative QFT \cite{connes1999hopf}, \cite{ebrahimi2005hopf}, \cite{figueroa2005combinatorial}, \cite{Kreimer05}. A self-contained introduction to this theory is given in \cite[Chapter 1]{connes08}. The species analogs of the Connes-Kreimer Hopf algebras of rooted trees have been defined by Aguiar and Mahajan \cite[Section 13.3]{aguiar2010monoidal}, and we direct the reader there for their relationship to the algebras we consider in this paper.

\subsection*{Losev-Manin Moduli Space} 

The usual geometric interpretation of $\Sig$ over the braid arrangement identifies the basis which is dual to the monomial $\tM$-basis, called the $\tH$-basis, with faces \cite[Chapter 10]{aguiar2010monoidal}. This is how Tits's classical interpretation of the Tits product, as projections of faces, is obtained. See also the geometric interpretation of Lie and Zie elements, which correspond to the (Lie algebra) homomorphisms
\[\Lie\hookrightarrow  \textbf{Zie}\hookrightarrow  \Sig,\] 
and their generalization to generic hyperplane arrangements \cite[Chapter 10]{aguiar2017topics}. Alternatively, we can identify the $\tH$-basis with faces of the permutohedron, which occurs naturally in the following construction. 

\begin{figure}[t]
	\centering
	\includegraphics[scale=0.67]{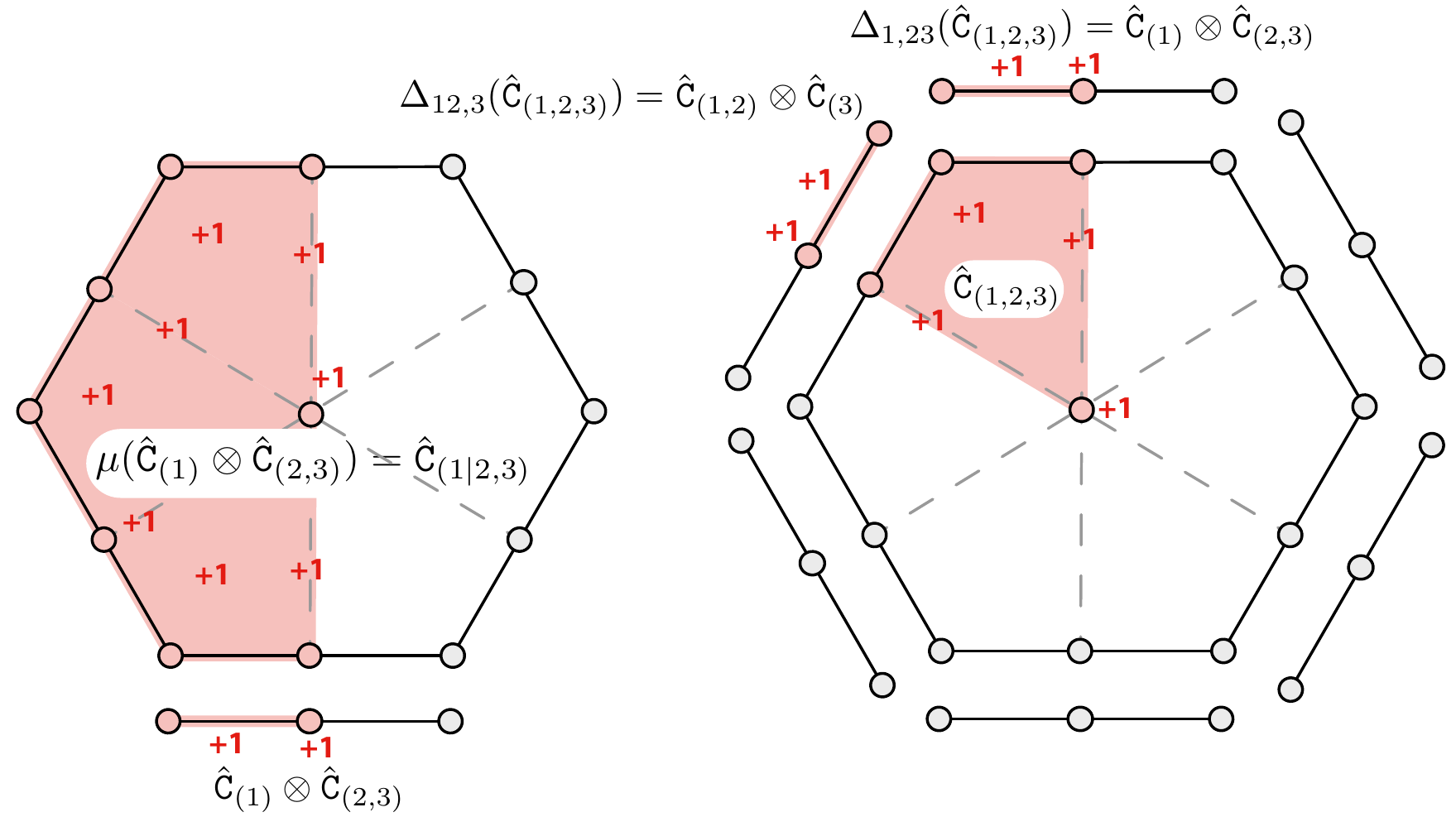}
	\caption{The geometric realization of the Hopf monoid $\Sig^\ast$, extended to the tropical toric compactification $\bT\Sigma$ of the type $A$ root system, showing multiplication $\mu$ (signed-quasishuffling, dual to projecting onto facets) and comultiplication $\Delta$ (deconcatenation, dual to embedding facets). The circular nodes show the $\bB$-points of $\bT\Sigma$, which may be identified with set compositions $\bB\Sigma=\Sigma$.}
	\label{fig:multcomult}
\end{figure}

First, let us consider the construction over the complex numbers $\bC$, which is more classical. Let $\bC\text{T}^I$ denote the moduli space parameterizing the `tube' $\bC^\times=\bC\setminus\{0\}$ with $I$-marked points,
\[
\bC\text{T}^I=(\bC^\times)^I/\bC^\times=\bigslant{\{\text{functions}\ I\to\bC^\times\}}{\bC^\times} 
.\] 
This is the maximal torus of $\text{PGL}_I(\bC)$. Let \emph{complex permutohedral space} $\bC\Sigma^I$ be the toric compactification of $\bC\text{T}^I$ with respect to the braid arrangement fan, sometimes called the toric variety associated to (type $A$) Weyl chambers \cite{procesipermvar}. Losev and Manin gave a realization of $\bC\Sigma^I$ as a moduli space parameterizing strings of Riemann spheres $\bC\text{P}^1$, glued at the poles, with $I$-marked points \cite{losevmanin}, \cite{batyrevblume10}. By generalizing our construction of $\Sigma$ in \autoref{comp}, as indicated in \autoref{complexperm}, complex permutohedral space is naturally a Hopf monoid in set species $\bC\Sigma$. The multiplication corresponds to embedding facets of the permutohedron, and the comultiplication corresponds to projecting the permutohedron onto facets.\footnote{\ Note that this Hopf monoid structure on permutohedral space is dual to the Hopf monoid structure of generalized permutohedra, studied in \cite{aguiar2017hopf}, which is induced by letting generalized permutohedra encode functions/forms on permutohedral space by viewing them as Newton/momentum polytopes. Such Newton polytopes appear in the worldline formalism, or Schwinger parametrization, approach to Feynman amplitudes \cite{schultka2018toric}.} 

A tropical version of the Losev-Manin moduli space is more relevant to this paper. Recall that one can also do algebraic geometry over the tropical rig $\bT=\bR\cup \{- \infty\}$ \cite{mikhalkin2009tropical}. We may view the type $A$ root system, which we denote by $\text{T}^I=\bT\text{T}^I$, as the tropical algebraic torus which is the tropicalization of $\bC\text{T}^I$; $\text{T}^I$ is the moduli space parameterizing the line $\bT^\times=\bR$ with $I$-marked points, 
\[
\text{T}^I=(\bT^\times)^I/\bT^\times  = \bR^I/\bR   =    \bigslant{ \{   \text{functions}\ I\to \bR  \}}\bR 
.\] 
Let \emph{tropical permutohedral space} $\bT\Sigma^I$ be the toric compactification\footnote{\ see e.g. \cite[Chapter 1]{meyer2011intersection} for the construction of tropical toric varieties} of $\text{T}^I$ with respect to the braid arrangement fan. We may interpret $\bT\Sigma^I$ as adding would-be limiting configurations to $\text{T}^I$ where points are separated by infinite distances (or proper times if we view $\bR$ as a worldline). Notice that $\bC\Sigma^I$ has a similar interpretation in terms of points moving infinity far away on a worldsheet. As in the complex case, the compactifications $\bT\Sigma^I$ form a Hopf monoid in set species $\bT\Sigma$. The Hopf monoid of set compositions $\Sigma$ is recovered by restricting the Hopf monoid $\bT\Sigma$ to $\bB$-points, where $\bB=\{-\infty,0 \}\subset \bT$ is the Boolean tropical rig. Moreover, the geometric realization of $\Sig^\ast$ as functions on the $\text{T}^I$ may be extended by continuity to functions on the compactifications $\bT\Sigma^I$, and the commutative Hopf algebraic structure of $\Sig^\ast$ on these functions is exactly that which is induced by the cocommutative Hopf monoid structure of the underlying space $\bT\Sigma$, see \autoref{fig:multcomult}. In particular, the comultiplication of $\Sig^\ast$, which is deconcatenation of compositions, is the restriction of functions to facets of the permutohedron. 

In this paper, we will use this geometric interpretation only to draw pictures. We hope to describe applications of these structures in future work.  

\subsection*{Structure} 

This paper has five sections. In \autoref{comb}, we describe combinatorial gadgets which index aspects of the type $A$ hyperplane arrangements. In particular, we introduce adjoint families, which generalize preposets and maximal unbalanced families, and index cones of the adjoint braid arrangement. In \autoref{alg}, we define the algebras in species which feature in this paper. We construct several bases of the indecomposable quotient of $\Sig^\ast$. In \autoref{georel}, we describe the two geometric realizations of $\Sig^\ast$. In \autoref{main}, we prove our main results. We show that the indecomposable quotient of $\Sig^\ast$ naturally lives on the chambers of the adjoint braid arrangement, has cobracket discrete differentiation across hyperplanes, and is characterized by the Steinmann relations. In \autoref{QFT}, we describe the connection with perturbative QFT, and show that generalized retarded products/functions correspond to the primitive part Lie algebra of $\Sig$. 

\subsection*{Acknowledgments} 

We would like to thank Penn State university for their support. We would also like to thank Nick Early for helpful discussions, and an anonymous referee for useful suggestions. 

\section{Combinatorial Background} \label{comb}

We briefly recall species, and Hopf monoids in species. We recall the combinatorial gadgets: set partitions, set compositions, preposets, and maximal unbalanced families. We introduce adjoint families, which simultaneously generalize preposets and maximal unbalanced families. 

We shall model each preposet $p$ of a finite set $I$ as the collection of all order preserving functions $\textbf{2}\hookrightarrow I$, where $\textbf{2}$ is the ordinal with two elements. This defines the cospecies of preposets $\text{O}$, which has subcospecies corresponding to partitions $\Pi^\ast$, linear orders $\text{L}^\ast$, and set compositions $\Sigma^\ast$. We then construct their adjoint analogs $\Pi^\vee$, $\text{L}^\vee$, $\Sigma^\vee$, $\text{O}^\vee$, which are species whose elements are certain collections of functions $I\twoheadrightarrow \textbf{2}$. The species of maximal unbalanced families is $\text{L}^\vee$, and the species of adjoint families is $\text{O}^\vee$. A preposet on $I$ is naturally an adjoint family if it is modeled as the collection of all order preserving functions $I\twoheadrightarrow \textbf{2}$, which gives an embedding of preposets into adjoint families $\text{O}\hookrightarrow \text{O}^\vee$. 

We adopt the following terminology for real hyperplane arrangements; we refer to cones which are generated by vectors contained in one-dimensional flats as \emph{conical subspaces} of the arrangement, and we let \emph{flats}/\emph{chambers}/\emph{faces}/\emph{cones} refer to pointwise products of characteristic functions of open or closed halfspaces of the arrangement (see \autoref{georel} for exact definitions). An exception is we let `permutohedral cone' refer to permutohedral conical subspaces, see below. The correspondence between our combinatorial gadgets and characteristic functions on the type $A$ hyperplane arrangements will be as follows.
\bgroup
\def\arraystretch{1.2}%
\begin{table}[H] 
\begin{tabular}{|c|c|c|c|c|}
\hline
  &  flats& chambers & faces & cones \\ \hline
braid arrangement  & $\Pi^\ast$    &  $\text{L}^\ast$ & $\Sigma^\ast$ &  $\text{O}$   \\ \hline
adjoint braid arrangement  &$\Pi^\vee$    &    $\text{L}^\vee$ &  $\Sigma^\vee$ &  $\text{O}^\vee$ \\ \hline
\end{tabular}
\end{table}
\egroup
\noindent Then, via the well-known dual cone construction, there are one-to-one correspondences between characteristic functions on the braid arrangement and conical subspaces of the adjoint braid arrangement \emph{which are generated by coroots}. This restriction to aspects of the adjoint braid arrangement which can be `seen' by coroots stops it being so pathological.
\bgroup
\def\arraystretch{1.2}  
\begin{table}[H] 
\begin{tabular}{|c|c|c|c|c|}
\hline
  &   \begin{tabular}{@{}c@{}}semisimple\\ subspaces\end{tabular} & \begin{tabular}{@{}c@{}}pointed\\  permutohedral cones\end{tabular} & \begin{tabular}{@{}c@{}}permutohedral\\ cones\end{tabular} & \begin{tabular}{@{}c@{}}generalized\\  permutohedral cones\end{tabular}  \\ \hline
 \begin{tabular}{@{}c@{}}adjoint braid\\ arrangement \end{tabular}  &$\Pi^\ast$    &    $\text{L}^\ast$ &  $\Sigma^\ast$ &  $\text{O}$ \\ \hline
\end{tabular}
\end{table}
\egroup
\noindent A \emph{semisimple subspace} is a linear subspace which is spanned by coroots. These subspaces feature in the study of reflection length in the affine Weyl group $\wt{A}_{n-1}$ \cite{MR3917218}. A generalized permutohedral cone is equivalently a conical subspace which is generated by coroots. They are studied in \cite{gutekunst2019root}.  

\subsection{Hopf Monoids in Species}
 
Our references for Hopf monoids in species are \cite{aguiar2010monoidal}, \cite{aguiar2013hopf}. See also \cite[Section 2]{aguiar2017hopf} for a quick introduction.  

Let $\textsf{Set}$ denote the cartesian category of sets $\cX\in \textsf{Set}$, and let $\textsf{Vec}$ denote the tensor category of vector spaces over a field $\Bbbk$ of characteristic zero. Let $\sS$ denote the monoidal category with objects finite sets  $I\in \sS$, morphisms bijective functions $\sigma:J\to I$, and monoidal product the restriction of the disjoint union of sets to finite sets. 

We define a \emph{set species} $\text{p}$ to be a presheaf on $\sS$, that is, any functor of the form
\[         
\text{p}:   \sS^{\text{op} } \to  \textsf{Set}
, \qquad   
I\mapsto  \text{p}[I]
.\]
Explicitly, to every finite $I$ we assign a set $\text{p}[I]$, and to every bijection of finite sets $\sigma:J\to I$ we assign a bijection $\text{p}[\sigma]:\text{p}[I]\to \text{p}[J]$ such that the composition of bijections is preserved. Often, $\text{p}[I]$ is the collection of labelings/`probes' of a set of objects by $I$, and $\text{p}[\sigma]:\text{p}[I]\to \text{p}[J]$ sends an $I$-labeling to its precomposition with $\sigma$. This is reminiscent of functorial geometry.

A \emph{vector cospecies} $\textbf{q}$ is a copresheaf of vector spaces on $\sS$, that is, any functor of the form
\[         
\textbf{q}:   \sS \to \textsf{Vec} , \qquad   I\mapsto  \textbf{q}[I] 
.\]
If $\text{p}$ is a set species, then we denote the vector cospecies of $\Bbbk$-valued functions on $\text{p}$ by $\textbf{p}^\ast$, i.e. given the monoidal functor
\[   \Bbbk^{(-)}:  \textsf{Set}^{ \text{op} } \to  \textsf{Vec}, \qquad \cX \mapsto \Bbbk^\cX:=\{ \text{functions } \cX \to \Bbbk  \},    \]
we let $\textbf{p}^\ast:=\Bbbk^{(-)} \circ \text{p}$. A \emph{vector species} $\textbf{p}$ is a presheaf of vector spaces on $\sS$, 
\[         \textbf{p}:   \sS^{\text{op} }\to \textsf{Vec}  , \qquad   I\mapsto  \textbf{p}[I]   .    \]
If $\text{p}$ is a set species, then we define an associated vector species $\textbf{p}$ by letting $\textbf{p}[I]\subseteq \Hom( \textbf{p}^\ast[I], \Bbbk )$ be the subspace of linear functionals corresponding to formal linear combinations of elements of $\text{p}[I]$, i.e. given the free vector space monoidal functor
\[   \Bbbk(-) :  \textsf{Set} \to  \textsf{Vec}, \qquad \cX\mapsto \Bbbk \cX ,  \]
we may naturally identify $\textbf{p}=\Bbbk(-) \circ \text{p}$. Finally, a \emph{set cospecies} $\text{q}$ is a copresheaf on $\sS$, 
\[         \text{q}:   \sS \to  \textsf{Set}  , \qquad   I\mapsto  \text{q}[I].  \]
If $\text{q}$ is a set cospecies, then we have the associated vector cospecies $\textbf{q}=\Bbbk(-) \circ \text{q}$. A set (resp. vector) (co)species is called \emph{positive} if its value on the empty set $\emptyset$ is the empty set (resp. zero-dimensional vector space), and \emph{connected} if its value on $\emptyset$ is the singleton set (resp. one-dimensional vector space).

Throughout this paper, we define (co)species by explicitly giving the values of functors on finite sets only, with their values on bijections then being induced in an obvious way. 




\begin{remark}
In \cite{aguiar2010monoidal}, \cite{aguiar2013hopf}, Aguiar and Mahajan explicitly consider only copresheaves on $\sS$. However, since the inversion of bijections is a dagger for $\sS$ (it is a groupoid), the general theory of species is a copy of the general theory of cospecies. 
\end{remark}

\begin{figure}[t]
	\centering
	\includegraphics[scale=0.9]{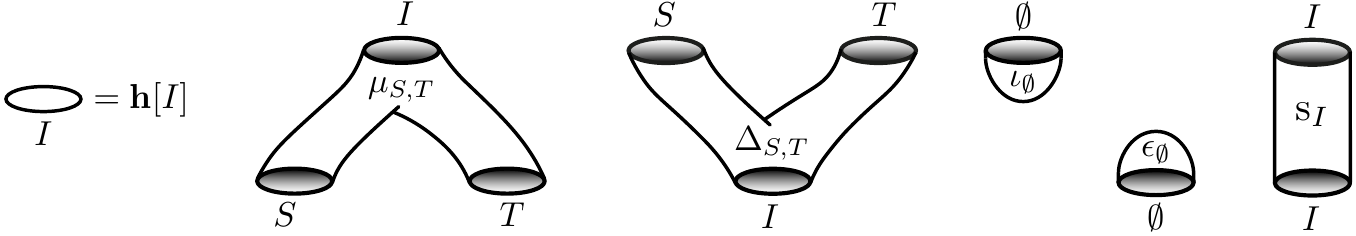}
	\caption{Data of a Hopf algebra in (co)species $\textbf{h}$. The topology here is just serving as notation for string diagrams.}
	\label{fig:data}
\end{figure} 

The functor category of set (resp. vector) (co)species is a symmetric monoidal category when equipped with the Day convolution for the disjoint union of finite sets, denoted $\times_\text{Day}$ (resp. $\otimes_\text{Day}$). Let a \emph{(co/bi/Hopf)algebra in (co)species} be a (co/bi/Hopf)monoid internal to vector (co)species. In \cite[Section 4.2-4.3]{aguiar2013hopf}, the specialized notion of a set-theoretic (co/bi/Hopf)monoid is given, which is required in order to overcome a technicality when internalizing comonoids in set (co)species. Let a \emph{(co/bi/Hopf)monoid in (co)species} be a set-theoretic (co/bi/Hopf)monoid in set (co)species.

Let us make this explicit for the case of a Hopf algebra in (co)species, which is the most relevant to us. A bialgebra in (co)species consists of a vector (co)species $\textbf{h}$ such that for each finite set $I\in \textsf{S}$, and each choice of disjoint subsets $S,T\subseteq I$ with $S\sqcup T=I$,\footnote{\ in particular, we allow $S$ or $T$ to be empty} we have linear maps called the multiplication and comultiplication, respectively
\[   \mu_{S,T}:  \textbf{h}[S]\otimes \textbf{h}[T]\to \textbf{h}[I] \qquad \text{and} \qquad \Delta_{S,T}: \textbf{h}[I]\to \textbf{h}[S]\otimes \textbf{h}[T],  \]
and linear maps called the unit and counit, respectively
\[  \iota_{\emptyset}  :\Bbbk\to  \textbf{h}[\emptyset] \qquad \text{and}  \qquad  \epsilon_{\emptyset}:  \textbf{h}[\emptyset]\to \Bbbk . \]
A Hopf algebra in (co)species must have an additional linear map for each finite set $I\in \textsf{S}$, called the antipode,
\[ \text{s}_I:  \textbf{h}[I]\to  \textbf{h}[I]. \]
See \autoref{fig:data}. All these maps should satisfy the usual bimonoid and Hopf monoid compatibility axioms. The compatibility of multiplication and comultiplication is the most interesting, whose string diagram is shown in \autoref{fig:exchange}. See \cite[Section 8.3]{aguiar2010monoidal} for more details. 

To define a bialgebra structure on a connected (co)species, it suffices to specify $\mu_{S,T}$ and $\Delta_{S,T}$ for $S,T\neq \emptyset$. Moreover, a bialgebra structure on a connected (co)species is necessarily a Hopf algebra, and there are explicit formulas for the antipode.
 
If $\text{p}$ is a cocommutative Hopf monoid in species with finite components $\text{p}[I]$, then $\textbf{p}^{\ast}=\Bbbk^{(-)} \circ \text{p}$ is naturally a commutative Hopf algebra in cospecies, and $\textbf{p}=\Bbbk(-) \circ \text{p}$ is naturally its dual cocommutative Hopf algebra in species \cite[Section 4.5]{aguiar2013hopf}. 

\begin{remark}
We may compare these definitions and constructions with more familiar Hopf monoids. Recall that groups, which are Hopf monoids internal to cartesian categories of spaces, are carried into ordinary commutative Hopf algebras via appropriate modifications of the functor $\Bbbk^{(-)}$ (e.g. smooth functions on a compact Lie group), and are carried into ordinary cocommutative Hopf algebras via appropriate modifications of the functor $\Bbbk(-)$ (e.g. distributions on an algebraic group), see \cite[Section 3]{MR2290769}. Note however that Hopf monoids in species are not groups; groups should be constructed with respect to the Hadamard product $\times_{\text{Had}}$ since this is the cartesian product for set species. In fact, set species trivially form a presheaf topos, and so groups in set species are equivalently presheaves of groups $\textsf{S}^{\text{op}}\to \textsf{Grp}$, an example from this paper being $I\mapsto \text{PGL}_I(\bC)$.
\end{remark}

\begin{figure}[t]
	\centering
	\includegraphics[scale=0.9]{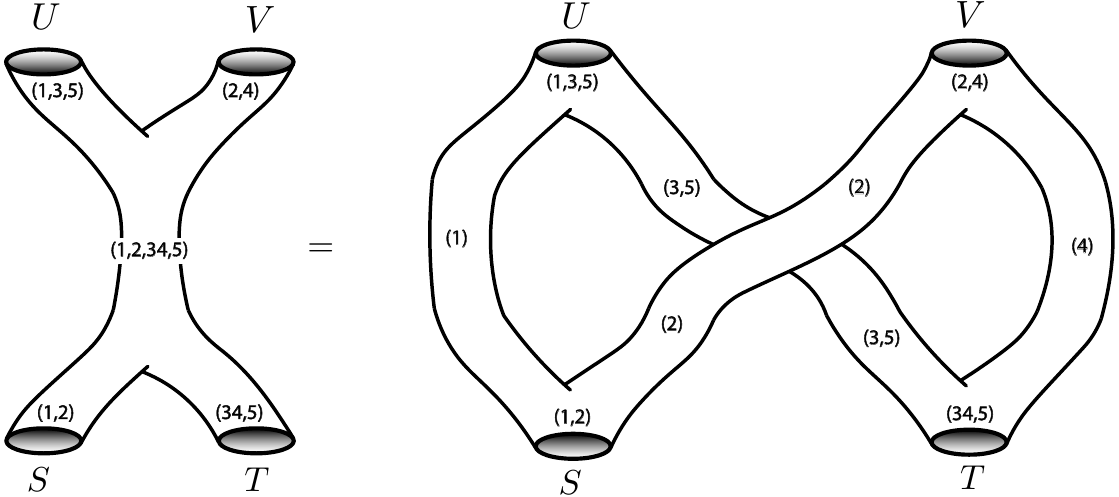}
	\caption{The string diagram for the compatibility of multiplication and comultiplication, where $I=S\sqcup T=U\sqcup V$. One can check compatibility by imagining passing elements of the Hopf monoid through the surfaces. An example for the Hopf monoid of set compositions $\Sigma$ (see \autoref{comp}) is shown, where $S=\{1,2\}$, $T=\{3,4,5\}$, $U=\{1,3,5\}$, $V=\{2,4\}$.}
	\label{fig:exchange}
\end{figure}


\begin{remark}
Set species, or more generally presheaves on $\textsf{S}$ valued in a cartesian category of spaces $\textsf{Spa}$, appear naturally as moduli spaces of spaces with marked points. This is analogous to the construction of singular simplicial complexes in homology, where instead of probing a topological space $\cX\in \textsf{Top}$ with simplices by composing the standard cosimplicial topological space\footnote{\ recall that a (co)simplicial topological space is a $\textsf{Top}$-valued (co)presheaf on the simplex category $\Delta$}
\[\Delta\to \textsf{Top} \]
with $\Hom_{\textsf{Top}}(-,\cX)$, one probes a space $\cX\in \textsf{Spa}$ with finite sets by composing the `standard' cospecies of finite discrete spaces 
\[\textsf{S}\to \textsf{Spa}\]
with $\Hom_{\textsf{Spa}}(-,\cX)$. Such a set species will automatically be a set-theoretic comonoid by restricting configurations of points (`forgetting marked points'), which is a special case of Schmitt's comonoid construction \cite[Section 8.7.8]{aguiar2010monoidal}. Dually, we may obtain cospecies by composing the cospecies of finite codiscrete spaces $\textsf{S}\to \textsf{Spa}$ with $\Hom_{\textsf{Spa}}(\cX,-)$.
\end{remark}

\subsection{Set Compositions and Set Partitions} \label{comp}
We now describe some primordial examples of set species, which, according to the philosophy of Rota's twelvefold way, should be defined as certain probes of ordinals and cardinals by finite sets. This approach also enables us to immediately suggest a generalization of the construction in \autoref{complexperm}.

Let us model finite ordinals by ordered tuples of integers
\[   
(k):= \{1,\dots,k\}, \qquad k\in \bN
\]
where the ordering is $1>\cdots>k$. Let $I\in \sS$ be a finite set with cardinality $n\in \bN$. Let $\widehat{\Sigma}[I]$ be the set of finite ordinals with $I$-marked points, that is
\[  
\widehat{\Sigma}[I]:= \big \{  \text{functions}\ I \xrightarrow{F} (k) : k\in \bN  \big \} =\bigsqcup_{k\in \bN} (k)^I 
.\] 
This defines the species of (set) \emph{decompositions} $\widehat{\Sigma}$. We can encode decompositions $F\in \widehat{\Sigma}[I]$ in formal expressions
\[  
F=  (S_1, \dots, S_k) , \qquad \text{where}\ \ S_j:= F^{-1}(j)
.\]
The commas in the expression indicate that the order does matter. The $S_j$ are called the \emph{lumps} of $F$. We let $l(F)=k$ denote the number of lumps of $F$. The \emph{opposite} of $F$ is defined by 
\[
\bar{F}:=(S_k,\dots, S_1), \qquad  \text{i.e.} \quad \bar{F}^{-1}(j)=F^{-1}(k+1-j)
.\] 
We have the following subsets of $\widehat{\Sigma}[I]$,
\[  
\Sigma[I]:= \big \{  \text{surjective functions}\ I \xrightarrow{F} (k) : k\in \bN  \big \}  
\]
and
\[  
\text{L}[I]:= \big \{  \text{bijective functions}\ I \xrightarrow{F} (k) : k\in \bN  \big \}  
.\]
This defines the species of (set) \emph{compositions} $\Sigma$ and \emph{linear orders} $\text{L}$. 
For $F$ a decomposition, let $F_+$ denote the composition whose formal expression is obtained from the formal expression of $F$ by deleting all copies of the empty set. This defines a retraction of species,
\[     
(-)_+:  \widehat{\Sigma} \twoheadrightarrow \Sigma, \qquad F\mapsto F_+
.\] 
Notice that $\widehat{\Sigma}$ is immediately a set-theoretic comonoid via Schmitt's comonoid construction, with the comultiplication given by restricting configurations of points. We can define a multiplication on $\widehat{\Sigma}$ by concatenating formal expressions for decompositions. More conceptually, this uses the fact that unmarked ordinals $\widehat{\Sigma}[\emptyset]$ form the additive monoid $(\bN,+)$ via ordinal sum,\footnote{\ the multiplication of $\bN$ is recovered as the Tits product \cite[Section 1.11]{aguiar2013hopf}}
\[ 
\widehat{\Sigma}[\emptyset]\times \widehat{\Sigma}[\emptyset]\to \widehat{\Sigma}[\emptyset]
,\qquad      
(k_1,k_2)\mapsto   (k_1)+(k_2):=(k_1+k_2)
.\] 
We can then define the map
\[   \iota_{k_1,k_2}:  (k_1)\sqcup (k_2)\to (k_1+k_2), \qquad    \iota_{k_1,k_2}(j):= \begin{cases}
j &\quad  \text{if}\quad   j\in (k_1)     \\
 j+k_1   &\quad \text{if}\quad   j\in (k_2)
\end{cases}    \] 
which allows us to glue configurations; for $F:S\to (k_1)$ and $G:T\to (k_2)$, the \emph{concatenation} of $F$ with $G$ is given by
\[   FG:I\to (k_1+k_2), \qquad     FG:=  \iota_{k_1,k_2} \circ (F\sqcup G) . \]
For $S\subseteq I$ and $F:I\to (k)$, let $F|_S$ denote the restriction of $F$ to $S$, 
\[         F|_S:  S\to (k) , \qquad   F|_S:= F \circ  (S \hookrightarrow I)   . \]
Then $\widehat{\Sigma}$ is a bimonoid in species, with multiplication concatenation and comultiplication restriction,
\[         \mu_{S,T} ( F, G ): =  FG \qquad \text{and} \qquad  \Delta_{S,T} (F) :=  (F|_S,  F|_T)    .     \]
This induces the structure of a Hopf monoid on $\Sigma$ (considered as a quotient of $\widehat{\Sigma}$), which restricts to the structure of a Hopf monoid on $\text{L}$. Explicitly, the multiplication is concatenation, and the comultiplication is restriction composed with $(-)_+$,
\[         \mu_{S,T} ( F_+, G_+ ): =  F_+ G_+ \qquad \text{and} \qquad  \Delta_{S,T} (F_+) :=  \big (  (F_+ |_S)_+,  (F_+ |_T)_+ \big )    .     \]
Note that compositions and linear orders are not closed under restriction. From now on, if $F$ is a composition, we let $F|_S:=(F|_S)_+$.

\begin{remark}
Let us briefly mention the connection to more familiar decategorified structures. Recall the bosonic Fock functor $\overline{\cK}_{V}(-)$ from the \hyperref[intro]{Introduction}, which preserves Hopf monoids. We have that $\overline{\cK}_V(\textbf{L})$ is the tensor algebra $\cT(V)$ on $V$, and $\overline{\cK}(\Sig)=\overline{\cK}_\Bbbk(\Sig)$ is the algebra of noncommutative symmetric functions $\textbf{NSym}$. We describe $\Sig$ in detail in \autoref{sec:thehopfalgebras}.
\end{remark}

Let us model finite cardinals by sets of integers,
\[[k]:=\{ 1 ,\dots, k\},  \qquad k\in \bN.\] 
The automorphism group of $[k]$ is the symmetric group $\text{Sym}_k$ of $\{ 1 ,\dots, k\}$, whereas $(k)$ did not have any automorphisms. Let $\widehat{\Pi}[I]$ be the set of finite cardinals with $I$-marked points, taken modulo automorphisms,
\[  \widehat{\Pi}[I]:= \Big \{  \bigslant{\text{functions}\ I \xrightarrow{P} [k]}{\text{Sym}_k}: k\in \bN \Big \}    . \] 
This defines the set species of \emph{departitions} $\widehat{\Pi}$. We can encode departitions $P\in \widehat{\Pi}[I]$ in formal expressions
\[ P=  (S_1| \dots| S_k) , \qquad \text{where}\ \ S_j:= P^{-1}(j).      \]
The bars in the expression indicate that the order does not matter. The $S_j$ are called the \emph{blocks} of $P$. We have the subsets
\[  \Pi[I]:= \Big \{  \bigslant{\text{surjective functions}\ I \xrightarrow{P} [k]}{\text{Sym}_k}: k\in \bN \Big \}    \]
and
 \[  \text{E}[I]:= \Big \{  \bigslant{\text{bijective functions}\ I \xrightarrow{P} [k]}{\text{Sym}_k}: k\in \bN \Big \}.  \]
This defines the set species of \emph{partitions} $\Pi$, and the \emph{exponential species} $\text{E}$. Notice that each component of the exponential species has cardinality one. The map 
\[  
Q_{(-)}: \widehat{\Sigma}\twoheadrightarrow \widehat{\Pi}, \qquad  (S_1,\dots, S_k)\mapsto  (S_1|\dots| S_k)   
\]
which quotients by the $\text{Sym}_k$ actions induces the structure of a bimonoid on $\widehat{\Pi}$. Then $\Pi$ and $\text{E}$ inherit the structure of Hopf monoids from $\widehat{\Pi}$. We summarize all these bimonoids in the following commutative diagram of homomorphisms.
\begin{center}
\begin{tikzcd}[column sep=large,row sep=large] 
  \widehat{\Sigma} 	\arrow[d, twoheadrightarrow, "\text{\textcolor{white}{mod cardinal symmetry}}"']  	\arrow[r, twoheadrightarrow, "\text{mod $\emptyset$}"]    & \Sigma	\arrow[d, twoheadrightarrow]  \arrow[r, hookleftarrow, "\text{bijective}", "\text{probes}"']  	&	\arrow[d, twoheadrightarrow, "\text{mod cardinal symmetry}"] 	\text{L} 	\\
     \widehat{\Pi}    	\arrow[r, twoheadrightarrow]  &  \Pi 		 \arrow[r, hookleftarrow] &  	\text{E} 
\end{tikzcd}
\end{center}
\noindent This diagram is fundamental in the recent work of Aguiar and Mahajan, who use it as a pivot towards generalizations of Hopf monoids in species \cite{aguiar2020bimonoids}.  

\begin{remark}
We have that $\overline{\cK}_V(\textbf{E})$ is the symmetric algebra $\cS(V)$ on $V$, and $\overline{\cK}(\boldsymbol{\Pi})=\overline{\cK}_\Bbbk(\boldsymbol{\Pi})$ is the algebra of symmetric functions $\textbf{Sym}$.
\end{remark}

For partitions $P,Q\in \Pi[I]$, if each block of $Q$ is a subset of a block of $P$, we write 
\[
P\leq Q
.\] 
That is we can obtain $P$ from $Q$ by merging blocks. For compositions $F,G\in \Sigma[I]$, we write   
\[
G\leq F
\]
if we can obtain $G$ from $F$ by merging contiguous lumps. Given compositions $G\leq F$ with $G=(S_1, \dots, S_k)$, let
\[l(F/G):=\prod^{k}_{j=1} l( F|_{S_j} )     \qquad \text{and} \qquad  (F/G)!:=\prod^{k}_{j=1} l( F|_{S_j} )!\,    . \]

\begin{remark}\label{complexperm}
In order to obtain the Hopf monoid structure on complex permutohedral space $\bC\Sigma$, mentioned during the discussion of the Losev-Manin moduli space in the \hyperref[intro]{Introduction}, one should now repeat this construction with the ordinals $(k)$ replaced by strings of $k$-many Riemann spheres $\bC\text{P}^1$, glued at the poles as in \cite[Section 3]{batyrevblume10}. 
\end{remark}

\subsection{Preposets} \label{preposet}
We describe one of the most important examples of a cospecies, which we present in terms of codiscrete categories on finite sets. For related examples of cospecies, see \cite[Chapter 13]{aguiar2010monoidal}, \cite{aguiar2017hopf}. For the relationship with the braid arrangement, see \cite[Section 13.5]{aguiar2010monoidal}.

Let $\textsf{cat}$ denote the category of small categories.\footnote{\ in particular, categories $\cC\in \textsf{cat}$ have the structure of a set on their preset of objects} For $\cC\in \textsf{cat}$ and objects $x,y\in \cC$, we write $x\geq y$ if there is a morphism $y\rightarrow x$.\footnote{\ we may think of $y\rightarrow x$ as the arrow of time, then $x\geq y$ corresponds to the time coordinate of $x$ being greater than the time coordinate of $y$} The \emph{interval category} $\textbf{2}\in \textsf{cat}$ has set of objects $\{1,2\}$, with a single non-identity morphism $2\rightarrow  1$. Given a finite set $I\in \sS$, the \emph{codiscrete category} $\text{Codisc}[I]\in \textsf{cat}$ has set of objects $I$, with exactly one morphism $i_2\rightarrow i_1$ for all $i_1,i_2\in I$. We define the set cospecies $[\textbf{2};-]$ by
\[ 
[\textbf{2};I]:= \big\{\text{injective functions }\textbf{2}\hookrightarrow I \big\}  \subset \Hom_{\textsf{cat}}\big(\textbf{2},\text{Codisc}[I]\big ) 
.\]
Then $[\textbf{2};I]$ is in natural bijection with the set of non-identity morphisms of $\text{Codisc}[I]$. For $i_1,i_2\in I$ with $i_1\neq i_2$, let $(i_1,i_2)\in [\textbf{2};I]$ correspond to the morphism $i_2\to i_1$, that is
\[
(i_1,i_2)(1):=i_1  
\qquad \text{and} \qquad  
(i_1,i_2)(2):=i_2
.\] 
Then morphism composition in $\text{Codisc}[I]$ induces the following partially defined product on $[\textbf{2};I]$,
\[
(i_1,i_2)\circ (i_3,i_4):=  
\begin{cases}
  (i_1,i_4)  &\quad  \text{if}\ i_2=i_3\ \text{and} \ i_1\neq i_4        \\
 (i_3,i_2)   &\quad \text{if}\  i_1=i_4\ \text{and} \ i_2\neq i_3 \\
\text{undefined}   &\quad \text{otherwise}.
\end{cases}
\]
Recall that (co)roots of $\text{SL}_I(\bC)$ are indexed by elements of $[\textbf{2};I]$. This partial product corresponds to the addition of (co)roots, see \autoref{sec:root}. Underlying this is the fact that $\text{SL}_I(\bC)$ is a group of units in the convolution algebra of $\text{Codisc}[I]$.   

A \emph{preposet} $p$ of $I$ is a reflexive and transitive relation $\geq_p$ on $I$, or equivalently a \hbox{surjective-on-objects} subcategory of $\text{Codisc}[I]$. In this paper, we identify $p$ with its `injective $\textbf{2}$-probes', 
\[            
p= [\textbf{2};p] :=  \big \{ (i_1,i_2)\in [\textbf{2};I]:  i_1\geq_p i_2   \big  \}\subset  \Hom_{\textsf{cat}  }\big ( \textbf{2}, p \big )
.   \]
Thus, we identify $p$ with its set of non-identity morphisms. This identifies preposets with subsets of $[\textbf{2};I]$ which are closed under the partial product, and hence also with subsets of coroots which are closed under addition.

An element $(i_1,i_2)\in p$ is called \emph{symmetric} if $(i_2,i_1)\in p$. We let $p_>$ denote the set nonsymmetric elements of $p$. The \emph{opposite} $\bar{p}$ of $p$ is defined by 
\[(i_1,i_2)\in \bar{p} \qquad  \text{if and only if} \qquad (i_2,i_1)\in p.\] 
The intersection of two preposets of $I$ is a preposet of $I$. For $X\subseteq [\textbf{2};I]$ any subset, the \emph{transitive closure} $\text{Cl}(X)$ of $X$ is the preposet of $I$ which is the intersection of all the preposets of $I$ which contain $X$. For $p$ and $q$ preposets of $I$, let $p\cup q$ denote the preposet of $I$ which is the transitive closure of the set union of $p$ and $q$. 

The \emph{blocks} of a preposet $p$ are the equivalence classes of the transitive and symmetric closure of the relation
\[i_1 \sim i_2\qquad \iff \qquad  (i_1,i_2)\in p \quad \text{or}\quad (i_2,i_1)\in p .\] 
The \emph{lumps} of $p$ are the equivalence classes of the equivalence relation
\[
i_1 \sim i_2\qquad \iff \qquad  (i_1,i_2)\in p \quad \text{and}\quad (i_2,i_1)\in p 
.\] 
Let $l(p)$ denote the number of lumps of $p$. For $p$ and $q$ preposets of $I$, we write $q\leq  p$ if $p \subseteq q$. We write $q \preceq    p$ if both $q\leq p$ and $p_>\subseteq q_>$, and we write $q \preceq_{l} p$ if both $q \preceq   p$ and $l(q)=l(p)$. For $p$ a preposet of $I$ and $S\subseteq I$, the \emph{restriction} $p|_S$ is the preposet of $S$ which is given by 
\[  
(i_1,i_2)\in  p|_S  \qquad  \text{if and only if} \qquad(i_1,i_2)\in p \quad \text{for all}\quad (i_1,i_2)\in [\textbf{2};S]
.            \]
Let
\[      \text{O}[I]:= \big\{  p: \text{$p$ is a preposet of $I$}    \big \} .             \]
This defines the cospecies of preposets $\text{O}$. We say a preposet is \emph{total} if for all $(i_1,i_2)\in [\textbf{2};I]$, at least one of $(i_1,i_2)\in p$ and $(i_2,i_1)\in p$ is true. Let $\Sigma^\ast$ denote the cospecies of total preposets, 
\[   \Sigma^\ast[I]  := \big \{ p\in \text{O}[I] : \text{$p$ is total}    \big  \} . \]
We say a preposet is \emph{totally-nonsymmetric} if for all $(i_1,i_2)\in [\textbf{2};I]$, exactly one of $(i_1,i_2)\in p$ and $(i_2,i_1)\in p$ is true. Let $\text{L}^\ast$ denote the cospecies of totally-nonsymmetric preposets, 
\[ \text{L}^\ast[I]  := \big \{ p\in \text{O}[I]: \text{$p$ is totally-nonsymmetric}     \big  \}  .     \] 
Let $\Pi^\ast$ denote the cospecies of preposets without nonsymmetric elements,
\[          \Pi^\ast[I]:=      \big \{ p\in \text{O}[I]   :  p_{>}=\emptyset \big  \}= \big \{ p\in \text{O}[I]   :  \bar{p}=p \big  \}   .     \]
The elements of the cospecies $\Sigma^\ast$, $\text{L}^\ast$, and $\Pi^\ast$ are in one-to-one correspondence with set compositions, linear orders, and set partitions respectively. Explicitly, given a partition $P=(S_1|\dots|S_k)$ of $I$, we let $P\in \Pi^\ast[I]$ denote the encoding of $P$ as the collection of $(i_1,i_2)\in [\textbf{2};I]$ such that $\{i_1,i_2\}\subseteq S_j$ for some $S_j\in P$; and given a composition $F=(S_1,\dots,S_k)$ of $I$, we let $F\in \Sigma^\ast[I]$ denote the encoding of $F$ as the collection of $(i_1,i_2)\in [\textbf{2};I]$ such that the lump of $F=(S_1,\dots,S_k)$ containing $i_1$ is to the left of, or is equal to, the lump containing $i_2$. 

The context will make it clear whether we mean the covariant or contravariant version of compositions and partitions. For example, given a two-lump composition $(S,T)$, we might write `$(S,T)\leq p$', which is only defined if we take $(S,T)\in \Sigma^\ast[I]\subseteq \text{O}[I]$. Notice this says that $S$ (resp. $T$) is an upward (resp. downward) closed subset of $p$.

\begin{remark}
Note that $\text{O}$ indexes cones of the braid arrangement, $\Sigma^\ast$ indexes faces, $\text{L}^\ast$ indexes chambers, and $\Pi^\ast$ indexes flats \cite[Section 3]{vic}, \cite[Section 13.5.1]{aguiar2010monoidal}. See \autoref{braid} for details. 
\end{remark}

\subsection{Adjoint Families} \label{ref:adfam}
Let us now describe the adjoint analogs of preposets, which we call \mbox{(pre-)adjoint} families. Similar combinatorial gadgets were considered by Epstein-Glaser-Stora \cite{epstein1976general}. We define the set species $[-;\textbf{2}]$ by
\[ 
[I;\textbf{2}]:=\big\{\text{surjective functions } I\twoheadrightarrow \textbf{2}\big\}\subset \Hom_{\textsf{cat}  }\big ( \text{Disc}[I], \textbf{2} \big )
.\] 
If we view $\textbf{2}$ as the ordinal with two elements, then $[I;\textbf{2}]$ is equivalently the set of compositions of $I$ with two lumps. Explicitly, for $S\sqcup T=I$ with $S,T\neq \emptyset$, define $(S,T)\in [I;\textbf{2}]$ by
\[(S,T)(i)=1\quad\   \text{if }\  i\in S \qquad \text{and} \qquad  (S,T)(i)=2\quad\   \text{if }\  i\in T.\]
Define the following partial product on $[I;\textbf{2}]$,
\[
(S,T)\circ (U,V):=  
\begin{cases}
   (S\cup U,T\cap V)   &\quad  \text{if}\ T\supset U      \\
\big (S\cap U,T\cup V  \big)    &\quad \text{if}\  S\supset V \\
\text{undefined}   &\quad \text{otherwise}.
\end{cases}
\]
Recall that fundamental weights of $\text{SL}_I(\bC)$ are indexed by elements of $[I;\textbf{2}]$. This partial product corresponds to the addition of fundamental weights, see \autoref{sec:root}. 

Let a \emph{pre-adjoint family} $\cF$ of $I$ be a subset of $[I;\textbf{2}]$ which is closed under this partial product, i.e. for all $(S,T),(U,V)\in \cF$, we have either 
\[
(S,T)\circ (U,V) \in \cF \qquad \text{or}\qquad (S,T)\circ(U,V)\quad \text{is undefined.}
\] 
Let an \emph{adjoint family} $\cF$ of $I$ be a subset of $[I;\textbf{2}]$ whose corresponding set of fundamental weights of $\text{SL}_I(\bC)$ is closed under taking non-negative linear combinations. The intersection of two adjoint families of $I$ is an adjoint family of $I$. For $X\subseteq [I;\textbf{2}]$ any subset, the \emph{closure} $\text{Cl}(X)$ of $X$ is the adjoint family of $I$ which is the intersection of all the adjoint families of $I$ which contain $X$. 

Given a preposet $p\in \text{O}[I]$, we obtain an associated subset $\cF_p\subseteq [I;\textbf{2}]$ by taking the `surjective $\textbf{2}$-coprobes' of $p$,
\[\cF_p=  [p;\textbf{2}] := \big\{ (S,T)\in[I;\textbf{2}]: (S,T)\leq p    \big\} \subset  \Hom_{\textsf{cat}  }\big (p, \textbf{2}\big) . \] 
This unforgetfully encodes $p$ in terms of its upward/downward closed subsets. (We have $\text{Cl}(\cF_p)=\cF_p$ by \autoref{mainthm1}, and so $\cF_p$ is in fact an adjoint family.) If $F\in \Sigma^\ast[I]$, then $\cF_F$ recovers the usual way of modeling set compositions as flags of subsets.  

\begin{remark} \label{marcello}
In general, the set of pre-adjoint families of $I$ does not coincide with the set of adjoint families of $I$. Indeed, there exists a nontrivial pre-adjoint family $\cF$, which is additionally totally-nonsymmetric in the sense defined below, with $\text{Cl}(\cF)=[I;\textbf{2}]$ \cite[p.92-93]{epstein1976general}.\footnote{\ We thank Marcelo Aguiar for making us aware of this result.}
\end{remark}

An element $(S,T)\in \cF$ is called \emph{symmetric} if $(T,S)\in \cF$.\footnote{\ a pre-adjoint family without symmetric elements is (equivalent to) a \emph{paracell} in the sense of Epstein-Glaser-Stora \cite[Definition 1]{epstein1976general}, \cite[Definition 2.3]{epstein2016}} We let $\cF_>$ denote the set of nonsymmetric elements of $\cF$. The \emph{opposite} $\bar{\cF}$ of $\cF$ is defined by 
\[(S,T)\in \bar{\cF} \qquad  \text{if and only if} \qquad (T,S)\in \cF.\] 
Then $\bar{\cF}_p=\cF_{\bar{p}}$. For $\cF$ and $\cG$ pre-adjoint families of $I$, we write $\cG \leq  \cF$ if $\cF \subseteq \cG$. We have $q\leq p\iff \cF_p\leq \cF_q$. 


Let
\[      \text{O}^\vee[I]:= \big\{  \cF: \text{$\cF$ is an adjoint family of $I$}    \big \} .             \]
This defines the species of adjoint families $\text{O}^\vee$. An adjoint family $\cF$ is \emph{total} if for all $(S,T)\in [I;\textbf{2}]$, at least one of $(S,T)\in \cF$ and $(T,S)\in \cF$ is true. We denote total adjoint families by $\cS$. Let $\Sigma^\vee$ denote the species of total adjoint families, 
\[     
 \Sigma^\vee[I] := \big \{ \cF\in \text{O}^\vee[I]   :\text{$\cF$ is total} \big  \}
. \]
An adjoint family $\cF$ is \emph{totally-nonsymmetric} if for all $(S,T)\in [I;\textbf{2}]$, exactly one of $(S,T)\in \cF$ and $(T,S)\in \cF$ is true. Totally-nonsymmetric adjoint families are (equivalent to) maximal unbalanced families \cite{billera2012maximal}, positive sum systems \cite{MR3467341}, and cells\cite[Definition 6]{epstein1976general}, \cite[Definition 2.5]{epstein2016}. The number of maximal unbalanced families is sequence \href{https://oeis.org/A034997}{A034997} in the OEIS. Let $\text{L}^\vee$ denote the species of totally-nonsymmetric adjoint families, 
\[       \text{L}^\vee[I] := \big \{ \cF\in \text{O}^\vee[I] :\text{$\cF$ is totally-nonsymmetric}    \big  \}.  \] 
Let $\Pi^\vee$ denote the species of adjoint families without nonsymmetric elements, 
\[          \Pi^\vee[I]:=      \big \{ \cF\in \text{O}^\vee[I]   :  \cF_{>}=\emptyset \big  \}= \big \{ \cF\in \text{O}^\vee[I]   :  \bar{\cF}=\cF \big  \}.     \]

\begin{remark}
Since taking dual cones intertwines intersections with Minkowski sums, $\text{O}^\vee$ indexes cones of the adjoint braid arrangement, $\Sigma^\vee$ indexes faces, $\text{L}^\vee$ indexes chambers, and $\Pi^\vee$ indexes flats. See \autoref{adjoint} for details. 
\end{remark}

\section{Algebraic Structures} \label{alg}
\noindent We define the various algebras in (co)species which feature in this paper. Our main references are \cite{aguiar2010monoidal}, \cite{aguiar2013hopf}. We also prove some additional results which we shall need.

\subsection{The Commutative Hopf Algebra of Preposets} 
We now define the commutative Hopf algebra of preposets, following \cite[Section 13.1.6]{aguiar2010monoidal}. Let 
\[   \textbf{O}[I]:= \Bbbk \text{O}[I] =\{ \text{formal $\Bbbk$-linear combinations of preposets over $I$}\}  .   \]
This defines the vector cospecies $\textbf{O}$. Given $p\in \text{O}[I]$, we also denote by $p$ the corresponding basis element of $\bO[I]$. For $(S,T)\in [I;\textbf{2}]$, $p\in \text{O}[S]$ and $q\in \text{O}[T]$, let 
\[(p\, |\, q)\in \text{O}[I]\] 
denote the (disjoint) set union of $p$ and $q$. For $(S,T)\in [I;\textbf{2}]$ and $p\in \text{O}[I]$, we have (generalized) \emph{deconcatenation}, 
\[
p\talloblong_S:=  
\begin{cases}
p|_S &\quad  \text{if $(S,T)\leq p$}\\
0 &\quad \text{otherwise}
\end{cases}\qquad \text{and} \qquad 
p\! \fatslash_{\, T}:=  
\begin{cases}
p|_T &\quad \text{if $(S,T)\leq p$}\\
0 &\quad \text{otherwise.}
\end{cases}
\]
The vector cospecies $\bO$ has the structure of a commutative bialgebra, with multiplication union and comultiplication deconcatenation,
\[  \mu_{S,T} (  p \otimes  q ): =  (p\, |\, q) \qquad \text{and} \qquad    \Delta_{S,T}  ( p) :=   p\! \talloblong_S \otimes \,  p\! \!  \fatslash_{\, T}  .  \]
The unit and counit are induced by identifying $1\in \Bbbk$ with the unique preposet on the empty set. Since $\textbf{O}$ is then a connected bialgebra, it is a Hopf algebra. The antipode may be given by Takeuchi's antipode formula \cite[Proposition 8.13]{aguiar2010monoidal}. 

\begin{remark}
The Hopf algebra in cospecies $\textbf{O}$ is particular in that it does not arise as the algebra of functions on a Hopf monoid of preposets.
\end{remark}


\subsection{The Hopf Algebras of Set Compositions} \label{sec:thehopfalgebras}

We now define the Hopf algebras of set compositions $\Sig$ and $\Sig^\ast$, following \cite[Section 11]{aguiar2013hopf}. Let 
\[    \Sig^\ast[I]:= \big \{   \text{functions}\ \Sigma[I]\to \Bbbk  \big \} .        \]
This defines the vector cospecies $\Sig^\ast$. For $F\in \Sigma^\ast[I]$, let $\tM_F\in \Sig^\ast[I]$ be given by $\tM_F(G):=\delta_{FG}$, where $G\in \Sigma[I]$. The set 
\[  \big \{\tM_F: F\in \Sigma^\ast[I]\big \}\] 
is called the \emph{$\tM$-basis}, or \emph{monomial basis}, of $\Sig^\ast[I]$. The cocommutative Hopf monoidal structure of $\Sigma$ then induces a commutative Hopf algebraic structure on $\Sig^\ast$, given in terms of the $\tM$-basis by
\[\mu_{S,T} ( \tM_F\otimes \tM_G ): =   {\displaystyle \sum_{H\preceq (F|G)}} \tM_{ H  }\qquad \text{and} \qquad \Delta_{S,T}  (\tM_F): =  \tM_{F\talloblong_S} \otimes\,     \tM_{F\! \fatslash_{\, T}}.\footnote{\ we let $\tM_0:=0$}\] 
The antipode is given by
\[  \text{s}_I(\tM_F):=  (-1)^{l(F)} \sum_{G\leq \bar{F}}  \tM_G. \]
Let
\[   \Sig[I] := \Hom\big(\Sig^\ast[I], \Bbbk  \big )   .    \]
This defines the vector species $\Sig$. For $F\in \Sigma[I]$, let $\tH_F\in \Sig[I]$ be given by $\tH_F( \tM_G ):=\delta_{FG}$, where $G\in \Sigma^\ast[I]$. The set 
\[\big \{\tH_F: F\in \Sigma[I] \big \}\] 
is called the \emph{$\tH$-basis} of $\Sig[I]$. The cocommutative Hopf algebraic structure on $\Sig$ induced by linearizing the cocommutative Hopf monoidal structure of $\Sigma$ is given in terms of the $\tH$-basis by
\[         \mu_{S,T} ( \tH_F\otimes \tH_G ): =  \tH_{FG} \qquad \text{and} \qquad   \Delta_{S,T}  (\tH_F) :=  \tH_{F|_S} \otimes    \tH_{F|_T}      .  \]
The antipode is given by
\[ \text{s}_I(\tH_F):= \sum_{G\geq \bar{F}}  (-1)^{l(G)}\, \tH_G . \] 
We have the perfect pairing
\[    \Sig^\ast \otimes_{\text{Had} }  \Sig \to \textbf{E},  \qquad    \tM_F\otimes \tH_G \mapsto \delta_{FG}  .       \]
This realizes $\Sig$ as the dual Hopf algebra of $\Sig^\ast$ \cite[Section 8.6.2]{aguiar2010monoidal}. The Hopf algebra $\Sig$ is equivalently the free cocommutative Hopf algebra on the positive coalgebra $\textbf{E}^\ast_+$ \cite[Section 11.2.5]{aguiar2010monoidal}. This construction naturally equips $\Sig$ with the $\tH$-basis. A second basis of $\Sig$, called the \emph{$\tQ$-basis}, is given by  
   \[ \tQ_F:= \sum_{G\geq F}   (-1)^{ l(G)-l(F) }  \dfrac{1}{    l(G/F) }   \tH_G\qquad \text{or equivalently} \qquad  \mathtt{H}_F=: \sum_{G\geq F} \dfrac{1}{( G/F )!}  \tQ_G . \]
The $\tQ$-basis appears naturally if we construct $\Sig$ according to \cite[Section 11.2.2]{aguiar2010monoidal}, i.e. as the free cocommutative Hopf algebra on the raw species $\textbf{E}^\ast_+$, ignoring the coalgebra structure. For $(S,T)\in [I;\textbf{2}]$ and $F\in \Sigma[I]$, we have \emph{deshuffling},
\[F\res_S\, :=  
\begin{cases}
F|_S &\quad  \text{if $(S|T)\leq Q_F$\footnotemark}\\
0 &\quad \text{otherwise.}
\end{cases}
\]\footnotetext{\ recall that $Q_F$ denotes the partition with blocks the lumps of $F$}The algebraic structure of $\Sig$ is given in terms of the $\tQ$-basis by
\[         \mu_{S,T} ( \tQ_F\otimes \tQ_G ) =  \tQ_{FG} \qquad \text{and} \qquad  \Delta_{S,T}  (\tQ_F) =  \tQ_{F\res_S} \otimes \,   \tQ_{F\res_T}.  \]
We collect the combinatorial descriptions of the algebraic structure of $\Sig$ in the following table.
\bgroup
\def\arraystretch{1.1}%
\begin{table}[H] 
\begin{tabular}{|c|c|c|}
\hline
 $\Sig$ & $\tH$-basis & $\tQ$-basis   \\ \hline
multiplication   & concatenation     &  concatenation    \\ \hline
comultiplication &   restriction   &  deshuffling    \\ \hline
\end{tabular}
\end{table}
\egroup
\noindent Let the \emph{$\tP$-basis} be the basis of $\Sig^\ast$ which is dual to the $\tQ$-basis, thus
\[ \tP_F: = \sum_{G\leq F  }   \dfrac{1}{ (F/G)!  }  \,   \tM_G  .\]
The algebraic structure of $\Sig^\ast$ is given in terms of the $\tP$-basis by 
\[      \mu_{S,T} ( \tP_F\otimes \tP_G ) =  \sum_{H\preceq_{l}  (F|G)}  \tP_{ H }   \qquad   \text{and} \qquad \Delta_{S,T}  (\tP_F) =  \tP_{F\talloblong_S} \otimes    \tP_{F\! \fatslash_{\, T}}. \]
Consider the homomorphism of Hopf algebras given by
\[    \mathbf{O} \twoheadrightarrow \Sig^\ast, \qquad    p \mapsto \tC_p:= \sum_{ F\leq p } \tM_F    .  \]
To see that this is a homomorphism, one should consider the homomorphism $\mathbf{O} \to \Sig^\ast$ given by \cite[Theorem 11.23]{aguiar2010monoidal}, where the homomorphism of monoids is $\textbf{O}\to \textbf{E}$, $p\mapsto \tH_I$. Then \cite[Equation 11.18]{aguiar2010monoidal} recovers the map $p \mapsto \tC_p$. The set 
\[ \big \{\tC_F: F\in \Sigma^\ast[I]\big \} \] 
is a third basis of $\Sig^\ast[I]$. We call this the \emph{$\tC$-basis}, or \emph{cone basis}. 



\begin{prop}
The algebraic structure of $\Sig^\ast$ is given in terms of the $\tC$-basis by
\[   \mu_{S,T} ( \tC_F\otimes \tC_G ) =   \sum_{H\preceq (F|G)}  (-1)^{l(F|G)-l(H)}\,  \tC_{ H  } \qquad \text{and} \qquad  \Delta_{S,T}  (\tC_F) =  \tC_{F\talloblong_S} \otimes    \tC_{F\! \fatslash_{\, T}}   . \]
\end{prop}
\begin{proof}
Notice that
\[       \tC_{F}=  (-1)^{ l(F)} \text{s}_I(\tM_{\bar{F}})     .      \]
The antipode $\text{s}_I$ reverses the order of multiplication and comultiplication \cite[Proposition 1.22]{aguiar2010monoidal}. Therefore, for the multiplication, we have
\[   (-1)^{ l(F)+l(G)}     \mu_{S,T} ( \tC_F\otimes \tC_G ) =   \mu_{S,T} \big ( \text{s}_I(\tM_{\bar{F}}) \otimes \text{s}_I(\tM_{\bar{G}}) \big  )  = \sum_{ H\preceq (F|G)   } \text{s}_I(\tM_{ \bar{H}  })=\sum_{ H\preceq (F|G)   }  (-1)^{ l(H)}   \tC_{H} .     \]
Then
 \[       \mu_{S,T} ( \tC_F\otimes \tC_G ) =     \sum_{ H\preceq (F|G)   }  (-1)^{l(F)+l(G)+l(H)}   \tC_{H}= \sum_{H\preceq (F|G)}  (-1)^{l(F|G)-l(H)}\,  \tC_{ H  }   . \]
For the comultiplication, we have
\[  (-1)^{ l(F)}  \Delta_{S,T}  (\tC_F) = \Delta_{S,T}\big  (  \text{s}_I(\tM_{\bar{F}}) \big   )  =     \text{s}_I(\tM_{\bar{F}\talloblong_T}) \otimes  \text{s}_I(  \tM_{\bar{F}\! \fatslash_{\, S}} ) =   \text{s}_I(\tM_{\overline{F\! \fatslash_{\,T}}}) \otimes  \text{s}_I(  \tM_{\overline{F\talloblong_{S}}} )   \]
\[ =   (-1)^{ l(F\! \fatslash_{\,T})+l(F\talloblong_{S})}\,  \tC_{F\! \fatslash_{\,T}}\otimes  \tC_{F\talloblong_{S}}=(-1)^{ l(F)}\,  \tC_{F\! \fatslash_{\,T}}\otimes  \tC_{F\talloblong_{S}} .\qedhere \]
\end{proof}


Using the geometric realization of $\Sig^\ast$ over the braid arrangement (see \autoref{braid}), and an inclusion-exclusion argument, one sees that
\[\tC_{ p  }=   \sum_{ G\preceq p }  (-1)^{l(p)-l(G)}  \tC_G        .  \]
The case $p=(F_1|\dots |F_k)$, where the $F_j$'s are compositions, is proved explicitly in \cite[Theorem 21]{early2017canonical}. We collect the combinatorial descriptions of the algebraic structure of $\Sig^\ast$ in the following table.
\bgroup
\def\arraystretch{1.1}%
\begin{table}[H] 
\begin{tabular}{|c|c|c|c|}
\hline
 $\Sig^\ast$ & $\tM$-basis & $\tP$-basis &  $\tC$-basis   \\ \hline
multiplication &quasishuffling  & shuffling     &  signed-quasishuffling   \\ \hline
comultiplication &   deconcatenation   &   deconcatenation  & deconcatenation \\ \hline
\end{tabular}
\end{table}
\egroup

\subsection{The Primitive Part and the Indecomposable Quotient} \label{lie}

We now describe the Lie algebra which forms the primitive part of $\Sig$, and its dual Lie coalgebra, which is the indecomposable quotient of $\Sig^\ast$. Following Aguiar and Mahajan in \cite{aguiar2017topics}, the Lie algebra will be denoted by $\textbf{Zie}$. The Lie coalgebra will be denoted by $\textbf{Zie}^\ast$. Therefore, we shall have a pair of dual maps, denoted
\[    \cU: \textbf{Zie}\hookrightarrow    \Sig \qquad \text{and} \qquad  \cU^\ast: \Sig^\ast \twoheadrightarrow \textbf{Zie}^\ast.    \]
By the definition of primitive elements and the fact that $\Sig$ is connected, the image of $\textbf{Zie}$ is the kernel of the comultiplication of $\Sig$,
\[  \cU(\textbf{Zie})[I] =  \bigcap_{ (S,T)\in [ I;\textbf{2} ] }  \text{ker} \big  (    \Delta_{S,T} :\Sig[I]\to \Sig[S]\otimes \Sig[T]    \big)      .  \]
Dually, $\textbf{Zie}^\ast$ is the quotient of $\Sig^\ast$ by the image of its multiplication,
\[   \cU^\ast(\Sig^\ast)[I]  = \,  \bigslant{  \Sig^\ast[I]  }{  \displaystyle\sum_{(S,T)\in [ I;\textbf{2} ]}  \text{im}\big (\mu_{S,T}: \Sig^\ast[S]\otimes \Sig^\ast[T]\to \Sig^\ast[I]    \big)   }      .            \]
See \cite[Section 5.5-5.6]{aguiar2013hopf} for more details. 


Let a \emph{tree} $\mathcal{T}$ over $I$ be a planar\footnote{\ i.e. a choice of left and right child is made at every node} full binary tree whose leaves are labeled bijectively with the blocks of a partition $Q_\mathcal{T}$ of $I$. The blocks of $Q_\mathcal{T}$ are called the \emph{lumps} of $\mathcal{T}$. They form a composition $F_\mathcal{T}$ by listing in order of appearance from left to right, called the \emph{debracketing} of $\mathcal{T}$. We may denote trees by nested products `$[\ \cdot\ ,\ \cdot\ ]$' of subsets or trees, see \autoref{fig:tree}. We make the convention that no trees exist over the empty set $\emptyset$. 
\begin{figure}[H]
	\centering
	\includegraphics[scale=0.55]{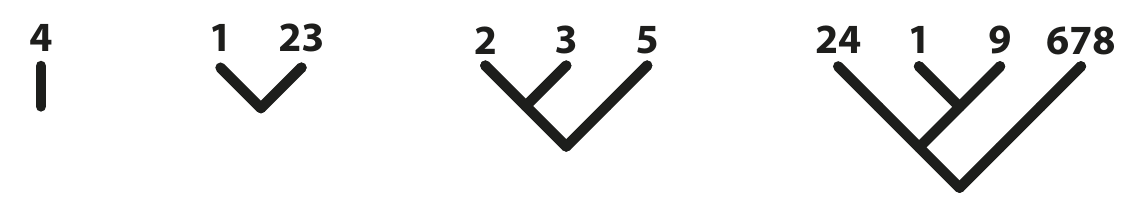}
	\caption{Let $I=\{1,2,3,4,5,6,7,8,9\}$. The trees $[4]$, $[1,23]$ ($\neq[23,1]$), $[[2,3],5]$, $[[24,[1,9]],678]$. The debracketing of $[[24,[1,9]],678]$ is the composition $(24,1,9,678)$. If we put $\mathcal{T}_1=[24,[1,9]]$ and $\mathcal{T}_2=[678]$, then $[\mathcal{T}_1, \mathcal{T}_2]$ would also denote this tree.}
	\label{fig:tree}
\end{figure}  
Given a tree $\mathcal{T}$, let $\text{antisym}(\mathcal{T})$ denote the set of $2^{l(F_{\mathcal{T}})-1}$ many trees which are obtained by switching left and right branches at nodes of $\mathcal{T}$. For $\mathcal{T}' \in \text{antisym}(\mathcal{T})$, let $(\mathcal{T}, \mathcal{T}')\in \bZ/2\bZ$ denote the parity of the number of node switches required to bring $\mathcal{T}$ to $\mathcal{T}'$. 


\begin{figure}[t]
	\centering
	\includegraphics[scale=0.51]{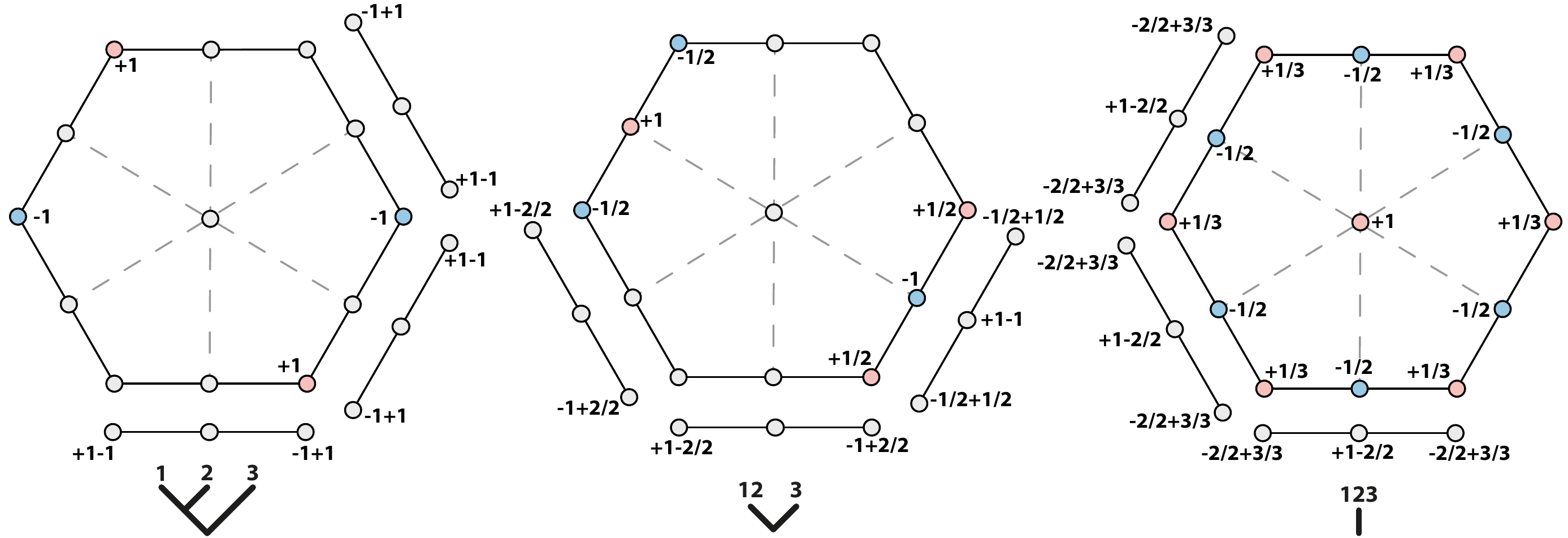}
	\caption{The primitive elements $\tQ_\mathcal{T}\in\Sig[I]$ corresponding to the trees $\mathcal{T}\in\big  \{  [[1,2],3], [12,3], [123]\big  \}$, where the $\tH$-basis $\Sigma$ has been identified with the midpoints of faces of the permutohedron (see the discussion of the Losev-Manin moduli space in the \hyperref[intro]{Introduction}). The (necessarily zero) comultiples of the primitive elements, i.e. their projections onto facets, are shown computed.}  
	\label{fig:lieele}
\end{figure} 

Let $\textbf{Zie}[I]$ be the vector space of formal linear combinations of trees over $I$, modulo antisymmetry and the Jacobi identity as interpreted on trees in the usual way. This defines the positive vector species $\textbf{Zie}$. If $\Lie$ is the species of the (positive) Lie operad, $\bE^\ast_+$ is the positive exponential species, and $\circ$ denotes the plethystic monoidal product of species, then   
\[
\textbf{Zie}[I]=    \Lie \circ \bE^\ast_+[I]= \bigoplus_{P\in \Pi[I]}   \Lie[P]  
.\footnote{\ as the argument of a (co)species, we let $P$ denote the set whose elements are the blocks of $P$}\] 
Define the following map, which embeds $\textbf{Zie}$ as the primitive part of $\Sig$,
\[     \cU:  \textbf{Zie} \hookrightarrow   \Sig, \qquad \mathcal{T} \mapsto   \tQ_\mathcal{T}  := \sum_{\mathcal{T}' \in \text{antisym}(\mathcal{T})}  (-1)^{ (\mathcal{T},\mathcal{T}') }  \tQ_{F_{\mathcal{T}'}}.      \]
See \autoref{fig:lieele}. The Lie bracket of $\textbf{Zie}$, which we denote by $\partial^\ast$, connects a pair of trees $\mathcal{T}_1$, $\mathcal{T}_2$ over disjoint sets by adding a new root whose children are the roots of $\mathcal{T}_1$ and $\mathcal{T}_2$,
\[ \partial^\ast: \textbf{Zie}\otimes_{\text{Day}} \textbf{Zie}\to \textbf{Zie}, \qquad    \partial_{S,T}^\ast(\mathcal{T}_1 \otimes \mathcal{T}_2) := [ \mathcal{T}_1 , \mathcal{T}_2]     .         \]
A geometric interpretation of this Lie bracket inside $\Sig$ can be seen in \autoref{fig:brack}. For $F\in \Sigma[I]$, let $[F]\in \textbf{Zie}[I]$ denote the `right comb-tree' with debracketing $F$; that is, if $F=(S_1,\dots, S_k)$, then 
\[     [F]:=  [ \dots [[ S_1  ,  S_2 ] , S_3 ], \dots  , S_k ]  . \]
Given a distinguished element $i_0\in I$, let 
\[\Sigma_{i_0}[I]  :=\big  \{ (S_1, \dots, S_k)\in \Sigma[I]:  i_0\in S_1   \big\}      .  \] 
If $F\in \Sigma_{i_0}[I]$, then $[F]$ is called a \emph{standard right comb-tree} over $I$ with respect to $i_0$.  

\begin{figure}[t]
	\centering
	\includegraphics[scale=0.7]{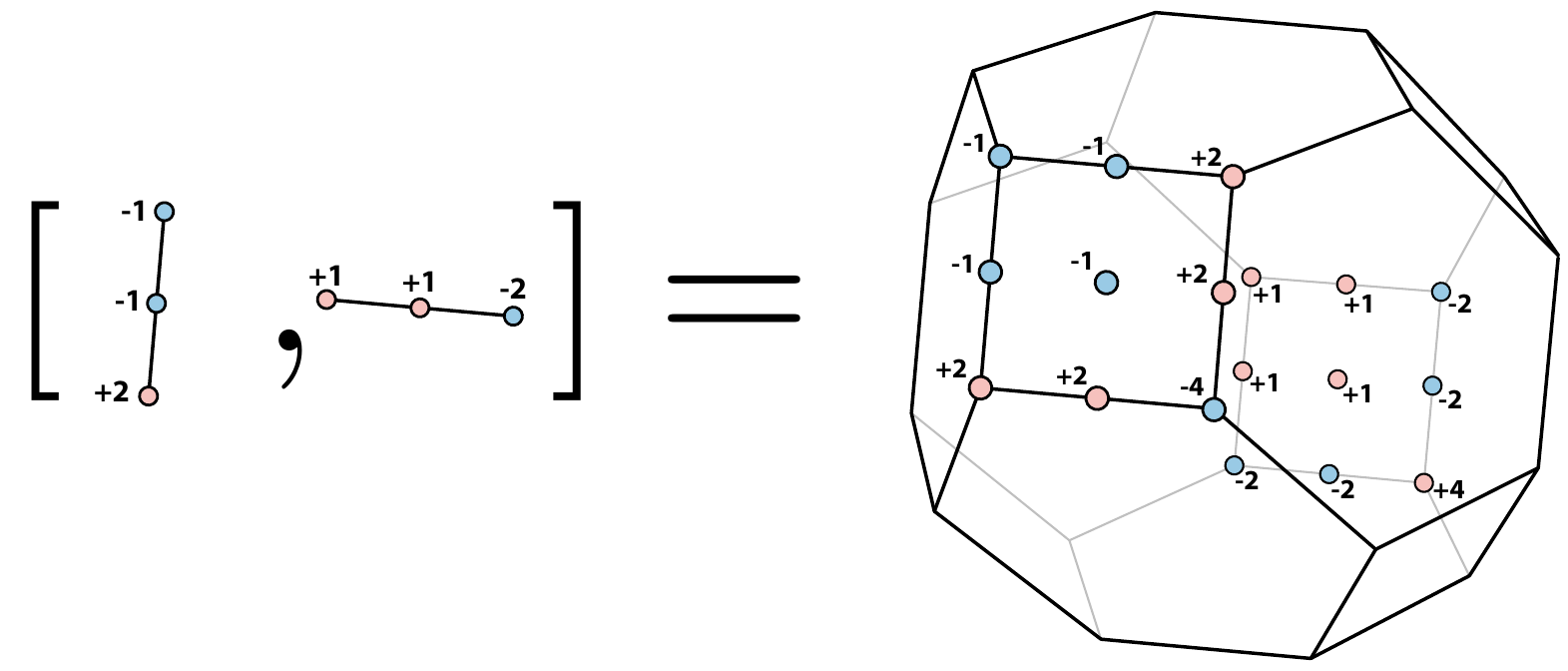}
	\caption{The Lie bracket of (the image of) $\textbf{Zie}$ given in terms of the $\tH$-basis. The multiplication of $\Sigma$ corresponds to embedding facets. The Lie bracket is the commutator of the multiplication, and so embeds a facet twice, on opposite sides and with opposite sign.}
	\label{fig:brack}
\end{figure}

\begin{prop}\label{prop:rightcomb}
The set of standard right-comb trees
\[\big \{  [F]:F\in \Sigma_{i_0}[I]  \big \} \] 
is a basis of $\textbf{Zie}[I]$. 
\end{prop}

\begin{proof}
To show that any tree $\mathcal{T}\in\textbf{Zie}[I]$ is a linear combination of standard right comb-trees, first move $i_0$ to the left most branch of $\mathcal{T}$ by using antisymmetry. Then comb branches to the right using the following consequence of antisymmetry and the Jacobi identity,
	\[  [S,[T,U ] ] =[[S,T],U] -[[S,U],T]    . \]
Therefore standard right-comb trees span $\textbf{Zie}[I]$. The result then follows since the dimension of $\Lie[I]$ is $(n-1)!\,$, see e.g. \cite[Theorem 10.38]{aguiar2017topics}.
\end{proof}


Let 
\[\textbf{Zie}^\ast[I]:= \Hom\big ( \textbf{Zie}[I], \Bbbk\big   )  .   \] 
This defines the positive vector cospecies $\textbf{Zie}^\ast$. For $F\in \Sigma^\ast[I]$, consider the set of trees $\boldsymbol{\mathcal{T}}(F)$ given by
\[    \boldsymbol{\mathcal{T}}(F) :=   \bigsqcup_{ \substack{ \mathcal{T}\in \textbf{Zie}[I]\\[1pt] F_\mathcal{T}=F} } \text{antisym}(\mathcal{T}) . \]
This is indeed a disjoint union because for each $\mathcal{T}'\in \boldsymbol{\mathcal{T}}(F)$, there is a unique tree with debracketing $F$ which is obtained from $\mathcal{T}'$ by switching branches at nodes. For $F\in \Sigma^\ast[I]$, let $\tp_F$ be the function on trees over $I$ given by
\[      \tp_F(\mathcal{T}'): =  \begin{cases}
(-1)^{(\mathcal{T}, \mathcal{T}')}   &\quad  \text{if}\   \mathcal{T}'\in  \boldsymbol{\mathcal{T}}(F), \ \text{where}\   \mathcal{T}'\in  \text{antisym}(\mathcal{T})       \\
0  &\quad \text{if}\  \mathcal{T}'\notin \boldsymbol{\mathcal{T}}(F).
\end{cases}
\]            

\begin{prop}
For $F\in \Sigma^\ast[I]$, we have
\[  \tp_F\in \textbf{Zie}^\ast[I]  . \]
\end{prop}
\begin{proof}
The definition of $\tp_F$ ensures that it satisfies antisymmetry. For the Jacobi identity, suppose that a tree $\mathcal{T}'_{STU}\in \boldsymbol{\mathcal{T}}(F)$ has a branch $[[S,T],U]$. We may assume that $\tp_F(\mathcal{T}'_{STU})=1$ by antisymmetry. Then the tree $\mathcal{T}''_{STU}$ obtained from $\mathcal{T}'_{STU}$ by replacing the branch $[[S,T],U]$ with $[S,[T,U]]$ has $\tp_F(\mathcal{T}''_{STU})=1$. Therefore the tree $\mathcal{T}'_{TUS}$ obtained from $\mathcal{T}'_{STU}$ by replacing the branch $[[S,T],U]$ with $[[T,U],S]$ has $\tp_F(\mathcal{T}'_{STU})=-1$. However, the tree $\mathcal{T}'_{UST}$ obtained from $\mathcal{T}'_{STU}$ by replacing the branch $[[S,T],U]$ with $[[U,S],T]$ has $\tp_F(\mathcal{T}'_{UST})=0$, because there does not exist a switching of the nodes of $[[U,S],T]$ which produces the debracketing $(S,T,U)$. Then
\[ \tp_F(\mathcal{T}'_{STU})   +\tp_F(\mathcal{T}'_{TUS})+\tp_F(\mathcal{T}'_{UST})=  1-1+0=0. \qedhere    \]
\end{proof}

Given a distinguished coordinate $i_0\in I$, let 
\[\Sigma_{i_0}^\ast[I]  := \big  \{ F=(S_1, \dots, S_k)\in \Sigma^\ast[I]:  i_0\in S_1  \big \}  .\]

\begin{prop}
The set of functions on trees
\[\big \{ \tp_F :F\in \Sigma^\ast_{i_0}[I]  \big \}  \] 
is the basis of $\textbf{Zie}^\ast[I]$ which is dual to the basis of standard right-comb trees from \autoref{prop:rightcomb}.
\end{prop}
\begin{proof}
We have
\[    \tp_F\big ( [F]\big )=1     \]
because $\mathcal{T}\in   \boldsymbol{\mathcal{T}}(F)$ with $(-1)^{(\mathcal{T}, \mathcal{T})}=(-1)^{0} = 1$. Let $F\in \Sigma^\ast_{i_0}[I]$ and $G\in \Sigma_{i_0}[I]$ with
\[[G] \in \boldsymbol{\mathcal{T}}(F).\] 
This means that there exists a tree $\mathcal{T}$ with debracketing $F$ such that $[G] \in \text{antisym}( \mathcal{T} )$, which implies that $\mathcal{T}\in  \text{antisym}([G])$. Since $[G]$ is a right comb-tree, it is the only tree in $\text{antisym}([G])$ which contains $i_0$ in its left most lump. Therefore we must have $\mathcal{T}=[G]$, and so $F= G$. The contrapositive of this is that if $F\neq G$, then $[G] \notin \text{antisym}( \mathcal{T} )$, and so
\[    \tp_F\big ( [G]\big )=0 . \qedhere \]
\end{proof}

We call the functions $\tp_F$, for $F\in \Sigma_{i_0}^\ast[I]$, the \emph{$\tp$-basis} of $\textbf{Zie}^\ast$. Of course, it depends on the choice $i_0\in I$. Let 
\[\cU^\ast:\Sig^\ast \twoheadrightarrow \textbf{Zie}^\ast\] 
denote the linear dual of the map $\cU$. Thus, $\cU^\ast$ is isomorphic to the indecomposable quotient map of $\Sig^\ast$, i.e. $\cU^\ast$ quotients out the image of the multiplication of $\Sig^\ast$.

\begin{prop}
The map $\cU^\ast$ is given by
\[     \cU^\ast: \Sig^\ast\twoheadrightarrow \textbf{Zie}^\ast, \qquad      \tP_F\mapsto  \tp_F      . \]
\end{prop}
\begin{proof}
Since $\cU^\ast$ is the dual of $\cU$, for $F\in \Sigma^\ast[I]$ and $G\in \Sigma_{i_0}[I]$, we have
\[\cU^\ast(\tP_F)([G])=\tP_F(\tQ_{[G]}) .\]
But, directly from the definitions of $\tQ_\mathcal{T}$ and $\tp_F$, we see that 
\[ \tP_F(\tQ_{[G]}) =   \tp_F([G])   .  \]
Thus
\[\cU^\ast(\tP_F)([G])= \tp_F([G]) . \]
Because the right-comb trees $[G]$, for $G\in \Sigma_{i_0}[I]$, form a basis of $\textbf{Zie}[I]$, we must have $\cU^\ast(\tP_F)=\tp_F$.
\end{proof}

Let
\[    \tc_p:= \cU^\ast(  \tC_p ) \qquad \text{and} \qquad  \tm_F:= \cU^\ast(  \tM_F ).            \]
Then $\{\tc_F :F\in \Sigma^\ast_{i_0}[I] \}$ and $\{ \tm_F :F\in \Sigma^\ast_{i_0}[I]\}$ are two more bases of $\textbf{Zie}^\ast[I]$. We call them the \emph{$\tc$-basis} and the \emph{$\tm$-basis} respectively. 
Since $\textbf{Zie}^\ast$ is the quotient of $\Sig^\ast$ by the image of its multiplication, we have the following three choices of generating relations for $\textbf{Zie}^\ast$, quasishuffling, shuffling, and signed-quasishuffling,
\[  \sum_{H\preceq (  F|G  )}  \tm_H=0  , \qquad     \sum_{H\preceq_{l}  (  F|G  )}  \tp_H=0    , \qquad        \sum_{H\preceq (  F|G  )} (-1)^{l(H)} \tc_H=0.      \]
We have a relation for each choice of triple
\[(S,T)\in [I;\textbf{2}],\qquad F\in \Sigma^\ast[S],\quad G\in \Sigma^\ast[T].\] 
The quotient of $\textbf{Zie}^\ast$ by the relations $\tp_F=0$, for $F\notin \text{L}^\ast[I]$, is the (positive) Lie cooperad $\Lie^\ast$, whose shuffle relations are well-known. 

The Lie cobracket $\partial$ of $\textbf{Zie}^\ast$, which is the linear dual map of the Lie bracket $\partial^\ast$ of $\textbf{Zie}$, is given in terms of the $\tp$-basis (and also the $\tc$-basis and $\tm$-basis) by the cocommutator of deconcatenation, 
\[  \partial:\textbf{Zie}^\ast\to  \textbf{Zie}^\ast\otimes_{\text{Day}} \textbf{Zie}^\ast, \qquad   \partial_{[S,T]}(\tp_F):=\tp_{F\talloblong_S}\otimes \tp_{F\! \fatslash_{\, T}} -\tp_{F\talloblong_T}\otimes \tp_{F\! \fatslash_{\, S} }.\]

\section{Geometric Realizations} \label{georel}

\noindent We now give two geometric realizations of the Hopf algebra $\Sig^\ast$. First, we realize $\Sig^\ast$ as \hbox{piecewise-constant} functions on the braid arrangement. Second, we realize $\Sig^\ast$ as functionals of piecewise-constant functions on the adjoint braid arrangement which arise from formal linear combinations of conical subspaces generated by coroots. The quotients obtained by restricting these realizations to chambers are the commutative Hopf algebra of linear orders $\textbf{L}^\ast$ for the braid arrangement (this is clear), and the indecomposable quotient Lie coalgebra $\textbf{Zie}^\ast$ for the adjoint braid arrangement (we prove this in \autoref{main}). 




\subsection{Root Datum and Hyperplane Arrangements} \label{sec:root}

We describe the braid arrangement and its corresponding adjoint arrangement. These hyperplane arrangements are naturally constructed over the root datum of $\text{SL}_I(\bC)$,\footnote{\ Note that some authors construct the braid arrangement over the root datum of $\text{GL}_I(\bC)$, and recover the braid arrangement in our sense as the essentialization. The corresponding `adjoint' arrangement is known as the all-subset arrangement, and the adjoint arrangement in our sense is recovered as a restriction.} or dually $\text{PGL}_I(\bC)$. 
Consider the set
\[      \bR^I:=\{  \text{functions}\     \lambda: I\to \bR    \}     \]
and let $\bZ^I\subset \bR^I$ be the subset of integer-valued functions. Then $\bR^I$ and $\bZ^I$ are groups by taking the pointwise addition of functions. For a subset $S\subseteq I$, let $\lambda_S\in \bZ^I$ be given by 
\[\lambda_S(i):=1 \quad \text{if} \quad  i\in S \qquad \text{and} \qquad \lambda_S(i):=0 \quad \text{if} \quad i\notin S.\] 
If we consider functions only up to translations of $\bR$, we obtain the quotient groups
\[     \text{T}^I:= \bR^I/ \bR \lambda_I \qquad \text{and} \qquad \text{P}^I:= \bZ^I/ \bZ \lambda_I  .  \]
The lattice $\text{P}^I\subset \text{T}^I$ is called the \emph{weight lattice} of $\text{SL}_I(\bC)$. For $(S,T)\in [I;\textbf{2}]$, the \emph{fundamental weight} $\lambda_{ST}\in \text{P}^I$ is the image of $\lambda_S\in \bR^I$ in $\text{T}^I$. Recall that for $\text{SL}_I(\bC)$, fundamental weights coincide with minuscule weights. The partial product on $[I;\textbf{2}]$ encodes the addition of fundamental weights, restricted to the case where the sum is again a fundamental weight,
\[       \lambda_{ST} +  \lambda_{UV}= \lambda_{ (S,T)\, \circ\, (U,V) }     .     \]
To see this, we have for example,
\bgroup
\renewcommand{\tabcolsep}{6mm}
\def\arraystretch{1.2}%
\begin{table}[H] 
\begin{tabular}{ c c c }
 $T\supset U$ & $S\supset V$ &  otherwise               \\ \hline
 $\ \ \,   [1:1:1:0:0:0:0]$ & $\ \ \, [1:1:1:1:1:0:0]$ &  $ \ \ \,[1:1:1:0:0:0:0]   $           \\ 
 $+[0:0:0:0:0:1:1]$ &  $+   [0:1:1:1:1:1:1]$    &$+[1:0:0:0:0:1:1]$ \\ 
 $ =[1:1:1:0:0:1:1]$   &$= [1:2:2:2:2:1:1]$&  $=[2:1:1:0:0:1:1]$ \\ 
  &$=                 [0:1:1:1:1:0:0]$ &               \\
\end{tabular}
\end{table}
\noindent Let 
\[\bR I:=\{ h:  h=(h_i)_{i\in I},\, h_i\in \bR   \}\] 
and let $\bZ I \subset \bR I$ be the free $\bZ$-module on $I$. We have the perfect pairing
\[ \la-, - \ra:\bR I\times \bR^I\to \bR, \qquad   \la h, \lambda  \ra := \sum_{i\in I}h_i \,  \lambda(i)    . \]
Let 
\[\text{T}^\vee_I  :=\big \{  h\in \bR I  :\la h, \lambda_I\ra =0    \big   \} \qquad \text{and}\qquad   \text{Q}^\vee_I   :=\big \{  h\in \bZ I  :\la h, \lambda_I \ra =0    \big   \}. \]

\begin{figure}[t]
	\centering
	\includegraphics[scale=0.4]{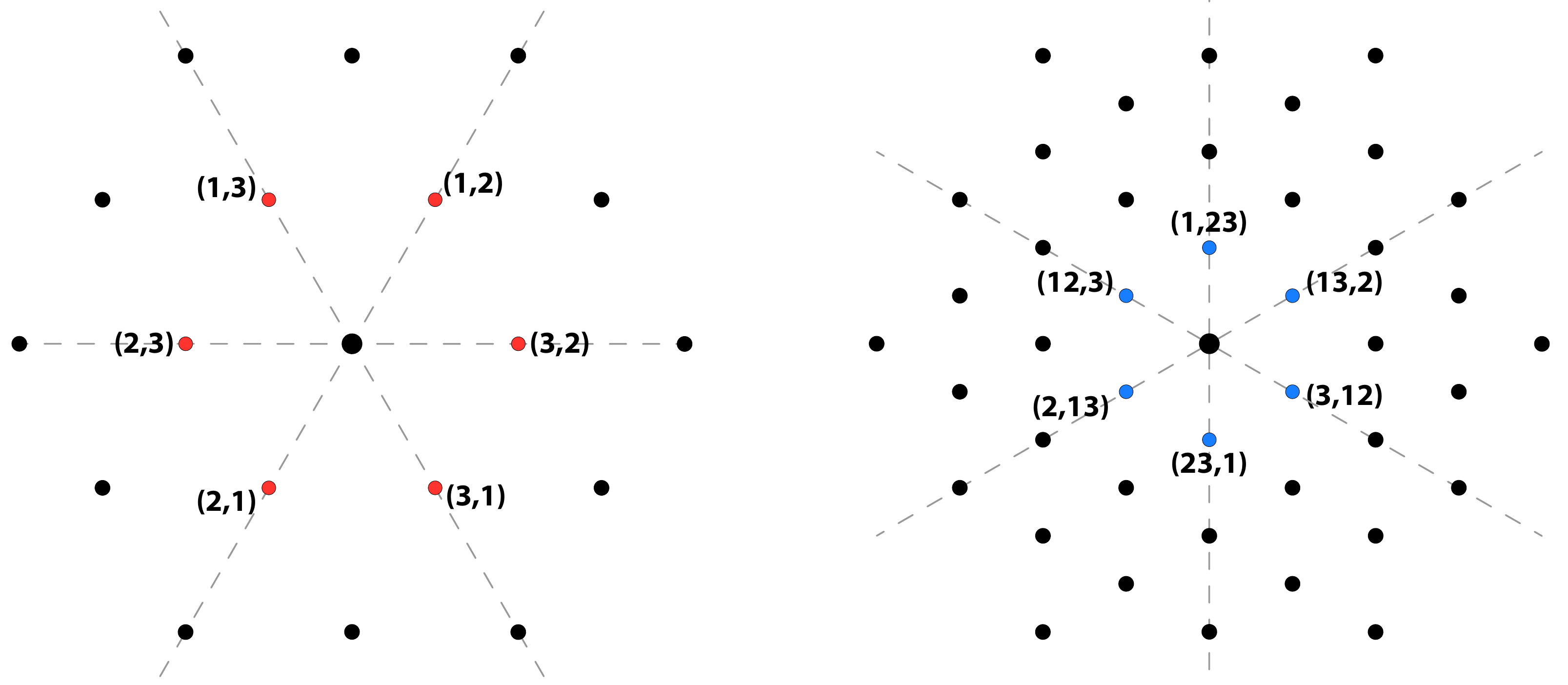}
	\caption{Coroots (red, left) and the coweight lattice, fundamental weights (blue, right) and the weight lattice, for $I=\{1,2,3\}$. Special hyperplanes and reflection hyperplanes are shown as dotted lines.}
	\label{fig:coroots}
\end{figure} 

\noindent We have the natural identifications
\[   \bR^I=\Hom( \bR I, \bR ),\qquad       \text{T}^I=\Hom(  \text{T}^\vee_I, \bR  ), \qquad  \text{P}^I=\Hom(  \text{Q}^\vee_I , \bZ ).\] 
The lattice $\text{Q}^\vee_I \subset \text{T}^\vee_I$ is called the \emph{coweight lattice} of $\text{SL}_I(\bC)$. For $(i_1,i_2)\in [\textbf{2};I]$, the \emph{coroot} $h_{i_1 i_2} \in \text{Q}^\vee_I $ evaluates the signed distance between the point labeled by $i_2$ to the point labeled by $i_1$, 
\[\la h_{i_1 i_2},\lambda\ra:=\lambda(i_1)-\lambda(i_2).\] 
The partial product on $[\textbf{2};I]$ encodes the addition of coroots, restricted to the case where the sum is again a coroot,
\[      h_{i_1 i_2} +  h_{i_3 i_4}= h_{ (i_1,i_2)\, \circ\, (i_3,i_4)  }     .     \]
For $(i_1,i_2)\in [\textbf{2};I]$, the \emph{reflection hyperplane} $\cH_{(i_1|i_2)}\subset \text{T}^I$ is given by
\[      \cH_{(i_1|i_2)}  :=  \big  \{   \lambda\in \text{T}^I:  \la h_{i_1  i_2},\lambda\ra =0\big       \}              .  \]
For the set species $\text{T}^{(-)}$ given by $I\mapsto \text{T}^I$, the transposition $i_1\leftrightarrow i_2$ acts by reflection in $\cH_{(i_1|i_2)}$. The collection of all reflection hyperplanes in $\text{T}^I$ is called the \emph{braid arrangement} over $I$. For $(S,T)\in [I;\textbf{2}]$, the \emph{special hyperplane} $\cH_{(S|T)}\subset \text{T}^\vee_I$ is given by
\[        \cH_{(S|T)}:=  \big  \{  h\in \text{T}^\vee_I:  \la h,\lambda_{ST}\ra =0\big       \}              . \]
The collection of all special hyperplanes in $\text{T}^\vee_I$ is called the \emph{adjoint braid arrangement} over $I$. Notice that special hyperplanes are equivalently hyperplanes which are spanned by coroots, whereas fundamental weights can span hyperplanes which are not necessarily reflection hyperplanes (this can be seen in \autoref{fig:convexhulls}). 

The significance of the hyperplanes is as follows; if $\lambda\in \cH_{(i_1|i_2)}$, then the points of $i_1$ and $i_2$ coincide, and if $h\in  \cH_{(S|T)}$, then the function
\[    h:  \text{T}^I\to \bR, \qquad       \lambda\mapsto \la h,\lambda \ra \]
does not depend upon the distances separating the points of $S$ from the points of $T$.

\begin{remark}
The set species $\text{T}^{(-)}$ is a set-theoretic comonoid, with comultiplication the restriction of configurations,
\[   
\Delta_{S,T}:  \text{T}^I\twoheadrightarrow \text{T}^S\times   \text{T}^T
,\qquad   
\Delta_{S,T}(\lambda):=(\lambda|_S, \lambda|_T)   
.\]
Dually, the set cospecies $\text{T}^\vee_{(-)}$ given by $I\mapsto \text{T}^\vee_I$ is a monoid, with multiplication the embedding of special hyperplanes,
\[  
 \mu_{S,T}:     \text{T}^\vee_S\times \text{T}^\vee_T   \hookrightarrow \text{T}^\vee_I , \qquad      \big \la \mu_{ST}(h_S, h_T),\lambda\big \ra:=\la h_S,\lambda|_S\ra + \la h_T,\lambda|_T\ra  
.\]
To make the comonoid $\text{T}^{(-)}$ into a Hopf monoid, one needs to take a certain compactification so that the multiplication has somewhere to land (see the discussion on the Losev-Manin moduli space in the \hyperref[intro]{Introduction}).
\end{remark}

\subsection{Realization over Braid Arrangement}\label{braid}

The vector space
\[\Bbbk^{\text{T}^I}:= \{   \text{functions}\  \text{T}^I\to \Bbbk  \}    \] 
is a $\Bbbk$-algebra, with multiplication the pointwise product of functions. For $(i_1,i_2)\in [\textbf{2};I]$, define the \emph{halfspace} $\hat{\tC}_{i_1 i_2}\in \Bbbk^{\text{T}^I}$ by
\[ \hat{\tC}_{i_1 i_2}: \text{T}^I \to \Bbbk , \qquad  \hat{\tC}_{i_1 i_2}(\lambda) :=
\begin{cases}
 1  &\quad  \text{if}\  \la h_{i_1 i_2}, \lambda \ra \geq 0      \\
0  &\quad \text{otherwise.}
\end{cases}
\]
Let $\shuff[I]$ denote the subalgebra of $\Bbbk^{\text{T}^I}$ which is generated by halfspaces. This defines the vector cospecies $\shuff$ of piecewise-constant functions on the braid arrangement. Each component $\shuff[I]$ is a polyhedral algebra in the sense of \cite{MR1731815}. Monomials in the halfspaces are called \emph{braid cones}. We let preposets index braid cones via
\[         \text{O}[I]\to   \shuff[I], \qquad   p\mapsto  \hat{\tC}_{p}  : =\prod_{(i_1,i_2)\in p}   \hat{\tC}_{i_1 i_2}    .       \]
This is the well-known one-to-one correspondence between preposets and cones of the braid arrangement \cite[Section 3]{vic}, \cite[Section 13.5.1]{aguiar2010monoidal}.	
Let the \emph{braid signature} be the function
\[  \text{T}^I\twoheadrightarrow \Sigma^\ast[I], \qquad   \lambda\mapsto F_\lambda:=  \big \{    (i_1,i_2)\in [\textbf{2};I]   :  \hat{\tC}_{i_1 i_2}(\lambda)=1  \big \}.          \]
Equivalently, $F_\lambda$ is the preposet which is the ordering induced on $I$ by representative configurations $\lambda:I\to \bR$. For $F\in \Sigma^\ast[I]$, the (relatively open) \emph{face} $\hat{\tM}_F$ is the function given by
\[      
\hat{\tM}_F:\text{T}^I\to \Bbbk, \qquad      \hat{\tM}_F(\lambda):=
\begin{cases}
 1  &\quad  \text{if}\  F_\lambda=F         \\
0  &\quad \text{otherwise}.   
\end{cases}
\]
The image under the braid signature of the complement of the reflection hyperplanes is $\text{L}^\ast[I]$. Thus, $F\mapsto \hat{\tM}_F$ puts $\text{L}^\ast$ in one-to-one correspondence with characteristic functions of connected components of the complement of the reflection hyperplanes.  

\begin{prop}\label{prop:rel}
We have 
\[ \hat{\tC}_{p} = \sum_{  F\leq p } \hat{\tM}_F.\]
\end{prop}
\begin{proof}
Let $\lambda\in \text{T}^I$. We have
\[   \hat{\tC}_{p}(\lambda)    =\prod_{ (i_1,i_2)\in  p }    \hat{\tC}_{i_1 i_2} (\lambda)=   1 \quad \iff  \quad  \hat{\tC}_{i_1 i_2} (\lambda)=1\quad \text{for all}\ \   (i_1,i_2)\in  p \quad\iff\quad F_\lambda \leq p .  \]
The support of a face $\hat{\tM}_F$ is the preimage of $F$ under the braid signature. Therefore $\lambda$ is in the support of exactly one face, and so
\[     F_\lambda \leq p \quad \iff\quad    \sum_{  F\leq p } \hat{\tM}_F(\lambda) =1.\] 
Then, since $\hat{\tC}_{p}$ and $\sum_{  F\leq p } \hat{\tM}_F$ take values $0$ and $1$ only, the result follows.  
\end{proof}

The set $\{    \hat{\tM}_F  : F\in \Sigma^\ast[I]  \}$ spans $\shuff[I]$ by \autoref{prop:rel}, and is linearly independent because the faces $\hat{\tM}_F$ are supported by disjoint sets. Therefore we have an isomorphism of cospecies, given by
\[   \Sig^\ast\to  \shuff, \qquad  \tM_F\mapsto \hat{\tM}_F \quad \text{or}\quad  \tC_p\mapsto \hat{\tC}_p  .      \]
We let this isomorphism induce the structure of the commutative Hopf algebra of compositions on the cospecies of functions $\shuff$. If we extend to the permutohedral compactification of $\text{T}^I$, we may interpret the algebraic structure as in \autoref{fig:multcomult}, i.e. in terms of embedding facets and projecting onto facets of the permutohedron. 

Recall that a (closed) \emph{conical space} $\gs$ is a module of the rig $(\bR_{\geq 0}, +, \times)$, and an \emph{open conical space} is a module of the rig $(\bR_{>0}, +, \times)$. For $X\subseteq [I,\textbf{2}]$ any subset, let $ \wt{\gs}_{X}, \gs_{X} \subset \text{T}^I$ denote the respectively open and closed conical spaces which are generated by the fundamental weights \[\{   \lambda_{ST}: (S,T)\in X    \}.\] 
This defines a Galois insertion, with adjoint families as the closed elements. Therefore adjoint families are in one-to-one correspondence with open/closed conical spaces over fundamental weights. Recall that to each preposet $p\in \text{O}[I]$, we associated the subset denoted $\cF_p\subseteq [I,\textbf{2}]$. In this case, let
\[  \gs_{p}:= \gs_{\cF_p}   = \big \{ \text{non-negative $\bR$-linear combinations of}\ \{\lambda_{ST}:(S,T)\leq p\}\big \}   .\] 

\begin{prop}\label{mainthm1}
The braid cone $\hat{\tC}_p$ is the characteristic function of the conical space $\sigma_{p}$. 
\end{prop}
\begin{proof}
Since
\[F_{\lambda_{ST}}=(S,T),\] 
we have $(S,T)\leq p$ if and only if $\hat{\tC}_p(\lambda_{ST})=1$. Then, since the support of $\hat{\tC}_p$ must be closed under taking non-negative linear combinations, we have that $\sigma_{p}$ is contained in the support of $\hat{\tC}_p$. Conversely, suppose that $\lambda_0\in \text{T}^I$ is in the support of $\hat{\tC}_p$, which means that $F_{\lambda_0}\leq p$. The proposition clearly holds when $p$ is a composition, i.e. we have
\[  \gs_{F}= \{   \lambda\in \text{T}^I: F_\lambda \leq F  \}.\] 
Therefore, in particular, $\lambda_0\in \gs_{F_{\lambda_0}}$. Then, since $F_{\lambda_0}\leq p$, we have $(S,T)\leq F_{\lambda_0}$ implies that $(S,T)\leq p$, and so $\lambda_0\in\gs_{F_{\lambda_0}}\subseteq \gs_p$.
\end{proof}


It is now clear that $\cF_p$ is an adjoint family, since an equivalent definition of $\cF_p$ is that it is the set of $(S,T)\in [I,\textbf{2}]$ such that $\lambda_{ST}$ is in the support of $\hat{\tC}_p$. We clearly have 
\[ \gs_q \subseteq \gs_p  \iff q\leq p  \qquad \text{and} \qquad  \wt{\gs}_q \subseteq \wt{\gs}_p \iff q\preceq p.          \]
Thus, the images of $\tM_F$ and $\tC_p$ in $\hat{\Sig}^\ast$ are the characteristic functions of $\wt{\gs}_F$ and $\gs_p$ respectively. Notice also that $\gs_p \cap \gs_q =\gs_{p\,\cup\, q}$. Therefore the polyhedral algebraic structure of $\shuff$ may be given by
\[   \shuff \otimes_{\text{Had}} \shuff \to \shuff, \qquad   \hat{\tC}_p\otimes  \hat{\tC}_q \mapsto \hat{\tC}_p \cdot \hat{\tC}_q  =   \hat{\tC}_{p\, \cup\, q} .      \]
Recall that $\otimes_{\text{Had}}$ denotes the Hadamard monoidal product of vector (co)species.

\subsection{Adjoint Braid Arrangement}  \label{adjoint}
The vector space 
\[\Bbbk^{\text{T}^\vee_I}:= \{   \text{functions}\  \text{T}^\vee_I\to \Bbbk  \}     \] 
is a $\Bbbk$-algebra, with multiplication the pointwise product of functions. For $(S,T)\in [I;\textbf{2}]$, define the \emph{halfspace} $\check{\mathtt{Y}}_{ST}\in \Bbbk^{\text{T}^\vee_I}$ by 
\[ \check{\mathtt{Y}}_{ST}:  \text{T}^\vee_I\to \Bbbk , \qquad  \check{\mathtt{Y}}_{ST}(h) :=
\begin{cases}
 1  &\quad  \text{if}\ \la h, \lambda_{ST}   \ra\geq 0       \\
0  &\quad \text{otherwise.}
\end{cases}
\]
Let $\check{\Sig}^\vee[I]$ denote the subalgebra of $\Bbbk^{\text{T}^\vee_I}$ which is generated by halfspaces. This defines the vector species $\check{\Sig}^\vee$ of piecewise-constant functions on the adjoint braid arrangement. Each component $\check{\Sig}^\vee[I]$ is a polyhedral algebra in the sense of \cite{MR1731815}. Let an \emph{adjoint cone} be a monomial in the halfspaces. We let adjoint families index adjoint cones via 
\[       \text{O}^\vee[I]\to  \check{\Sig}^\vee[I], \qquad \cF\mapsto \check{\mathtt{Y}}_{\cF}:= \prod_{(S,T)\in \cF}    \check{\mathtt{Y}}_{ST}      .  \] 
The definition of adjoint families (as the closed elements of a Galois insertion) ensures that this is a one-to-one correspondence between adjoint families on $I$ and adjoint cones of $ \text{T}^\vee_I$.
Let the \emph{adjoint signature} be the function
\[  \text{T}^\vee_I\twoheadrightarrow \Sigma^\vee[I], \qquad   h\mapsto \cS_h:=  \big \{    (S,T)\in [I;\textbf{2}]   :  \check{\mathtt{Y}}_{ST}(h)=1  \big \}.          \]
For $\cS\in \Sigma^\vee[I]$, the (relatively open) \emph{adjoint face} $\check{\mathtt{H}}_\cS$ is the function given by
\[     \check{\mathtt{H}}_\cS:  \text{T}^\vee_I\to \Bbbk , \qquad  \check{\mathtt{H}}_\cS(h):=
\begin{cases}
 1  &\quad  \text{if}\  \cS_h=\cS         \\
0  &\quad \text{otherwise}.
\end{cases}\]     
The image under the adjoint signature of the complement of the special hyperplanes is $\text{L}^\vee[I]$. Thus, $\cS\mapsto\check{\mathtt{H}}_\cS$ puts $\text{L}^\vee$ in one-to-one correspondence with characteristic functions of connected components of the compliment of the special hyperplanes. 

\begin{prop} \label{rel2}
We have
\[    \check{\mathtt{Y}}_{\cF}=  \sum_{ \cS\leq \cF  }   \check{\mathtt{H}}_\cS . \]
\end{prop}
\begin{proof}
Let $h\in \text{T}^\vee_I$. We have
\[   \check{\mathtt{Y}}_{\cF}(h)    =\prod_{ (S,T)\in  \cS } \check{\mathtt{Y}}_{ST} (h)=   1 \quad \iff  \quad  \check{\mathtt{Y}}_{ST} (h)=1 \quad \text{for all}\ \  (S,T)\in  \cF \quad\iff\quad \cS_h \leq \cF .  \]
The support of $\check{\mathtt{H}}_\cS$ is the preimage of $\cS$ under the adjoint signature. Therefore $h$ is in the support of exactly one adjoint face, and so
\[     \cS_h \leq \cF \quad \iff\quad    \sum_{  \cS\leq \cF } \check{\mathtt{H}}_\cS(h) =1.\] 
Then, since $\mathtt{Y}_\cF$ and $\sum_{  \cS\leq \cF } \check{\mathtt{H}}_\cS$ take values $0$ and $1$ only, the result follows.  
\end{proof}

For $\cF\in \text{O}^\vee[I]$ an adjoint family, let $\sigma^\vee_{\cF}$ denote the dual conical space of $\sigma_{\cF}$, given by
\[       \sigma^\vee_{\cF}:=    \big \{  h \in \text{T}^\vee_I :  \la  h, \lambda   \ra \geq 0 \   \text{ for all }  \lambda \in \sigma_{\cF}  \big   \}    .      \]

\begin{prop} 
The adjoint cone $\check{\mathtt{Y}}_{\cF}$ is the characteristic function of the conical space $\sigma^\vee_{\cF}$. 
\end{prop}

\begin{proof}
The result follows from the fact that cone duality intertwines intersections with Minkowski sums.
\end{proof}


For $p\in \text{O}[I]$ a preposet, let $\sigma^\vee_{p}$ denote the dual conical space of $\sigma_{p}$, given by
\[       \sigma^\vee_{p}:=    \big \{  h\in \text{T}^\vee_I :  \la  h, \lambda   \ra \geq 0 \quad \text{ for all }\ \  \lambda \in \sigma_{p}  \big   \}    .      \]
The conical spaces $\sigma^\vee_F$ are tangent cones to permutohedra, and the conical spaces $\sigma^\vee_p$ are tangent cones to generalized permutohedra. This follows from the characterization of generalized permutohedra as polyhedra whose normal fans are coarsenings of the braid arrangement, see e.g. \cite{aguiar2017hopf}. By cone duality, the conical space $\sigma^\vee_{p}$ is generated by the coroots 
\[\big \{h_{i_1 i_2} : (i_1, i_2)\in p\big \}.\] 
Let $\wt{\sigma}^\vee_p$ denote the open conical space which is generated by the same coroots. 

The set $\{   \check{\mathtt{H}}_\cS: \cS\in \Sigma^\vee[I]   \}$ spans $\check{\Sig}^\vee[I]$ by \autoref{rel2}, and is linearly independent because the adjoint faces $\check{\mathtt{H}}_\cS$ are supported by disjoint sets. Let us give each component $\check{\Sig}^\vee[I]$ the structure of a real Hilbert space by letting adjoint faces be an orthonormal basis. If $\tY\in \check{\Sig}^\vee[I]$ is the characteristic function of a region $X\subset \text{T}^\vee_I$, let the \emph{characteristic functional} of $X$ be the Riesz representation of $\tY$. 

\subsection{Realization over Adjoint Braid Arrangement}  \label{2adjoint}

For $p \in \text{O}[I]$, let $\check{\tC}_{p}$ be the characteristic functional of $\sigma^\vee_p$, thus
\[  \check{\tC}_{p}: \check{\Sig}^\vee[I]\to \Bbbk, \qquad \check{\tC}_{p}( \check{\mathtt{H}}_\cS ) :=
\begin{cases}
 1  &\quad  \text{if}\    \cS \leq  \CMcal{F}_p      \\
0  &\quad \text{otherwise.}
\end{cases}\] 
Let $\adshuff[I]$ denote the span of $\{ \check{\tC}_{p}:p\in \text{O}[I] \}$ in $\Hom(\check{\Sig}^\vee[I], \Bbbk)$. This defines the cospecies $\adshuff$ of functionals which arise from formal linear combinations of conical spaces over coroots. 

\begin{prop}
We have a cospecies isomorphism $\mathcal{D}$, given by
\[  \mathcal{D}:  \shuff \to \adshuff, \qquad    \hat{\tC}_p\mapsto      \check{\tC}_p   .               \]
\end{prop}
\begin{proof}
Since $\hat{\tC}_p$ is the characteristic function of $\sigma_p$, and the Riesz representation of $\check{\tC}_p$ is the characteristic function of $\sigma^\vee_p$, we see that $\adshuff[I]$ is naturally the dual of the polyhedral algebra $\hat{\Sig}^\ast[I]$ in the sense of \cite[Theorem 2.7]{MR1731815}, and $\mathcal{D}$ is the duality map.
\end{proof}


We let $\mathcal{D}$ induce the structure of the commutative Hopf algebra of compositions on the cospecies of functionals $\adshuff$. The image of the pointwise product in $\adshuff$ is called \emph{convolution},
\[   \adshuff\otimes_{\text{Had}} \adshuff\to \adshuff, \qquad   \check{\tC}_p  \otimes  \check{\tC}_q \mapsto  \check{\tC}_p  \star  \check{\tC}_q :=    \check{\tC}_{ p\, \cup\, q  }.  \]
The restriction of convolution to coroot cones is Minkowski sum. 



For $F\in \Sigma^\ast[I]$, let 
\[\check{\tM}_F\in \Hom(\check{\Sig}^\vee[I], \Bbbk)\] 
be the characteristic functional of the relative interior $\wt{\sigma}^\vee_{\bar{F}}$ of the permutohedral cone $\sigma^\vee_{\bar{F}}$, with sign $(-1)^{  l(F)-1 }$.

\begin{thm}
For $F\in \Sigma^\ast[I]$, we have
\[   \mathcal{D}( \hat{\tM}_F)  = \check{\tM}_F.\]
\end{thm}
\begin{proof}
For $F\in \Sigma^\ast[I]$, we have
\[  \mathcal{D}  (\hat{\tM}_F)=  \mathcal{D} \Bigg (     \prod^{\text{pointwise}}_{(i_1,i_2)\in F}   (\hat{\tC}_{I}-\hat{\tC}_{i_2 i_1}) \Bigg  )= \prod^{\text{convol}}_{(i_1,i_2)\in F}  \mathcal{D} (\hat{\tC}_{I}-\hat{\tC}_{i_2 i_1})= \prod^{\text{convol}}_{(i_1,i_2)\in \bar{F}}    \big(  \check{\tC}_{I} -  \check{\tC}_{i_1 i_2} \big)   = \check{\tM}_F.\footnote{\ in $\hat{\tC}_{I}$ and $\check{\tC}_{I}$, $I$ denotes the partition of $I$ into singletons $(i_1|\dots|i_n)$, modeled as the empty preposet $p=\emptyset$}   \]
The final equality follows by multiplying out the convolution product, and then doing \hbox{inclusion-exclusion} of faces of the permutohedral cone $\gs_F^\vee$.
\end{proof}


Therefore, the isomorphism $\mathcal{D}:\shuff \to \adshuff$ is also given by $\hat{\tM}_F\mapsto \check{\tM}_F$. We let $\check{\tP}_F$ denote the image of $\tP_F$ in $\adshuff$.

\section{The Indecomposable Quotient} \label{main}
\noindent We now show that the indecomposable quotient Lie coalgebra of the adjoint realization of $\Sig^\ast$ is simply the restriction of functionals to chambers. Recall that the indecomposable quotient of $\Sig^\ast$ is isomorphic to 
\[\textbf{Zie}^\ast=\textbf{Lie}^\ast\circ \textbf{E}_+,\]
where $\textbf{Lie}^\ast$ is the Lie cooperad, $\textbf{E}_+$ is the positive exponential species, and $\circ$ is plethysm. We show that the resulting geometric realization of $\textbf{Zie}^\ast$, which we know consists of functionals in the span of characteristic functionals of generalized permutohedral cones, is also characterized as functionals which satisfy the Steinmann relations from QFT. We shall see that the cobracket of this realization of $\textbf{Zie}^\ast$ may be interpreted as the discrete differentiation of functionals across special hyperplanes. 
\subsection{Permutohedral Cones and the Steinmann Relations }

Let $\check{\textbf{Z}}\textbf{ie}^\ast$ be the quotient cospecies of $\adshuff$ which is obtained by restricting functionals to adjoint chambers, thus
\[       \check{\textbf{Z}}\textbf{ie}^\ast[I] := \bigslant{\adshuff[I] }{   \big \la    f \in \adshuff[I] :    f( \check{\mathtt{H}}_\cS  )= 0\ \text{for all}\  \cS\in \text{L}^\vee[I]    \big \ra   }     .          \]
We denote the corresponding quotient map by
\[\check{\cU}^\ast: \adshuff \twoheadrightarrow  \check{\textbf{Z}}\textbf{ie}^\ast  . \] 
Define the following functionals on adjoint chambers, 
\[  \check{\tp}_F:= \check{\cU}^\ast(  \check{\tP}_F ), \qquad    \check{\tc}_p:= \check{\cU}^\ast(  \check{\tC}_p ), \qquad  \check{\tm}_F:= \check{\cU}^\ast(  \check{\tM}_F ).            \]
See \autoref{fig:qlambda}. In particular, the functionals $\check{\tc}_p$ are characteristic functionals of generalized permutohedral tangent cones, taken modulo higher codimensions. Therefore we may characterize the subspace $\check{\textbf{Z}}\textbf{ie}^\ast[I]\subset \Hom(\textbf{L}^\vee[I], \Bbbk)$ as the span of characteristic functionals of (generalized) permutohedral tangent cones.

\begin{figure}[t]
	\centering
	\includegraphics[scale=0.7]{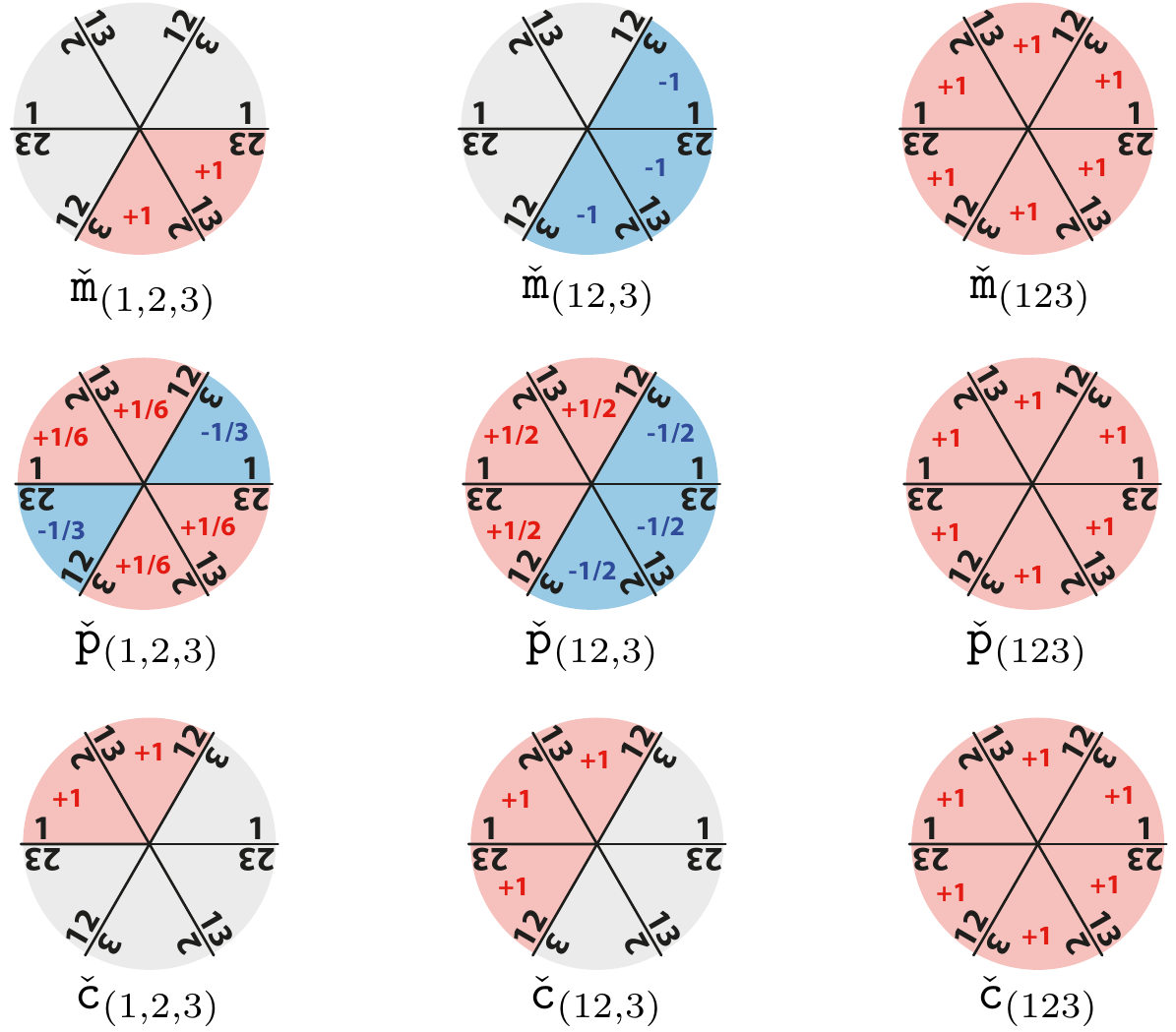}
	\caption{The images in the quotient $\check{\textbf{Z}}\textbf{ie}^\ast$ of the various bases of $\Sig^\ast$.}
	\label{fig:qlambda}
\end{figure} 

\begin{remark}\label{rem}
The adjoint analog of $\check{\textbf{Z}}\textbf{ie}^\ast$, i.e. the quotient of $\shuff$ obtained by restricting functions to Weyl chambers, is the cospecies $\hat{\textbf{L}}^\ast$ given by
\[       \hat{\textbf{L}}^\ast[I] :=   \bigslant{\hat{\Sig}^\ast[I] }{   \big \la  \mathtt{M}_{F}   :  F\notin \text{L}^\ast[I]    \big \ra   }     .        \]
This is a geometric realization of the commutative Hopf algebra of linear orders $\textbf{L}^\ast$. Since $\textbf{L}^\ast$ is the coenveloping algebra of the Lie cooperad $\Lie^\ast$, we have a natural isomorphism 
\[\textbf{L}^\ast\cong \textbf{E}\circ \Lie^\ast.\] 
Therefore $\textbf{L}^\ast$ is a right $\Lie^\ast$-comodule. If we look at what the $\Lie^\ast$-coaction should be for the realization $\hat{\textbf{L}}^\ast$, we see that it is discrete differentiation of functions across reflection hyperplanes. This is geometrically the same as the Lie structure we shall give $\check{\textbf{Z}}\textbf{ie}^\ast$. 
\end{remark}
Let a \emph{Steinmann functional} $f\in \Hom(\textbf{L}^\vee[I], \Bbbk)$ over $I$ be a $\Bbbk$-valued function on adjoint chambers which satisfies the Steinmann relations \cite[Section 4.2]{lno2019}. We denote the cospecies of Steinmann functionals by $\boldsymbol{\Gamma}^\ast$. In \cite{lno2019}, it was shown that restricting to Steinmann functionals is necessary and sufficient for the discrete differentiation of functionals across hyperplanes to be a Lie cobracket, so that $\boldsymbol{\Gamma}^\ast$ has the structure of a Lie coalgebra. We now show that $\check{\textbf{Z}}\textbf{ie}^\ast$ is exactly the Lie coalgebra of Steinmann functionals, so $\boldsymbol{\Gamma}^\ast=\check{\textbf{Z}}\textbf{ie}^\ast$. 


See \cite[Definition 2.3]{lno2019} for the precise definition of the discrete derivative. 

\begin{lem}\label{cobracket}
For $F\in \Sigma^\ast[I]$ and $(S,T)\in [I;\textbf{2}]$, the discrete derivative $\boldsymbol{\partial}_{[S,T]}\, \check{\tc}_F$ of the functional $\check{\tc}_F$ across the special hyperplane $\cH_{(S|T)}$ is given by
\[       \boldsymbol{\partial}_{[S,T]}\, \check{\tc}_F =  \boldsymbol{\mu}_{(S|T)}\big( \check{\tc}_{  F\talloblong_S } \otimes   \check{\tc}_{  F\fatslash_{\, T} } - \check{\tc}_{  F\talloblong_T } \otimes   \check{\tc}_{  F\fatslash_{\, S} } \big ).\footnote{\ see \cite[Definition 1.6]{lno2019} for the definition of $\boldsymbol{\mu}_{(S|T)}$}     \]
\end{lem}

\begin{proof}
Let $\mathtt{X}$ be a codimension one adjoint face which is supported by the special hyperplane $\cH_{(S|T)}$. Let $\mathtt{X}^{[S,T]}$, respectively $\mathtt{X}^{[T,S]}$, be the adjoint chamber with facet $\mathtt{X}$ such that $\check{\mathtt{Y}}_{ST} \cdot \mathtt{X}^{[S,T]}=\mathtt{X}^{[S,T]}$, respectively $\check{\mathtt{Y}}_{ST}\cdot  \mathtt{X}^{[T,S]}=0$. 
First, suppose that $(S,T)\leq F$. In this case, we need to show that
\[       \boldsymbol{\partial}_{[S,T]}\, \check{\tc}_F =  \boldsymbol{\mu}_{(S|T)}( \check{\tc}_{  F|_S } \otimes   \check{\tc}_{  F|_{T} }   )      .      \]
By the definition of the derivative, we have
\[        \boldsymbol{\partial}_{[S,T]} \check{\tc}_F(\mathtt{X})=      \check{\tc}_F(\mathtt{X}^{[S,T]})-\check{\tc}_F(\mathtt{X}^{[T,S]})           .            \]
However, since $(S,T)\leq F$, we have $\check{\tc}_F(\mathtt{X}^{[T,S]})=0$, so that
\[         \boldsymbol{\partial}_{[S,T]} \check{\tc}_F(\mathtt{X})=      \check{\tc}_F(\mathtt{X}^{[S,T]})          .            \]
Then, directly from the definitions, we see that
\[    \boldsymbol{\mu}_{(S|T)}( \check{\tc}_{  F|_S } \otimes   \check{\tc}_{  F|_{T} }) (\mathtt{X})=1 \qquad \iff \qquad    \check{\tc}_F(\mathtt{X}^{[S,T]})=1     .    \]
Since these functionals take values $0$ or $1$, the result follows.
The case $(T,S)\leq F$ then follows by antisymmetry of the derivative. Finally, if $S$ is not an initial or terminal interval of $F$, then
\[         \boldsymbol{\mu}_{(S|T)}\big( \check{\tc}_{  F\talloblong_S } \otimes   \check{\tc}_{  F\fatslash_{\, T} } - \check{\tc}_{  F\talloblong_T } \otimes   \check{\tc}_{  F\fatslash_{\, S} } \big )  =  \boldsymbol{\mu}_{(S|T)}(0-0)=0.       \]
Also, in this case, we have
\[ \check{\tc}_F(\mathtt{X}^{[S,T]})=\check{\tc}_F(\mathtt{X}^{[T,S]}),  \]
and so $ \boldsymbol{\partial}_{[S,T]}\, \check{\tc}_F =0$, as required.
\end{proof}

\begin{cor}
The functionals $\{\check{\tc}_F: F\in \Sigma^\ast[I]\}$ satisfy the Steinmann relations. More generally, characteristic functionals of generalized permutohedral tangent cones $\{\check{\tc}_p: p\in \text{O}[I]\}$ satisfy the Steinmann relations. 
\end{cor}

\begin{proof}
In \autoref{cobracket}, we showed that the first derivatives of $\check{\tc}_F$ factorize. The result then follows since a functional satisfies the Steinmann relations if and only if its first derivatives factorize, see \cite[Section 4.2]{lno2019}.
\end{proof}

But what about the converse?

\begin{thm}
A functional on adjoint chambers satisfies the Steinmann relations if and only if it is a linear combination of characteristic functionals of (generalized) permutohedral tangent cones. Therefore \[\check{\textbf{Z}}\textbf{ie}^\ast=\boldsymbol{\Gamma}^\ast\]
as cospecies. 
\end{thm}

\begin{proof}
Let $n=|I|$. For $m\in \bN$ with $m\leq n$, let $f_m$ denote a Steinmann functional over $I$ such that $\boldsymbol{\partial}_{[F]} f_m=0$ for all $F\in  \Sigma_{i_0}^\ast[I]$ with $l(F)>m$. Notice that if $l(F)=m$, then $\boldsymbol{\partial}_{[F]} f_m$ is a constant functional, whose value we denote by $\upsilon(\boldsymbol{\partial}_{[F]} f_m)$.
Define
\[      f_{m-1}  =  f_m -  \sum_{ \{ F\in  \Sigma_{i_0}^\ast[I]\, :\, l(F)=m  \} } \upsilon(\boldsymbol{\partial}_{[F]} f_m)\, \check{\tc}_F .      \]
Then, for $F\in  \Sigma_{i_0}^\ast[I]$ with $l(F)>m-1$, we have
\[       \boldsymbol{\partial}_{[F]}   f_{m-1}=  \boldsymbol{\partial}_{[F]}  f_m-    \sum_{ \{ G\in  \Sigma_{i_0}^\ast[I]\, :\, l(G)=m\} }   \upsilon(\boldsymbol{\partial}_{[G]} f_m)\,   \boldsymbol{\partial}_{[F]}\check{\tc}_G=0  .   \]
To see this, notice that if $l(F)>m$ then everything is equal to zero, and if $l(F)=m$, then $\boldsymbol{\partial}_{[F]}\check{\tc}_G$ is constant with 
\[       \upsilon(\boldsymbol{\partial}_{[F]}\check{\tc}_G) =\delta_{FG} .    \]
This shows that we can systematically perturb a Steinmann functional with linear combinations of the functionals $\{\check{\tc}_F: F\in \Sigma_{i_0}^\ast[I]\}$, gradually making it more and more symmetric, until we obtain a constant functional, which will be a scaling of $\check{\tc}_{(I)}$.
\end{proof}

\begin{remark}\label{ref}
This result is not surprising, given that the Steinmann relations are about enjoying factorization of the discrete derivative across special hyperplanes (in the case of characteristic functionals, the derivative is just taking the boundary), and then \cite[Theorem 6.1]{aguiar2017hopf}, which tells us that generalized permutohedra are exactly those polyhedra which enjoy type $A$ factorization of their boundaries.
\end{remark}


Let us equip the cospecies $\check{\textbf{Z}}\textbf{ie}^\ast$ with the cobracket $\partial$ of discrete differentiation across hyperplanes, so that it is now exactly the Lie coalgebra $\boldsymbol{\Gamma}^\ast$ from \cite{lno2019},
\[      \partial:   \check{\textbf{Z}}\textbf{ie}^\ast\to \check{\textbf{Z}}\textbf{ie}^\ast \otimes_{\text{Day}} \check{\textbf{Z}}\textbf{ie}^\ast  , \qquad    \partial_{[S,T]}\check{\tc}_F : =  \check{\tc}_{  F\talloblong_S } \otimes   \check{\tc}_{  F\! \fatslash_{\, T} } - \check{\tc}_{  F\talloblong_T } \otimes   \check{\tc}_{  F\! \fatslash_{\, S} }   .   \] 
We now dispense with the notations $\boldsymbol{\Gamma}$ and $\boldsymbol{\Gamma}^\ast$.

\subsection{An Isomorphism of Lie Coalgebras}

If a preposet $p\in \text{O}[I]$ has at least two blocks, then $\check{\tC}_p(\check{\mathtt{H}}_\cS  )= 0$ for all $\cS \in \text{L}^\vee[I]$, and so $\check{\tc}_p=0$. Therefore the Lie coalgebra $\check{\textbf{Z}}\textbf{ie}^\ast$ is a quotient of the indecomposable quotient of $\adshuff$, and so we have a surjective morphism of cospecies
\[  \textbf{Zie}^\ast\twoheadrightarrow \check{\textbf{Z}}\textbf{ie}^\ast, \qquad \tc_p\mapsto \check{\tc}_p .  \]
We now show that these two quotients of $\adshuff$ actually coincide, so that in fact $\textbf{Zie}^\ast\cong \check{\textbf{Z}}\textbf{ie}^\ast$.


\begin{lem} \label{main3}
The set of functionals
\[    \big \{   \check{\tc}_F: F\in \Sigma^\ast_{i_0}[I]    \big  \} \subset \check{\textbf{Z}}\textbf{ie}^\ast[I]     \]
is linearly independent.
\end{lem}
\begin{proof}
We prove by induction on $n=|I|$. Suppose we have coefficients $a_F\in \Bbbk$ such that
\[      \sum_{ F\in \Sigma^\ast_{i_0}[I]  } a_F\, \check{\tc}_F=0.         \]  
Let $(S,T)\in [I, \textbf{2}]$ with $i_0\in S$. Then
\[         \sum_{ F\in \Sigma^\ast_{i_0}[I]} a_F \, \big( \partial_{[S,T]}\, \check{\tc}_{F} \big)      =	\partial_{[S,T]}  \Bigg(\sum_{F\in \Sigma^\ast_{i_0}[I]} a_F\, \check{\tc}_{F} \Bigg)=   \partial_{[S,T]}\, 0=0. \]
Let $m=|S|$. If $m=1$, then each term $\partial_{[S,T]}\, \check{\tc}_{F}$ which is not zero will be of the form 
\[\check{\tc}_{(i_0)} \otimes \check{\tc}_{G},\] 
for some $G\in \Sigma^\ast[T]$. Then, since $|T|=n-1<n$, by the induction hypothesis we have $a_{F}=0$ for all $F\in \Sigma^\ast_{i_0}[I] $ with a first lump of cardinality one. 

We now do induction on $m$, and assume that $a_{F}=0$ for all $F\in \Sigma^\ast_{i_0}[I]$ with a first lump of cardinality less than $m$. By the induction hypothesis on $m$, each $\partial_{[S,T]}\, \check{\tc}_{F}$ which is not zero, and such that $a_F$ is also not zero, will be of the form 
\[\check{\tc}_{(S)} \otimes \check{\tc}_{G},\] 
for some $G\in \Sigma^\ast[T]$. Therefore, by the induction hypothesis on $n$, we have $a_F=0$ if $F$ has a first lump of cardinality $m$.
\end{proof} 

\begin{lem}\label{basis}
The set of functionals
\[    \big \{   \check{\tc}_F: F\in \Sigma^\ast_{i_0}[I]    \big  \}\subset \check{\textbf{Z}}\textbf{ie}^\ast[I]       \]
is a basis of $\check{\textbf{Z}}\textbf{ie}^\ast[I]$. 
\end{lem}
\begin{proof}
This set of functionals is linearly independent by \autoref{main3}. This set spans $\check{\textbf{Z}}\textbf{ie}^\ast[I]$ because it is the image of the $\tc$-basis under the quotient $\textbf{Zie}^\ast\twoheadrightarrow \check{\textbf{Z}}\textbf{ie}^\ast$.
\end{proof} 
\begin{thm} \label{maincor}
The quotient
\[\textbf{Zie}^\ast \twoheadrightarrow    \check{\textbf{Z}}\textbf{ie}^\ast , \qquad   \tp_F\mapsto \check{\tp}_F \quad \text{or} \quad  \tm_F\mapsto \check{\tm}_F \quad \text{or} \quad \tc_p\mapsto \check{\tc}_p  \] 
is an isomorphism of Lie coalgebras.
\end{thm}
\begin{proof}
This quotient is an isomorphism at the level of cospecies by \autoref{basis}. This quotient preserves the cobracket by \autoref{cobracket}.
\end{proof} 
Thus, we have shown that Steinmann functionals, equivalently the span of (generalized) permutohedral cones modulo higher codimensions, equipped with the discrete differentiation of functionals across special hyperplanes, is the indecomposable quotient of the adjoint realization of $\Sig^\ast$. 

We finish with a conceptual remark, which we hope further clarifies the situation. For a partition $P\in \Pi^\ast[I]$, let $\adshuff_{P} [I]$ denote the subspace of $\adshuff[I]$ which consists of those functionals that are supported by the semisimple subspace $\sigma^\vee_{P}$,
\[\adshuff_{P} [I]:=\big\{     f\in \adshuff[I]: f(  \check{\mathtt{H}}_\cS  )\neq 0 \implies  \sigma^\vee_{\cS} \subset \sigma^\vee_{P}  \big\}   .     \]
Since the Hopf structure of $\adshuff$ was induced by the identification $\tC_p\mapsto \check{\tC}_p$ with $\Sig^\ast$, the higher multiplication of $\adshuff$ is given by
\[      \Delta_P: \adshuff(P)\hookrightarrow    \adshuff[I], \qquad  \check{\tC}_{p_1} \otimes  \dots \otimes  \check{\tC}_{p_k} \mapsto  \check{\tC}_{(p_1| \dots | p_k)}.       \]
Then, since $\check{\textbf{Z}}\textbf{ie}^\ast$ was the indecomposable quotient of $\adshuff$, the image of $\Delta_P$ must be $\adshuff_{P} [I]$,
\[         \adshuff_{P} [I] \cong  \adshuff(P).    \]
In this sense, the higher multiplication of the adjoint realization of $\Sig^\ast$ is simply the embedding of semisimple subspaces. For the Lie coalgebra $\textbf{Zie}^\ast$, one then quotients out the images of all these embeddings, leaving just the chambers. 

\subsection{Bring to Basis for Steinmann Functionals}

\begin{figure}[t]
	\centering
	\includegraphics[scale=0.7]{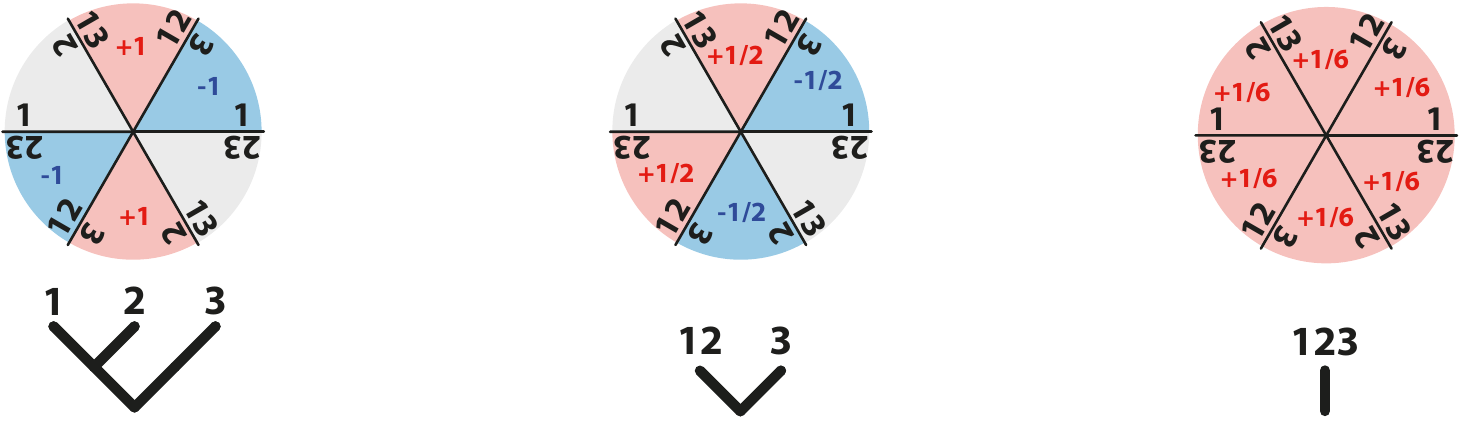}
	\caption{The primitive elements $\check{\tQ}_\mathcal{T}\in \check{\textbf{Z}}\textbf{ie}[I]$ for $\mathcal{T}\in\big  \{  [[1,2],3], [12,3], [123]\big  \}$.}
	\label{fig:lieonadjoint}
\end{figure} 

We now evaluate derivatives of functionals at the first Eulerian idempotent in order to bring Steinmann functionals to the $\check{\tp}$-basis. This will work in the same way as a Taylor series expansion. Let $\check{\Sig}$ be the dual species of the cospecies $\adshuff$, 
\[    \check{\Sig}[I]:=    \Hom\big(\adshuff[I], \Bbbk\big )   .  \]
This is naturally a quotient of the span of adjoint faces, 
\[            \check{\Sig}=      \bigslant{ \check{\Sig}^\vee }{ \sim }  \ .       \]
Let $ \check{\textbf{Z}}\textbf{ie}$ denote the Lie algebra which is dual to the Lie coalgebra $\check{\textbf{Z}}\textbf{ie}^\ast$, 
\[  \check{\textbf{Z}}\textbf{ie}[I]:=    \Hom\big (\check{\textbf{Z}}\textbf{ie}^\ast[I]  , \Bbbk \big )   .  \]
This is naturally the subspecies of $\check{\Sig}$ which is spanned by adjoint chambers. The underlying species of $ \check{\textbf{Z}}\textbf{ie}$ is given by
\[         \check{\textbf{Z}}\textbf{ie}= \bigslant{ \check{\bL}^\vee }{ \textbf{Stein} } ,     \]
where $\textbf{Stein}$ is the span of the Steinmann relations \cite[Section 4.2]{lno2019}. The \emph{first Eulerian idempotent} $\tE_I\in \Sig[I]$ is defined by putting $\tE_{\emptyset}:=0$, and 
\[ \tE_I:=\tQ_{(I)}= -\sum_{F\in \Sigma[I]}\dfrac{(-1)^{ l(F) }}{l(F)}  \tH_F      \]
for $I$ nonempty. See \cite[Section 14.1]{aguiar2013hopf}. The first Eulerian idempotent is a primitive series in $\Sig$. Its image in the adjoint realization $ \check{\textbf{Z}}\textbf{ie}$ is then as follows. For $I$ nonempty, let $\check{\mathtt{E}}_{I}\in  \check{\textbf{Z}}\textbf{ie}[I]$ such that for all $F\in \Sigma^\ast[I]$, we have
\[   
  \check{\tp}_F(\check{\mathtt{E}}_{I}) : =
\begin{cases}
1  &\quad\text{if $F=(I)$} \\
0  &\quad\text{otherwise.} 
\end{cases}
\]
To define $\check{\mathtt{E}}_{I}$, it is enough to consider just the basis elements $\{ \check{\tp}_F:F\in   \Sigma^\ast_{i_0}[I]  \}$. The definition is then satisfied because $\check{\tp}$-basis elements satisfy shuffle relations. This defines the series
\[    \bE \to  \check{\textbf{Z}}\textbf{ie} , \qquad       \tH_I \mapsto    \check{\tE}_{I}.    \]
In order to obtain the explicit isomorphism between $\textbf{Zie}$ and $ \check{\textbf{Z}}\textbf{ie}$, we should now act on $\check{\tE}_{I}$ with the dual derivative $\boldsymbol{\partial}^\ast_\mathcal{T}$ of \cite[Definition 2.4]{lno2019}, giving
\[   
\textbf{Zie} \to  \check{\textbf{Z}}\textbf{ie}
, \qquad   
 \mathcal{T} \mapsto \check{\tQ}_\mathcal{T} :=   \partial^\ast_\mathcal{T} (\check{\tE}_{S_1} \otimes \dots \otimes \check{\tE}_{S_k})
.    \]
See \autoref{fig:lieonadjoint}. We have computed $\check{\tE}_{I}$ for $n=4$, see \autoref{fig:eularian}. For $1\leq n\leq 4$, $\check{\tE}_{I}$ may be presented as a sum of $n!$ adjoint faces, all with coefficient $1/n!$. We have not studied the cases $n\geq 5$.
\begin{figure}[t]
	\centering
	\includegraphics[scale=0.6]{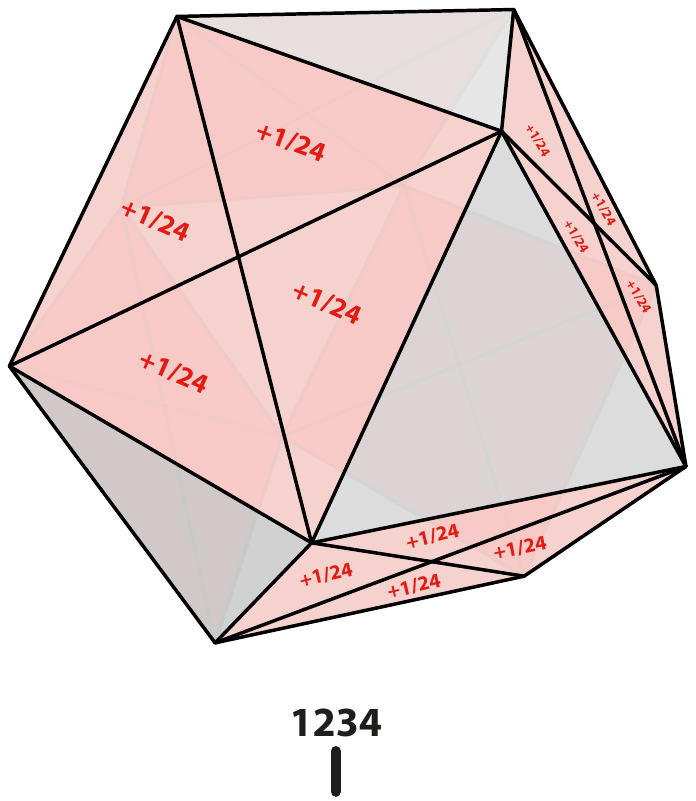}
	\caption{The adjoint realization of the first Eulerian idempotent $\tE_{I}$, for $n=4$, shown on the Steinmann sphere. The $24$ adjoint chambers which are contained in square facets are all included with coefficient $1/24$. Note that due to the Steinmann relations, this presentation is not unique.}
	\label{fig:eularian}
\end{figure} 


\begin{thm}\label{exp}
Given a Steinmann functional $f\in \check{\textbf{Z}}\textbf{ie}^\ast[I]$, we have
\[      f=\sum_{F\in \Sigma_{i_0}[I]}  \partial_{ [F] }  f  (\check{\tE}_{S_1} \otimes \dots \otimes \check{\tE}_{S_k})\, \,  \check{\tp}_F.      \]
\end{thm}
\begin{proof}
Let $a_G\in \Bbbk$ be the coefficients in the expansion of $f$ in the $\check{\tp}$-basis,
\[       f=  \sum_{G\in \Sigma^\ast_{i_0}[I] }  a_G\, \check{\tp}_G.   \]
For $F=(S_1, \dots, S_k) \in \Sigma_{i_0}[I]$, we have
\[      \partial_{[F]} f=      \sum_{G\in \Sigma^\ast_{i_0}[I] }  a_G\,    \partial_{[F]}\check{\tp}_G .\]
Then
\[ \partial_{[F]}\check{\tp}_G    =  \begin{cases}
  \check{\tp}_{G|_{S_1}} \otimes \dots \otimes \check{\tp}_{G|_{S_k}}&\quad  \text{if}\ F\leq G \\
0\in \check{\textbf{Z}}\textbf{ie}^\ast(Q_F) &\quad \text{otherwise.}
\end{cases}   \]
Therefore, since the first Eulerian idempotent is equal to $\tQ_{(I)}$ for $I$ nonempty, and the $\tQ$-basis is dual to the $\tP$-basis, we have
\[  \partial_{[F]}\check{\tp}_G ( \check{\tE}_{S_1} \otimes \dots \otimes \check{\tE}_{S_k}  )   =   \delta_{FG}.     \]
Thus,
\[        \partial_{ [F] }  f  (\check{\tE}_{S_1} \otimes \dots \otimes \check{\tE}_{S_k})=a_F. \qedhere   \]
\end{proof}

\section{Generalized Retarded Functions}\label{QFT}

\noindent We finish by describing (at least one aspect of) the connection with pAQFT. We hope to further expose the role species have to play in QFT in future work. 

Let $\check{\cU}:  \check{\textbf{Z}}\textbf{ie}  \hookrightarrow \Sig$ denote the linear dual map of the composition
\[     
\Sig^\ast \xrightarrow{\sim} \adshuff  \twoheadrightarrow  \check{\textbf{Z}}\textbf{ie}^\ast 
,\qquad      
\tM_F\mapsto \check{\tM}_F \mapsto \check{\tm}_F       
.\]
Thus, the embedding $\check{\cU}$ realizes the adjoint geometric realization of the Lie algebra $\textbf{Z}\textbf{ie}$ as the primitive part of its universal enveloping algebra $\Sig$. Given an adjoint chamber $\check{\mathtt{H}}_\cS\in \check{\textbf{L}}^\vee[I]$, we denote by 
\[
\check{\mathtt{D}}_{\cS}\in         \check{\textbf{Z}}\textbf{ie}[I]= \bigslant{ \check{\bL}^\vee[I] }{ \textbf{Stein}[I] }     
\] 
its image under the quotient by the Steinmann relations.   

\begin{figure}[t]
	\centering
	\includegraphics[scale=0.6]{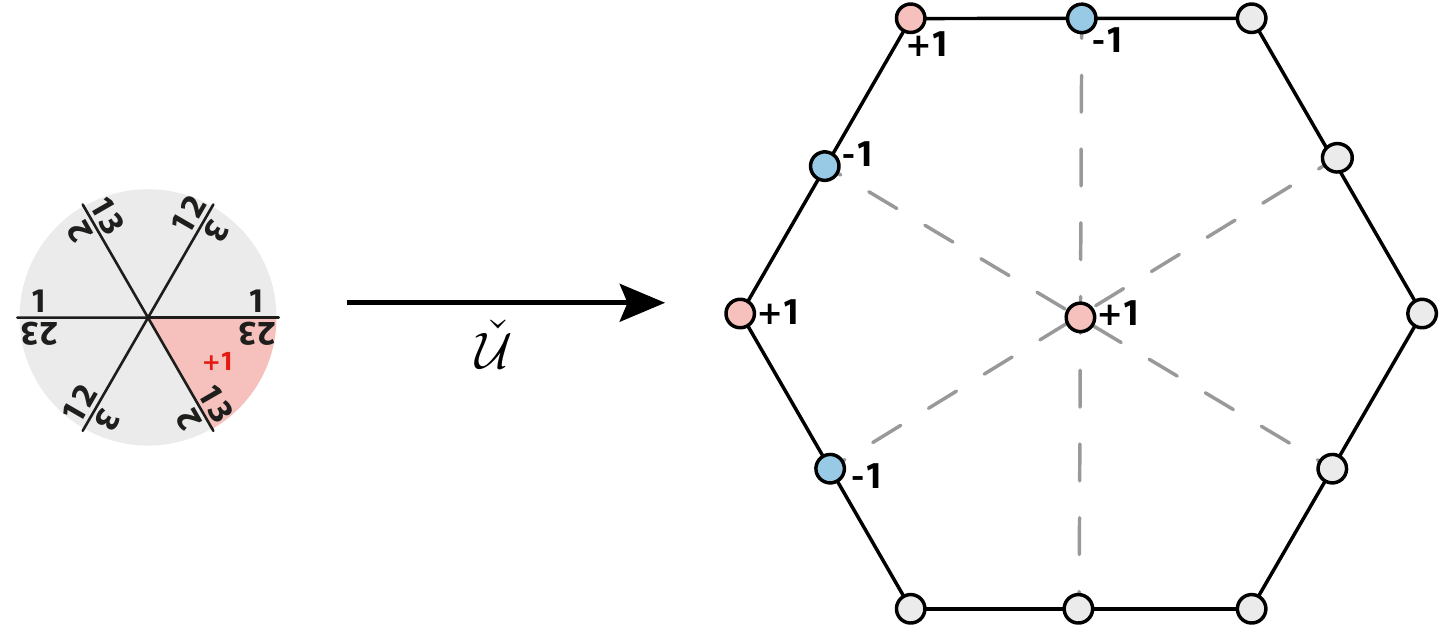}
	\caption{The primitive element of an adjoint chamber, over $I=\{1,2,3\}$. One can compute the primitive element by evaluating the $\check{\tM}$-basis on the chamber, or by using Epstein-Glaser-Stora's expansion in terms of the Tits product. The primitive element corresponds to the generalized retarded function of the chamber (`geometric cell'), expressed in terms of vacuum expectation values of operator products of \hbox{time-ordered} products.}
	\label{fig:zieofshard}
\end{figure}

\begin{prop}
The map $\check{\cU}$ is given by
\[  \check{\cU}:  \check{\textbf{Z}}\textbf{ie}  \hookrightarrow \Sig, \qquad   \check{\mathtt{D}}_{\cS} \mapsto  \mathtt{D}_\cS=\cU_{\cS}:=  \sum_{F\in \Sigma[I] }   \check{\tm}_F(\check{\mathtt{H}}_\cS) \,   \tH_F.\footnote{\ the notation $\cU_{\cS}$ is used in the physics literature, however we shall prefer the notation $\mathtt{D}_\cS$} \]
\end{prop}
\begin{proof}
This follows immediately from the fact that the $\tH$-basis and $\tM$-basis are dual. Indeed, we have that $\check{\mathtt{D}}_\cS\in \check{\textbf{Z}}\textbf{ie}[I]= \Hom(\check{\textbf{Z}}\textbf{ie}^\ast[I], \Bbbk)$ is given by
\[    \check{\tm}_G\mapsto   \check{\tm}_G( \check{\mathtt{H}}_\cS).  \]
Therefore $\check{\cU}(\check{\mathtt{D}}_\cS)$ is given by
\[   \check{\cU}(\check{\mathtt{D}}_\cS):\Sig^\ast[I]\to \Bbbk, \qquad   \tM_G\mapsto    \check{\tm}_G( \check{\mathtt{H}}_\cS)=    \sum_{F\in \Sigma[I] }   \check{\tm}_F( \check{\mathtt{H}}_\cS) \,   \tH_F\ (\tM_G)  . \qedhere  \]
\end{proof}
See \autoref{fig:zieofshard}. Recalling our definition of $\check{\tm}_F$, we see that the embedding $\check{\cU}$ appears in the mathematical foundations of perturbative QFT, where it expresses generalized retarded products in terms of operator products of time-ordered products, and hence also generalized retarded functions in terms of vacuum expectation values of products of time-ordered products, sometimes called generalized time-ordered functions. As far as we are aware, its first appearance is in \cite[Equation 79, p. 260]{ep73roleofloc}. See also \cite[Equation 1, p.26]{epstein1976general}, and more recently \cite[Equations 35, 36]{epstein2016}. 

The elements $\mathtt{D}_\cS \in \Sig[I]$, for $\cS\in \text{L}^\vee[I]$, were defined for generic real hyperplane arrangements in \cite[Equation 14.1]{aguiar2017topics} (so in particular for reflection hyperplane arrangements of other Dynkin types). Aguiar and Mahajan call these special primitive elements \emph{Dynkin elements}. From this perspective, the Steinmann relations turn up as the kernel of the map
\[ \textbf{L}^\vee \twoheadrightarrow \textbf{Zie} \hookrightarrow   \Sig, \qquad   \cS  \mapsto \mathtt{D}_\cS.\]
That is, the Steinmann relations are a set of generating relations for the type $A$ Dynkin elements. The $4$-point Steinmann relations appear in \cite[Exercise 14.67]{aguiar2017topics}. For a direct proof that the $\mathtt{D}_\cS$ are indeed primitive elements, i.e. without using the dual polyhedral algebra construction of this paper, see \cite[Proposition 14.1]{aguiar2017topics}.


The \emph{Tits product} is the action of $\Sigma$ on itself by Hopf powers, thus
\[      \Sigma\times_{ \text{Had} } \Sigma \to \Sigma, \qquad   (F,G) \mapsto F\cdot G:= \mu_F \big (  \Delta_F(  G  ) \big ) .      \]
Explicitly, if $F=(S_1, \dots , S_{k_F})$ and $G=(T_1,\dots, T_{k_G})$ are compositions of $I$, then
\[  F\cdot G=  (  T_1\cap S_1, \dots, T_{k_G}\cap S_1, \dots \dots, T_{1}\cap S_{k_F}, \dots ,   T_{k_G}\cap S_{k_F}        )_+     . \]
The Tits product is associative and unital, and so promotes $\Sigma$ to a presheaf of monoids,
\[\Sigma: \textsf{S}^{\text{op}}\to \textsf{Mon},\qquad I\mapsto \Sigma[I].\]
We then linearize the Tits product, to obtain
\[     \Sig\otimes_{ \text{Had} }   \Sig \to \Sig, \qquad   \tH_F\otimes \tH_G \mapsto  \tH_F\triangleright \tH_G :=   \tH_{F\cdot G}  .      \]
This promotes $\Sig$ to a presheaf of $\Bbbk$-algebras. The components of this species are known as Tits algebras, or algebras of proper sequences in Epstein-Glaser-Stora's formalism \cite[Section 4.1]{epstein1976general}. In QFT, the Tits product is motivated by considering causal factorization, or causal additivity, \cite[p.2]{epstein2016}. In \cite[Lemma 8]{epstein1976general}, the following expansion of $\mathtt{D}_\cS$ is given. 

\begin{prop}[Epstein-Glaser-Stora]
For $\cS\in \text{L}^\vee[I]$, we have
\[     
\mathtt{D}_{\cS}=\prod^{\text{Tits}}_{ (S,T)\in \cS}   \big (\tH_{(  I )}-   \tH_{(  T,S )}\big  )       
\]
where the right-hand side is well-defined because these elements commute in the Tits product.
\end{prop}

For more on the structure of Dynkin elements, see \cite[Chapter 14]{aguiar2017topics} and \cite[Section 14.5]{aguiar2013hopf}. Note that in the notation of \cite{aguiar2013hopf}, the Dynkin element of \autoref{fig:zieofshard} is $\mathtt{D}_3$. This corresponds to the fact that the adjoint face contains the orthogonal projection of the basis element $e_3\in \bR I$ onto $\text{T}_I^\vee$.

\bibliographystyle{alpha}
\bibliography{steinmann}

\newcommand{\etalchar}[1]{$^{#1}$}
\begin{thebibliography}{CHDD{\etalchar{+}}20}

\bibitem[AA17]{aguiar2017hopf}
Marcelo Aguiar and Federico Ardila.
\newblock Hopf monoids and generalized permutahedra.
\newblock {\em arXiv preprint arXiv:1709.07504}, 2017.

\bibitem[AB08]{brown08}
Peter Abramenko and Kenneth~S. Brown.
\newblock {\em Buildings}, volume 248 of {\em Graduate Texts in Mathematics}.
\newblock Springer, New York, 2008.
\newblock Theory and applications.

\bibitem[AM06]{MR2225808}
Marcelo Aguiar and Swapneel Mahajan.
\newblock {\em Coxeter groups and {H}opf algebras}, volume~23 of {\em Fields
  Institute Monographs}.
\newblock American Mathematical Society, Providence, RI, 2006.
\newblock With a foreword by Nantel Bergeron.

\bibitem[AM10]{aguiar2010monoidal}
Marcelo Aguiar and Swapneel Mahajan.
\newblock {\em Monoidal functors, species and {H}opf algebras}, volume~29 of
  {\em CRM Monograph Series}.
\newblock American Mathematical Society, Providence, RI, 2010.
\newblock With forewords by Kenneth Brown, Stephen Chase and Andr\'{e} Joyal.

\bibitem[AM13]{aguiar2013hopf}
Marcelo Aguiar and Swapneel Mahajan.
\newblock Hopf monoids in the category of species.
\newblock In {\em Hopf algebras and tensor categories}, volume 585 of {\em
  Contemp. Math.}, pages 17--124. Amer. Math. Soc., Providence, RI, 2013.

\bibitem[AM17]{aguiar2017topics}
Marcelo Aguiar and Swapneel Mahajan.
\newblock {\em Topics in hyperplane arrangements}, volume 226 of {\em
  Mathematical Surveys and Monographs}.
\newblock American Mathematical Society, Providence, RI, 2017.

\bibitem[AM20]{aguiar2020bimonoids}
Marcelo Aguiar and Swapneel Mahajan.
\newblock {\em Bimonoids for Hyperplane Arrangements}, volume 173.
\newblock Cambridge University Press, 2020.

\bibitem[Ara61a]{Huz1}
Huzihiro Araki.
\newblock Generalized retarded functions and analytic function in momentum
  space in quantum field theory.
\newblock {\em Journal of Mathematical Physics}, 2(2):163--177, 1961.

\bibitem[Ara61b]{Huz2}
Huzihiro Araki.
\newblock Wightman functions, retarded functions and their analytic
  continuations.
\newblock {\em Progr. Theoret. Phys. Suppl. No.}, 18:83--125, 1961.

\bibitem[Bar78]{barratt1978twisted}
M.~G. Barratt.
\newblock Twisted {L}ie algebras.
\newblock In {\em Geometric applications of homotopy theory ({P}roc. {C}onf.,
  {E}vanston, {I}ll., 1977), {II}}, volume 658 of {\em Lecture Notes in Math.},
  pages 9--15. Springer, Berlin, 1978.

\bibitem[BB11]{batyrevblume10}
Victor Batyrev and Mark Blume.
\newblock The functor of toric varieties associated with {W}eyl chambers and
  {L}osev-{M}anin moduli spaces.
\newblock {\em Tohoku Math. J. (2)}, 63(4):581--604, 2011.

\bibitem[BBT18]{billerabooleanprod}
Louis~J. Billera, Sara~C. Billey, and Vasu Tewari.
\newblock Boolean product polynomials and {S}chur-positivity.
\newblock {\em S\'{e}m. Lothar. Combin.}, 80B:Art. 91, 12, 2018.

\bibitem[BD01]{baez2001finite}
John~C. Baez and James Dolan.
\newblock From finite sets to {F}eynman diagrams.
\newblock In {\em Mathematics unlimited---2001 and beyond}, pages 29--50.
  Springer, Berlin, 2001.

\bibitem[BF00]{klaus2000micro}
Romeo Brunetti and Klaus Fredenhagen.
\newblock Microlocal analysis and interacting quantum field theories:
  renormalization on physical backgrounds.
\newblock {\em Comm. Math. Phys.}, 208(3):623--661, 2000.

\bibitem[BG02]{MR1923173}
Marek Bo\.{z}ejko and M\u{a}d\u{a}lin Gu\c{t}\u{a}.
\newblock Functors of white noise associated to characters of the infinite
  symmetric group.
\newblock {\em Comm. Math. Phys.}, 229(2):209--227, 2002.

\bibitem[Bjo15]{MR3467341}
Anders Bjorner.
\newblock Positive sum systems.
\newblock In {\em Combinatorial methods in topology and algebra}, volume~12 of
  {\em Springer INdAM Ser.}, pages 157--171. Springer, Cham, 2015.

\bibitem[BK05]{Kreimer05}
Christoph Bergbauer and Dirk Kreimer.
\newblock The {H}opf algebra of rooted trees in {E}pstein-{G}laser
  renormalization.
\newblock {\em Ann. Henri Poincar\'{e}}, 6(2):343--367, 2005.

\bibitem[BL75]{bros}
J.~Bros and M.~Lassalle.
\newblock Analyticity properties and many-particle structure in general quantum
  field theory. {II}. {O}ne-particle irreducible {$n$}-point functions.
\newblock {\em Comm. Math. Phys.}, 43(3):279--309, 1975.

\bibitem[BLL98]{bergeron1998combinatorial}
F.~Bergeron, G.~Labelle, and P.~Leroux.
\newblock {\em Combinatorial species and tree-like structures}, volume~67 of
  {\em Encyclopedia of Mathematics and its Applications}.
\newblock Cambridge University Press, Cambridge, 1998.
\newblock Translated from the 1994 French original by Margaret Readdy, with a
  foreword by Gian-Carlo Rota.

\bibitem[BMM{\etalchar{+}}12]{billera2012maximal}
L.J. Billera, J.~Tatch Moore, C.~Dufort Moraites, Y.~Wang, and K.~Williams.
\newblock Maximal unbalanced families.
\newblock {\em arXiv preprint arXiv:1209.2309}, 2012.

\bibitem[BP99]{MR1731815}
Alexander Barvinok and James~E. Pommersheim.
\newblock An algorithmic theory of lattice points in polyhedra.
\newblock In {\em New perspectives in algebraic combinatorics ({B}erkeley,
  {CA}, 1996--97)}, volume~38 of {\em Math. Sci. Res. Inst. Publ.}, pages
  91--147. Cambridge Univ. Press, Cambridge, 1999.

\bibitem[BZ09]{MR2555523}
Nantel Bergeron and Mike Zabrocki.
\newblock The {H}opf algebras of symmetric functions and quasi-symmetric
  functions in non-commutative variables are free and co-free.
\newblock {\em J. Algebra Appl.}, 8(4):581--600, 2009.

\bibitem[Car07]{MR2290769}
Pierre Cartier.
\newblock A primer of {H}opf algebras.
\newblock In {\em Frontiers in number theory, physics, and geometry. {II}},
  pages 537--615. Springer, Berlin, 2007.

\bibitem[Cav16]{cavalierires}
Renzo Cavalieri.
\newblock Hurwitz theory and the double ramification cycle.
\newblock {\em Jpn. J. Math.}, 11(2):305--331, 2016.

\bibitem[CHDD{\etalchar{+}}19]{caron2019cosmic}
Simon Caron-Huot, Lance~J Dixon, Falko Dulat, Matt Von~Hippel, Andrew~J McLeod,
  and Georgios Papathanasiou.
\newblock The cosmic {G}alois group and extended {S}teinmann relations for
  planar ${N}=4$ {SYM} amplitudes.
\newblock {\em Journal of High Energy Physics}, 2019(9):61, 2019.

\bibitem[CHDD{\etalchar{+}}20]{Caron-Huot:2020bkp}
Simon Caron-Huot, Lance~J. Dixon, James~M. Drummond, Falko Dulat, Jack Foster,
  \"Omer G\"urdo\u{g}an, Matt von Hippel, Andrew~J. McLeod, and Georgios
  Papathanasiou.
\newblock {The Steinmann Cluster Bootstrap for $N=4$ Super Yang-Mills
  Amplitudes}.
\newblock {\em PoS}, CORFU2019:003, 2020.

\bibitem[CJM11]{MR2836109}
Renzo Cavalieri, Paul Johnson, and Hannah Markwig.
\newblock Wall crossings for double {H}urwitz numbers.
\newblock {\em Adv. Math.}, 228(4):1894--1937, 2011.

\bibitem[CK99]{connes1999hopf}
A.~Connes and D.~Kreimer.
\newblock Hopf algebras, renormalization and noncommutative geometry.
\newblock In {\em Quantum field theory: perspective and prospective ({L}es
  {H}ouches, 1998)}, volume 530 of {\em NATO Sci. Ser. C Math. Phys. Sci.},
  pages 59--108. Kluwer Acad. Publ., Dordrecht, 1999.

\bibitem[CM08]{connes08}
Alain Connes and Matilde Marcolli.
\newblock {\em Noncommutative geometry, quantum fields and motives}, volume~55
  of {\em American Mathematical Society Colloquium Publications}.
\newblock American Mathematical Society, Providence, RI; Hindustan Book Agency,
  New Delhi, 2008.

\bibitem[DFG18]{drummond2018cluster}
James Drummond, Jack Foster, and {\"O}mer G{\"u}rdo{\u{g}}an.
\newblock Cluster adjacency properties of scattering amplitudes in ${N}=4$
  supersymmetric {Y}ang-{M}ills theory.
\newblock {\em Physical review letters}, 120(16):161601, 2018.

\bibitem[Dut19]{dutsch2019perturbative}
Michael Dutsch.
\newblock {\em From classical field theory to perturbative quantum field
  theory}, volume~74 of {\em Progress in Mathematical Physics}.
\newblock Birkh\"{a}user/Springer, Cham, 2019.
\newblock With a foreword by Klaus Fredenhagen.

\bibitem[Ear17]{early2017canonical}
Nick Early.
\newblock Canonical bases for permutohedral plates.
\newblock {\em arXiv preprint arXiv:1712.08520}, 2017.

\bibitem[EFK05]{ebrahimi2005hopf}
Kurusch Ebrahimi-Fard and Dirk Kreimer.
\newblock The {H}opf algebra approach to {F}eynman diagram calculations.
\newblock {\em J. Phys. A}, 38(50):R385--R407, 2005.

\bibitem[EG73]{ep73roleofloc}
H.~Epstein and V.~Glaser.
\newblock The role of locality in perturbation theory.
\newblock {\em Ann. Inst. H. Poincar\'{e} Sect. A (N.S.)}, 19:211--295 (1974),
  1973.

\bibitem[EGS75a]{egs74}
H.~Epstein, V.~Glaser, and R.~Stora.
\newblock {\GG{1}}{G}eometry of the {N} point p space function of quantum field
  theory.
\newblock In {\em Hyperfunctions and theoretical physics ({R}encontre, {N}ice,
  1973; d\'{e}di\'{e} \`a la m\'{e}moire de {A}. {M}artineau)}, pages 143--162.
  Lecture Notes in Math., Vol. 449. 1975.

\bibitem[EGS75b]{epstein1976general}
H.~Epstein, V.~Glaser, and R.~Stora.
\newblock {General properties of the n-point functions in local quantum field
  theory}.
\newblock In {\em {Institute on Structural Analysis of Multiparticle Collision
  Amplitudes in Relativistic Quantum Theory Les Houches, France, June 3-28,
  1975}}, pages 5--93, 1975.

\bibitem[Eps16]{epstein2016}
Henri Epstein.
\newblock Trees.
\newblock {\em Nuclear Phys. B}, 912:151--171, 2016.

\bibitem[FGB05]{figueroa2005combinatorial}
H\'{e}ctor Figueroa and Jos\'{e}~M. Gracia-Bond\'{i}a.
\newblock Combinatorial {H}opf algebras in quantum field theory. {I}.
\newblock {\em Rev. Math. Phys.}, 17(8):881--976, 2005.

\bibitem[GK18]{MR3753672}
John Gough and Joachim Kupsch.
\newblock {\em Quantum fields and processes}, volume 171 of {\em Cambridge
  Studies in Advanced Mathematics}.
\newblock Cambridge University Press, Cambridge, 2018.
\newblock A combinatorial approach.

\bibitem[GM02]{guta00}
M\u{a}d\u{a}lin Gu\c{t}\u{a} and Hans Maassen.
\newblock Symmetric {H}ilbert spaces arising from species of structures.
\newblock {\em Math. Z.}, 239(3):477--513, 2002.

\bibitem[GMP19]{gutekunst2019root}
Samuel~C Gutekunst, Karola M{\'e}sz{\'a}ros, and T~Kyle Petersen.
\newblock Root cones and the resonance arrangement.
\newblock {\em arXiv preprint arXiv:1903.06595}, 2019.

\bibitem[Joy81]{joyal1981theorie}
Andr\'{e} Joyal.
\newblock Une th\'{e}orie combinatoire des s\'{e}ries formelles.
\newblock {\em Adv. in Math.}, 42(1):1--82, 1981.

\bibitem[Joy86]{joyal1986foncteurs}
Andr\'{e} Joyal.
\newblock Foncteurs analytiques et esp\`eces de structures.
\newblock In {\em Combinatoire \'{e}num\'{e}rative ({M}ontreal, {Q}ue.,
  1985/{Q}uebec, {Q}ue., 1985)}, volume 1234 of {\em Lecture Notes in Math.},
  pages 126--159. Springer, Berlin, 1986.

\bibitem[KTT11]{kamiya2010ranking}
Hidehiko Kamiya, Akimichi Takemura, and Hiroaki Terao.
\newblock Ranking patterns of unfolding models of codimension one.
\newblock {\em Adv. in Appl. Math.}, 47(2):379--400, 2011.

\bibitem[KTT12]{kamiya2012arrangements}
Hidehiko Kamiya, Akimichi Takemura, and Hiroaki Terao.
\newblock Arrangements stable under the {C}oxeter groups.
\newblock In {\em Configuration spaces}, volume~14 of {\em CRM Series}, pages
  327--354. Ed. Norm., Pisa, 2012.

\bibitem[KW17]{MR3636409}
Ralph~M. Kaufmann and Benjamin~C. Ward.
\newblock Feynman categories.
\newblock {\em Ast\'{e}risque}, (387):vii+161, 2017.

\bibitem[LM00]{losevmanin}
A.~Losev and Y.~Manin.
\newblock New moduli spaces of pointed curves and pencils of flat connections.
\newblock volume~48, pages 443--472. 2000.
\newblock Dedicated to William Fulton on the occasion of his 60th birthday.

\bibitem[LMPS19]{MR3917218}
Joel~Brewster Lewis, Jon McCammond, T.~Kyle Petersen, and Petra Schwer.
\newblock Computing reflection length in an affine {C}oxeter group.
\newblock {\em Trans. Amer. Math. Soc.}, 371(6):4097--4127, 2019.

\bibitem[LNO19]{lno2019}
Zhengwei Liu, William Norledge, and Adrian Ocneanu.
\newblock The adjoint braid arrangement as a combinatorial {L}ie algebra via
  the {S}teinmann relations.
\newblock {\em arXiv preprint arXiv:1901.03243}, 2019.

\bibitem[Mey11]{meyer2011intersection}
Henning Meyer.
\newblock Intersection theory on tropical toric varieties and compactifications
  of tropical parameter spaces.
\newblock {\em Ph.D. thesis}, TU Kaiserslautern, 2011.

\bibitem[MNT13]{menous2013mould}
Fr{\'e}d{\'e}ric Menous, Jean-Christophe Novelli, and Jean-Yves Thibon.
\newblock Mould calculus, polyhedral cones, and characters of combinatorial
  {H}opf algebras.
\newblock {\em Advances in Applied Mathematics}, 51(2):177--227, 2013.

\bibitem[Mor06]{morton2006categorified}
Jeffrey Morton.
\newblock Categorified algebra and quantum mechanics.
\newblock {\em Theory Appl. Categ.}, 16:No. 29, 785--854, 2006.

\bibitem[MR09]{mikhalkin2009tropical}
Grigory Mikhalkin and Johannes Rau.
\newblock {\em Tropical geometry}.
\newblock MPI for Mathematics, 2009.

\bibitem[NT06]{MR2209212}
Jean-Christophe Novelli and Jean-Yves Thibon.
\newblock Construction de trig\`ebres dendriformes.
\newblock {\em C. R. Math. Acad. Sci. Paris}, 342(6):365--369, 2006.

\bibitem[Pet13]{MR3134040}
Dan Petersen.
\newblock The operad structure of admissible {$G$}-covers.
\newblock {\em Algebra Number Theory}, 7(8):1953--1975, 2013.

\bibitem[PR06]{MR2228332}
Patricia Palacios and Mar\'{\i}a~O. Ronco.
\newblock Weak {B}ruhat order on the set of faces of the permutohedron and the
  associahedron.
\newblock {\em J. Algebra}, 299(2):648--678, 2006.

\bibitem[Pro90]{procesipermvar}
C.~Procesi.
\newblock The toric variety associated to {W}eyl chambers.
\newblock In {\em Mots}, Lang. Raison. Calc., pages 153--161. Herm\`es, Paris,
  1990.

\bibitem[PRW08]{vic}
Alex Postnikov, Victor Reiner, and Lauren Williams.
\newblock Faces of generalized permutohedra.
\newblock {\em Doc. Math.}, 13:207--273, 2008.

\bibitem[Rej16]{rejzner2016pQFT}
Kasia Rejzner.
\newblock {\em Perturbative algebraic quantum field theory}.
\newblock Mathematical Physics Studies. Springer, Cham, 2016.
\newblock An introduction for mathematicians.

\bibitem[Rue61]{Ruelle}
D.~Ruelle.
\newblock Connection between {W}ightman functions and {G}reen functions in
  {$p$}-space.
\newblock {\em Nuovo Cimento (10)}, 19:356--376, 1961.

\bibitem[Sch93]{Bill93}
William~R. Schmitt.
\newblock Hopf algebras of combinatorial structures.
\newblock {\em Canad. J. Math.}, 45(2):412--428, 1993.

\bibitem[Sch18]{schultka2018toric}
Konrad Schultka.
\newblock Toric geometry and regularization of {F}eynman integrals.
\newblock {\em arXiv preprint arXiv:1806.01086}, 2018.

\bibitem[Sch20]{perturbative_quantum_field_theory}
Urs Schreiber.
\newblock {G}eometry of physics - perturbative quantum field theory.
\newblock {\em nLab}, ncatlab.org, 2020.
\newblock Revision 197.

\bibitem[Ste60a]{steinmann1960zusammenhang}
O.~Steinmann.
\newblock \"{U}ber den {Z}usammenhang zwischen den {W}ightmanfunktionen und den
  retardierten {K}ommutatoren.
\newblock {\em Helv. Phys. Acta}, 33:257--298, 1960.

\bibitem[Ste60b]{steinmann1960}
O.~Steinmann.
\newblock Wightman-{F}unktionen und retardierte {K}ommutatoren. {II}.
\newblock {\em Helv. Phys. Acta}, 33:347--362, 1960.

\bibitem[Ste71]{steinbook71}
Othmar Steinmann.
\newblock {\em Perturbation expansions in axiomatic field theory}.
\newblock Springer-Verlag, Berlin-New York, 1971.
\newblock Lecture Notes in Physics, Vol. 11.

\bibitem[Sto93]{stover1993equivalence}
Christopher~R. Stover.
\newblock The equivalence of certain categories of twisted {L}ie and {H}opf
  algebras over a commutative ring.
\newblock {\em J. Pure Appl. Algebra}, 86(3):289--326, 1993.

\bibitem[Str75]{streater1975outline}
Ray~F. Streater.
\newblock Outline of axiomatic relativistic quantum field theory.
\newblock {\em Reports on Progress in Physics}, 38(7):771, 1975.

\bibitem[Thi01]{thibon01nsymqsym}
Jean-Yves Thibon.
\newblock Lectures on noncommutative symmetric functions.
\newblock In {\em Interaction of combinatorics and representation theory},
  volume~11 of {\em MSJ Mem.}, pages 39--94. Math. Soc. Japan, Tokyo, 2001.

\bibitem[War19]{ward2019massey}
Benjamin~C Ward.
\newblock Massey products for graph homology.
\newblock {\em arXiv preprint arXiv:1903.12055}, 2019.

\end{thebibliography}

\end{document}